\documentclass[times,final]{elsarticle}
\usepackage[margin={2.5cm, 2.5cm},nohead]{geometry}
\usepackage{framed,multirow}

\usepackage{latexsym}
\usepackage{amsmath, amssymb, amsfonts, amsthm}
\usepackage{graphicx}
\usepackage{xcolor}
\usepackage{hyperref}
\usepackage[hypcap]{caption}

\newcommand{\revi}[1]{#1}
\newcommand{\revA}[1]{#1}
\newcommand{\revB}[1]{#1}
\newcommand{\revC}[1]{#1}
\newcommand{\revD}[1]{#1}
\newcommand{\revAA}[1]{#1}
\newcommand{\revDD}[1]{#1}
\newcommand{\revZ}[1]{#1}

\newcommand{\V}[1]{\mbox{\boldmath $ #1 $}}

\newtheorem{thm}{Theorem}[section]

\newtheorem{remark}[thm]{Remark}

\newtheorem{definition}[thm]{Definition}
\newtheorem{proposition}[thm]{Proposition}
\numberwithin{equation}{section}

\newcommand*{\vcenteredhbox}[1]{\begingroup
\setbox0=\hbox{#1}\parbox{\wd0}{\box0}\endgroup}

\newcommand{\bbR}{\mathbb{R}}

\newcommand{\vu}{\mathbf{u}}
\newcommand{\vv}{\mathbf{v}}
\newcommand{\vx}{\mathbf{x}}

\newcommand{\vlambda}{\boldsymbol{\lambda}}
\newcommand{\vxi}{\boldsymbol{\xi}}
\newcommand{\veta}{\boldsymbol{\eta}}

\newcommand{\cF}{\mathcal{F}}

\newcommand{\T}{\mathcal{T}}
\newcommand{\M}{\mathbb{M}}
\newcommand{\J}{\mathbb{J}}
\newcommand{\trace}{\textrm{trace}}

\begin{document}

\begin{frontmatter}
\title{Anisotropic mesh quality measures and adaptation for polygonal meshes}
\author[1]{Weizhang Huang}
\address[1]{Department of Mathematics, The University of Kansas, Lawrence, KS 66045, US}
\ead{whuang@ku.edu}

\author[2]{Yanqiu Wang\corref{cor1}}
\address[2]{School of Mathematical Sciences, Nanjing Normal University, Nanjing, China} %
\cortext[cor1]{Corresponding author.}
\ead{yqwang@njnu.edu.cn} %


\begin{abstract}
\revZ{Anisotropic mesh quality measures and adaptation} are studied for \revC{convex polygonal meshes}.
 Three sets of alignment and equidistribution measures are developed,
\revZ{based on least squares fitting, generalized barycentric mappings,
and the singular value decomposition, respectively.}
Numerical tests show that all three sets of mesh quality measures provide good measurements
for the quality of polygonal meshes under given anisotropic metrics.
Based on \revi{the second set of quality measures} and using a moving mesh partial differential equation,
an anisotropic adaptive polygonal mesh method is constructed for the numerical solution of
second-order elliptic equations.
Numerical examples are presented to demonstrate the effectiveness of the method.
\end{abstract}


\end{frontmatter}

\section{Introduction}

During the last decade, polygonal/polyhedral meshes have gained considerable
attention in the scientific computing community partially due to their flexibility
to deal with complicated geometry, curved boundaries, local mesh refining/coarsening,
and also to their connection with the Voronoi diagram.
Various numerical methods have been developed for solving partial differential equations (PDEs) on
polygonal/polyhedral meshes, including the mimetic finite difference method (see the recent
survey paper \cite{Lipnikov14}), \revA{the finite element method
\cite{Manzini14, Sukumar04, Sukumar06, Talischi09, Talischi10, Wicke07}},
the finite volume method \cite{Herbin96, SunWuZhang13},
the virtual element method (see \cite{Veiga13} and references therein),
the discontinuous Galerkin method \cite{Cockburn09, Gassnera09, MuWangWang15},
and the weak Galerkin method \cite{mwy-wg-stabilization, wy-mixed}.
The error analysis of these methods usually require the mesh to satisfy \revZ{certain shape regularity conditions
\cite{Brezzi05, Gillette12, Herbin96, MuWangWang15, wy-mixed}.
These conditions focus on the isotropic case
and} state that each polygon or polyhedron cannot be too ``thin''.
In practice and for the anisotropic case,
one typically wants the mesh elements to be as regular in shape and uniform in size
as possible under certain metrics
since the approximation error \revZ{and condition number} of the discrete system depend closely on the mesh geometry.
For simplicial meshes, this dependence has been well studied under the topic of ``mesh quality measures''
and many different computable quality measures have been designed;
e.g., see \revZ{\cite{ChenSunXu2006, Hua05c, HR11, HS03, LJ94, Shewchuk2002}}.
More importantly, mesh quality measures often play an essential role in constructing
efficient mesh adaptation algorithms since they also provide a definition of optimal
meshes.

The objective of this work is to develop anisotropic mesh quality measures for \revC{convex polygonal meshes},
from which one can design \revZ{algorithms for anisotropic polygonal mesh adaptation
that allow mesh elements to have a large aspect ratio}. The key is to keep elements aligned to some extent
with the geometry of the solution.
On simplicial meshes,
it has been amply demonstrated that a properly generated anisotropic mesh can
lead to a much more accurate solution than
an isotropic mesh with comparable size; see for example \revZ{\cite{Ape99,AD92,CHMP97,DS89,FP01,FP03,Hua05b,Hua05,Pic03}}.
However, there do not seem to exist any systematic studies of anisotropic 
quality measures and adaptation for polygonal/polyhedral meshes.

%

A major difference between triangular and polygonal meshes is that, all triangles
are affine similar to a single reference triangle and their
quality in shape and size can be measured by comparing with the reference triangle.
Unfortunately, this does not work for polygonal meshes since polygons even with the same number (greater than $3$) of vertices
cannot be mapped into a single reference polygon under affine mappings.
This poses great difficulty in studying polygonal mesh quality measures,
as one can no longer use many techniques developed for triangular meshes.
In this work, in parallel to the so-called alignment (for regularity)
and equidistribution (for size) measures for triangular meshes \cite{Hua01b, Hua05c},
we shall develop three sets of polygonal mesh alignment and equidistribution measures
\revZ{based on least squares fitting, generalized barycentric mappings,
and the singular value decomposition, respectively.}
Numerical tests show that all three sets of measures give good
\revZ{indications} of the actual anisotropic mesh quality under
given metrics, with individual emphases on slightly different aspects.
Based on \revi{the second set of quality measures} and
using the so-called Moving Mesh PDE (MMPDE) moving mesh method \cite{HR11},
we construct an anisotropic adaptive polygonal mesh method for the numerical solution of
second-order elliptic equations,
with the aim of minimizing the $L^2$ norm of approximation errors.
Numerical examples will be presented to demonstrate the effectiveness of the method.

We point out that, although this paper only concerns two-dimensional cases,
our anisotropic quality measures {\revB{and the mesh adaptation method built upon the measures}
can be readily extended into three-dimensions.
A brief description of the 3D extension will be given at the end of Sections \ref{sec:qualitymeasures} and \ref{sec:MMPDE}. 
Moreover, the MMPDE moving mesh method is only an example
of mesh adaptation methods that can be used for generating anisotropic polygonal meshes.
\revA{Other techniques include those using anisotropic Voronoi diagrams \cite{Labelle03}, 
refinements \cite{CHMP97, Weisser2019}, and curvature directions \cite{Alliez03}.
}

The paper is organized as follows. Section \ref{sec:qualitymeasures} is devoted to the development of
three sets of alignment and equidistribution quality measures for polygonal meshes.
An anisotropic adaptive polygonal mesh method is constructed in Section \ref{sec:MMPDE}
based on \revi{the second set of the quality measures} and the MMPDE moving mesh method.
Numerical results obtained with the proposed method
are presented in Section \ref{sec:numericalResults}. Conclusions are drawn
in the last section.

Throughout the paper, the short name ``$n$-gon" stands for a polygon with $n$ vertices.
We consider meshes consisting of convex polygons in this work.  We also assume
that all polygons in a mesh are non-degenerate, i.e., any three vertices of a polygon
do not lie on the same line.

\section{Anisotropic mesh quality measures}
\label{sec:qualitymeasures}

In this section we study three sets of mesh quality measures for a general polygonal mesh $\T$
on a two-dimensional bounded, polygonal domain $\Omega$ under the metric specified by a given tensor $\M$.
The metric tensor $\M = \M(\vx)$ is assumed to be symmetric and uniformly
positive definite on $\Omega$. It depends on the physical solution
in the case of mesh adaptation for the numerical solution of
PDEs (See Section \ref{sec:MMPDE}).
The simplest case is $\M = I$ (the Euclidean metric) where
the mesh quality measures evaluate the shape regularity and size uniformity of the mesh.

\revi{
A unitary regular $n$-gon is usually considered to be of high quality in the Euclidean metric and
can thus be used as a reference $n$-gon.
\revB{Unlike triangles, however, $n$-gons with $n > 3$} generally are not affine similar
to a single reference $n$-gon. Hence the idea of measuring the quality of a polygonal mesh
by comparing its elements to regular $n$-gons through affine mappings does not work in general.
To make it work, one needs to build proper mappings connecting arbitrary polygons with the reference ones,
or re-define reference polygons of ``good quality", or do both.

We first clarify what polygons, in addition to regular $n$-gons, are considered to be of ``good quality''
in terms of shape regularity.
Size uniformity is not considered for now because we can always \revZ{rescale polygons to} a unitary size.
There are several pioneering works in this direction \cite{attene19, Cangiani17, Cangiani19, Gillette12, Lipnikov13, wy-mixed},
most of which require, either explicitly or implicitly, that a shape-regular polygon lies between an outer circle and an inner circle 
such that the ratio between their radii is bounded above.
This guarantees the Bramble-Hilbert lemma and consequently the standard convergence rates of
polynomial approximation on the polygon \cite{Brenner07}.
To facilitate further discussions, we define

   \smallskip
   \begin{definition}
   For a given polygon $T$, its in-radius is the maximum radius of all disks contained inside $T$, and its outer-radius is the
   minimum radius of all disks containing $T$.
   \end{definition}
   \smallskip

Some of the works on polygonal shape regularity also require that the polygon has no ``short edge''.
Recently, researchers noticed that the no ``short edge'' requirement can be dropped in the case of virtual element methods \cite{Brenner18}.
But for other numerical methods, the situation is not completely clear and the no ``short edge'' is still required at least 
in some theoretical analysis \cite{Gillette12}.
Partly because of this and partly because handling ``short edges'' in the mesh quality measure turns out
to be a highly non-trivial issue,
in this paper we simply define a ``good quality'' polygon to be a polygon satisfying the outer-radius/in-radius ratio requirement, i.e.,  
the ratio between the out-radius and in-radius is bounded above by a constant. 
Numerical results and more discussions on short edges will be presented \revZ{in Section \ref{sec:compareMeasures}}.
}

\subsection{Continuous mesh quality measures} \label{subsec:cF}

In this subsection, we shall take a close look
at the development of quality measures for a triangular mesh and then establish continuous
mesh quality measures which will serve as a unified framework for developing quality measures
for a polygonal mesh.

\revi{
  Assume that a ``computational mesh" $\T_C$ has been chosen so that
  for any $T\in\T$, there exists a corresponding
  reference polygon $T_C\in\T_C$ with an associated bijection $\cF_T:\: T_C\to T$.
  Generally speaking, $\T_C$ should be chosen such that its polygons have good quality
  and comparable sizes under the Euclidean metric.
  We emphasize that $\T_C$ can be taken as either a conventional mesh, \revDD{i.e., a partition of a polygonal region,}
  or just a collection of polygons which does not need to form a mesh.
  In the first case, $\T_C$ must have the same topological structure as $\T$
  so that there is a one-to-one correspondence between elements in $\T_C$ and $\T$. 
  This type of computational mesh is used in the mesh adaptation algorithm to be described in Section \ref{sec:MMPDE}
  where how to set up such a $\T_C$ is also discussed. 
  In the second case, the cardinality of the set $\T_C$ should be equal to the number of polygons in $\T$.
Particularly, if several elements in $\T$ have the same reference polygon, the same number of copies of the reference polygon should be included in $\T_C$.
  An example of such a $\T_C$ is the collection of unitary regular $n$-gons
 such that corresponding to each polygon in $\T$ there is a unitary regular polygon of the same number of sides
 in $\T_C$ serving as its reference polygon.
  One may replace the unitary regular $n$-gons in $\T_C$ by good quality $n$-gons when needed.
  Either way, in this definition $\T_C$ is not a conventional mesh by itself.

  Given $\T_C$ and the bijections $\cF_T$ for all $T\in \T$,
  the global bijection $\cF$ is defined to be the collection of all local $\cF_T$'s,
  and we want it to be piecewise differentiable so that its Jacobian matrix exists almost everywhere.
}

Let us recall \revA{how the mesh quality measures for a triangular mesh are defined} in \cite{HR11}.
Assume that
both $\T$ and $\T_C$ consist of triangles. Then, one can simply set
$\cF_T: T_C \to T$ to be an affine mapping.
Denote the vertices of $T$ and $T_C$ by $\vx_i$ and $\vxi_i$, $i=1, 2, 3$, respectively, \revC{and both ordered counter-clockwise}.
For a given constant-valued metric $\M_T$ on $T\in \T$, it is not difficult to see that
the triangle $T$ (measured in the metric $\M_T$)
is similar in shape to the triangle $T_C$ (measured in the Euclidean metric) if
\[
(\vx_i - \vx_j)^t \M_T (\vx_i - \vx_j) = \sigma_T (\vxi_i - \vxi_j)^t (\vxi_i - \vxi_j),\quad i, j = 1, 2, 3,\; i \neq j,
\]
where $\sigma_T$ is a constant \revi{indicating the scale/size of $T$. 
When $\M_T=I$, $\sigma_T$ is simply the ratio between the areas of $T$ and $T_C$.
In general, one may consider $\sigma_T$ as the ratio between the ``area" of $T$  (measured  in the metric $\M_T$) and the area of $T_C$  (measured in the Euclidean metric).
}

Using the affine mapping $\cF_T: T_C \to T$ and its Jacobian matrix
$\J_T$ (which is constant on triangle $T$), we can rewrite the above condition \revB{as}
\[
(\vxi_i - \vxi_j)^t \J_T^t \M_T \J_T (\vxi_i - \vxi_j) = \sigma_T (\vxi_i - \vxi_j)^t (\vxi_i - \vxi_j),\quad i, j = 1, 2, 3,\; i \neq j.
\]
It can be shown (cf. \cite[Lemma~4.1.1]{HR11}) that the above equation implies
\begin{equation}
\J_T^t \M_T \J_T = \sigma_T I .
\label{JMJ-1}
\end{equation}

\revi{
Equation (\ref{JMJ-1}) gives a condition \revB{for $T$ to have a good shape} with reference to $T_C$.
Next, we want to make sure that all $T$s in $\T$ have comparable size/``area", measured under their individual metric $\M_T$.
Recall that triangles in $\T_C$ are quasi-uniform in size by assumption
and that $\sigma_T$ is the ratio between the ``area" of $T$ and the area of $T_C$.
Then, if there exists a constant $\sigma_h$ such that $\sigma_T=\sigma_h$ for all $T\in\T$, 
we can claim that $\T$ is ``uniform with respect to $\T_C$", measured in the piecewise-constant metric $\M_T$ for $T\in \T$.
In the ideal situation,  $\sigma_h$ can be calculated using compatibility.
In fact, taking} determinant of both sides of (\ref{JMJ-1}) and using the fact that
$\det(\J_T) = |T|/|T_C|$, where $|T|$ and $|T_C|$ are respectively the areas of $T$ and $T_C$
under the Euclidean metric, one gets
\[
|T| \revZ{\sqrt{\det (\M_T)}} = \sigma_h |T_C| .
\]
\revB{Summing it up} over all triangles gives
\begin{equation}
\sigma_h = \frac{\sum_{T \in \T} |T| \sqrt{\det (\M_T)}}{\sum_{T_C \in \T_C} |T_C|}.
\label{sigma-00}
\end{equation}
Moreover, (\ref{JMJ-1}) should now be written as
\begin{equation}
\J_T^t \M_T \J_T = \sigma_h I ,\quad \forall T \in \T.
\label{JMJ-2}
\end{equation}

Equation (\ref{JMJ-2}) describes an ideal situation. In practice,
\revD{it may be difficult, or impossible in many cases, to make} all $T\in \T$ be exactly similar to their counterparts and $\sigma_T$
be constant for all elements.
Nevertheless, from (\ref{JMJ-2}) the so-called alignment and equidistribution quality measures
can be developed for a triangular mesh \revZ{\cite{Hua05c}}.
These quality measures can tell us how closely a given mesh is to being an ideal mesh.

Next, we shall find a continuous analogue of (\ref{JMJ-2}),
based on which the alignment and equidistribution quality measures for polygonal meshes can be developed.
To this end, we define the computational domain as
\[
\Omega_c = \bigcup\limits_{T_C \in \T_C} T_C .
\]
Note that $\Omega_c$ can be a conventional two-dimensional domain or a collection of polygons.
It is not difficult to see that a continuous analogue of (\ref{JMJ-2}) is
\begin{equation}
\label{sim-1}
\revD{\J^t\, \M\, \J = \sigma I,\quad \forall \vx \in \Omega,}
\end{equation}
where $\sigma$ is a constant,  $\J = \frac{\partial \vx}{\partial \vxi}$
is the Jacobian matrix of the mapping $\cF: \Omega_c \to \Omega$, and $\vxi$ and $\vx$ are
the coordinates of ${\Omega_c}$ and $\Omega$, respectively.
\revD{Note that although $\J$ is usually expressed as a function of $\vxi\in\Omega_c$, we can always get its value on $\Omega$
using the one-to-one mapping $\cF$. This applies to all functions either defined on $\Omega_c$ or vice versa on $\Omega$.
\revZ{Hereafter, we use} this change of variable implicitly.}
Readers with differential geometry backgrounds may immediately realize that
Equation (\ref{sim-1}) can also be explained using the language of coordinate transformation of differential forms.

At the continuous level, the mesh is represented by the mapping $\cF$
from $\Omega_c$ to $\Omega$.
Moreover, the constant $\sigma$ in (\ref{sim-1}) is determined by compatibility.
\revZ{Indeed,} by taking determinant on both sides of (\ref{sim-1}) and integrating over ${\Omega_c}$ one gets
\begin{equation}
\label{sigma-0}
\sigma = \frac{1}{|{\Omega_c}|} \int_{\Omega} \sqrt{ \det(\M) } d \vx ,
\end{equation}
where $|{\Omega_c}|$ denotes the area of ${\Omega_c}$.

The condition (\ref{sim-1}) implies that all eigenvalues of matrix $\J^t \M \J $ are equal to $\sigma$.
Following \cite{Hua05c}, this condition can be split into two:
\begin{itemize}
\item {\em The alignment condition} requires all eigenvalues of $\J^t \M \J $ be equal,
  i.e., their arithmetic mean is equal to the geometric mean:
 \begin{equation}
\label{eq:contAlignment}
\frac{1}{2} \trace (\J^t \M \J) = \det (\J^t \M \J)^{1/2} ;
 \end{equation}
\item {\em The equidistribution condition} then requires the geometric mean of the eigenvalues be equal to $\sigma$:
\begin{equation}
\label{eq:contEquidistribution}
\det(\J) \sqrt{\det(\M)} = \sigma.
\end{equation}  
\end{itemize}

%
%

Using (\ref{eq:contAlignment}) and (\ref{eq:contEquidistribution}),
we can define \revB{the measure functions} as
\begin{equation}
\label{eq:contMeasurements}
q_{ali}(\vx)  = \frac{\trace (\J^t \M \J) }{2 \det (\J^t \M \J)^{1/2}},\qquad
q_{eq}(\vx) = \frac{\det(\J) \sqrt{\det(\M)} }{ \sigma} .
\end{equation}
In the ideal case when perfect alignment and equidistribution are reached everywhere,
one has $q_{ali} \equiv 1$ and $q_{eq} \equiv 1$ over the entire $\Omega$.
In practice, one usually does not have perfect alignment or equidistribution.
In this case, functions $q_{ali}(\vx)$ and $q_{eq}(\vx)$ indicate how closely
the alignment and equidistribution conditions are satisfied at point $\vx$.
The overall alignment and equidistribution measures can then be defined as
\begin{equation}
\label{eq:contGlobalMeasurements}
Q_{ali} = \max_{\vx\in\Omega} q_{ali} (\vx),\qquad Q_{eq} = \max_{\vx\in\Omega} q_{eq} (\vx).
\end{equation}
From the inequality of the arithmetic and geometric means and the fact that
\[
\int_{{\Omega_c}} q_{eq} d \vxi = |\Omega_c|,
\]
it is not difficult to show that both $Q_{ali}$ and $Q_{eq}$ have the range $[1,\infty)$. Moreover, the smaller they are,
the better the alignment and equidistribution of the mapping $\cF$.

\revB{The continuous mesh quality measures defined above} will serve as a unified framework for defining
actual mesh quality measures for a polygonal mesh.
\revC{Clearly, $Q_{ali}$ and $Q_{eq}$
depend on the metric tensor $\M$, the computational mesh $\T_C$ (which determines $\Omega_c$) and the mapping $\cF$.}
Recall that $\T_C$ can be a conventional mesh or a collection of polygons and the corresponding $\Omega_c$
is a conventional domain or a collection of polygons.
From (\ref{eq:contMeasurements}), one can see that the determining factor for $Q_{ali}$ and $Q_{eq}$
is the Jacobian matrix $\J$ of the bijection $\cF$.
We also point out that the mapping $\cF$ is defined through the local
mapping $\cF_T: T_C \to T$ for all $T\in \T$. Unlike triangular meshes, there does not exist
an affine mapping between $T$ and $T_C$ in general for polygonal meshes.
Thus, one has to specify $\cF$ for a given polygonal mesh $\T$.
Commonly, $\cF_T$ is chosen to be a non-affine mapping for a general $T_C$ or
an affine mapping for a specially chosen $T_C$; see the next three subsections.

We now study three discretizations/discrete approximations to the continuous mesh quality measures.
The key components are the choices of $\T_C$ and the corresponding mapping
$\cF$ (to represent the mesh $\T$).

\subsection{Approximation 1: use least squares fitting}
\label{sec:app1}

In this approximation, we assume that the computational mesh $\T_C$ has been chosen
to be a conventional mesh or a collection of reference polygons.
Typically $\T_C$ should consist of good quality polygons in the Euclidean metric.

As part of the definition, we choose to approximate the metric tensor $\M$ on $T$
by its average \begin{equation}
\M_T = \frac{1}{|T|} \int_T \M d \vx.
\label{eq:def1M}
\end{equation}
This approximation is reasonably accurate while making $\M$ be constant
on each element. To have a constant metric on each element is important
since it can significantly simplify the analysis and derivation.

As mentioned before, $\cF$ is defined through the local mapping $\cF_T: T_C \to T$ \revD{that is not necessarily
affine in general.
The existence of such a local mapping is guaranteed by
the Riemann mapping theorem, which states that there exists a conformal mapping between
  any two open $n$-gons and the mapping can be extended continuously to the boundary}.
As we will see in the next subsection, there actually exist many such mappings.
They all satisfy
\begin{equation}
\vx_i = \cF_T(\vxi_i), \; i = 1, ..., n,
\label{interp-cond-1}
\end{equation}
where $\vxi_i$ and $\vx_i$, $i=1,\cdots n$, are the vertices of $T_C$ and $T$, respectively, \revC{and both ordered counter-clockwise}.
Instead of constructing a specific example of these mappings, we consider to
fit an affine mapping \revZ{$\vx = A_T \vxi + \mathbf{c}$} in the least squares sense based on the above conditions.
Note that such a mapping is an approximation to all bijections between $T_C$ and $T$
satisfying (\ref{interp-cond-1}). Although the approximation is very rough, it serves our purpose
to define mesh quality measures. The condition for the least squares fitting is
\begin{equation} \label{eq:ls1}
A_T \vxi_i + \mathbf{c} = \vx_i,\quad i = 1, ..., n.
\end{equation}

\revi{
We now discuss how to compute $A_T$ using the least squares method.
To this end, one needs to first eliminate the effect of the translation $\mathbf{c}$.
Note that for any given set of coefficients $(\alpha_1,\alpha_2,\ldots, \alpha_n)$ satisfying $0\le \alpha_i\le 1$ and $\sum_{i=1}^n\alpha_i = 1$,
the  convex combinations $\vx_{a} \triangleq \sum_{i=1}^n \alpha_i\vx_i\in T$ and $\vxi_a \triangleq \sum_{i=1}^n \alpha_i\vxi_i\in T_C$ satisfy
\begin{equation} \label{eq:ls2}
A_T \vxi_a + \mathbf{c} = \vx_a ,
\end{equation}
where the subscript ``$a$" stands for ``anchor point". Subtracting \eqref{eq:ls2} from \eqref{eq:ls1} gives
\begin{equation} \label{eq:ls4}
A_T (\vxi_i - \vxi_a) = \vx_i - \vx_a,\quad i = 1, ..., n.
\end{equation}
A least squares solution for the above equation is
\[
A_T = E_T E_{T_C}^+,
\]
where $E_T$ and $E_{T_C}$ are matrices defined by
\begin{align*}
& E_T = [\vx_1 - \vx_a,\; \vx_2 - \vx_a,\; \cdots,\; \vx_n - \vx_a] \in \bbR^{2\times n},
\\
& E_{T_C} = [\vxi_1-\vxi_a,\; \vxi_2 - \vxi_a,\; \cdots,\;\vxi_n-\vxi_a]  \in \bbR^{2\times n},
\end{align*}
and $E_{T_C}^+$ is the pseudo-inverse of $E_{T_C}$.
}
Since both $T$ and $T_C$ are convex and non-degenerate,
both $E_T$ and $E_{T_C}$ have full row rank.
Then, the pseudo-inverse of $E_{T_C}$ can be found as
\[
E_{T_C}^+ = E_{T_C}^t (E_{T_C} E_{T_C}^t)^{-1} \in \bbR^{n\times 2},
\]
which gives rise to
\begin{equation}
\label{J-1}
A_T = E_T E_{T_C}^t (E_{T_C} E_{T_C}^t)^{-1} .
\end{equation}
By construction, $A_T$ is a linear approximation to the Jacobian matrix of any piecewise continuously differentiable
bijection between $\T_C$ and $\T$ satisfying (\ref{interp-cond-1}).

\revi{
  \begin{remark} \label{rem:Q1anchor}
  When $n=3$, it is not difficult to see that $A_T$ is equal to the Jacobian matrix of the unique affine mapping
  between $T_C$ and $T$, and hence is independent of the choice of the anchor point.
    When $n>3$, the choice of anchor point does affect $A_T$.
    So far this effect has not been explored theoretically.
    \revZ{Numerical results to be presented in Section \ref{sec:compareMeasures}} suggest that
  the effect is generally mild and negligible. Therefore,
  unless mentioned otherwise, in the rest of this paper we set the anchor point to be the arithmetic center
$$
\vx_T \triangleq \frac{1}{n} \sum_{i=1}^n \vx_i .
$$
  \end{remark}

  \begin{remark}
    Note that \eqref{eq:ls2} can also be written as
    \begin{equation} \label{eq:ls3}
    \begin{bmatrix}A_T & \mathbf{c}\end{bmatrix} \, \begin{bmatrix}\vxi_i\\1\end{bmatrix} = \vx_i,\quad i = 1, ..., n.
    \end{equation}
    A least squares solution for the augmented equation \eqref{eq:ls3} is $[A_T\, \mathbf{c}] = E_T \tilde{E}_{T_C}^t (\tilde{E}_{T_C} \tilde{E}_{T_C}^t)^{-1}$, where
    $$
    \tilde{E}_{T_C} = \begin{bmatrix}\vxi_1& \vxi_2 & \cdots &\vxi_n \\ 1&1&\cdots&1\end{bmatrix}  \in \bbR^{3\times n} .
      $$
      When $\sum_{i=1}^n\vxi_i=\mathbf{0}$, the matrix $\tilde{E}_{T_C} \tilde{E}_{T_C}^t$
      is block diagonal and hence can be easily inverted. A simple calculation shows that 
      the least squares solution to \eqref{eq:ls3} (and thus \eqref{eq:ls2}) is the same as
      the least squares solution to \eqref{eq:ls4} with $\vxi_a = \mathbf{0}$ and $\vx_a = \vx_T$.
      In this case, $\mathbf{c} =\vx_T$.
  \end{remark}

}

Replacing $\J$ and $\M$ by $A_T$ and $\M_T$ in
(\ref{eq:contMeasurements}) and (\ref{eq:contGlobalMeasurements}),
we obtain the first set of mesh quality measures for a polygonal mesh,
\begin{align}
& Q_{ali,1}  = \max_{T \in \T} \frac{\trace ([E_T E_{T_C}^t (E_{T_C} E_{T_C}^t)^{-1}]^t \M_T [E_T E_{T_C}^t (E_{T_C} E_{T_C}^t)^{-1}]) }{2 \det ([E_T E_{T_C}^t (E_{T_C} E_{T_C}^t)^{-1}]^t \M_T [E_T E_{T_C}^t (E_{T_C} E_{T_C}^t)^{-1}])^{1/2}},
\label{ali-1} \\
& Q_{eq,1} = \max_{T \in \T} \frac{\det(E_T E_{T_C}^t (E_{T_C} E_{T_C}^t)^{-1} ) \sqrt{\det(\M_T)} }{ \sigma_{h,1}} ,
\label{eq-1}
\end{align}
where
\begin{equation}
\label{sigma-1}
\sigma_{h,1} = \frac{1}{N_{p}} \sum_{T \in \T} \det(E_T E_{T_C}^t (E_{T_C} E_{T_C}^t)^{-1} ) \sqrt{\det(\M_T)}
\end{equation}
and $N_{p}$ is the number of the polygons in $\T$.
Notice that $\sigma_{h,1}$ is not a direct approximation of $\sigma$ defined in (\ref{sigma-0}).
Instead, it is defined by \revB{summing up} the following discrete equidistribution condition over all polygons,
\[
\det(E_T E_{T_C}^t (E_{T_C} E_{T_C}^t)^{-1} ) \sqrt{\det(\M_T)}  = \sigma_{h,1} .
\]
The same strategy will be used in \revB{defining other two sets of mesh quality measures}.
This also makes the current formulas consistent with those developed \revZ{in \cite{Hua05c, HR11} for simplicial meshes}.

\revB{To show how well $Q_{ali,1}$ and $Q_{eq,1}$ work, we present numerical results obtained
  with Lloyd's algorithm for generating centroidal Voronoi tessellations (CVTs) \cite{Du99, Lloyd1982},
  a type of Voronoi tessellation where the generator of each polygon is identical to the barycenter of the polygon.
  \revi{CVTs usually consist of evenly distributed and high quality polygons under the Euclidean metric.}
  The Lloyd's algorithm is known to produce a sequence of Voronoi meshes that converges to a CVT, although very slowly.   
  }
  \revi{  
  For simplicity, we eliminate short edges in the mesh in each Lloyd's iteration, by combining nearby vertices into one.
  One may check the last paragraph of Section \ref{sec:compareMeasures} for more details.
  }
\revB{We expect to see that both $Q_{ali,1}$ and $Q_{eq,1}$ (with $\M = I$),
for the sequence of Voronoi meshes generated by the Lloyd's algorithm, decrease and converge to one
if they are correct indicators for the shape regularity and size uniformity.}

In Fig.~\ref{fig:LloydMesh8}, we show Voronoi meshes of $8\times 8$, $16\times 16$
and $32\times 32$ cells obtained with Lloyd's algorithm.
Visually we can see that the meshes are becoming better
in the sense that the cells are getting more regular in shape and more uniform in size.
To compute $Q_{ali,1}$ and $Q_{eq,1}$, we choose the unitary regular $n$-gon with vertices
$\vxi_i = (\cos \frac{2\pi i}{n}, \,\sin \frac{2\pi i}{n})^t$, $i=1, ..., n$ as $T_C$ for each $n$-gon in $\T$.
The results are reported in Fig.~\ref{fig:LloydQC1}.
One can see that both $Q_{ali,1}$ and $Q_{eq,1}$ decrease towards $1$ as the number of iterations increases.
This reflects the convergence nature of Lloyd's algorithm.
\revA{Moreover,
  the decrease is more significant in the first few iterations and then \revB{gets much slower}.
  By comparing Figs.~\ref{fig:LloydMesh8}-\ref{fig:LloydQC1} side-by-side,  
  it confirms that the quick initial drop in the quality measures
  comes from the fast ``smoothing out'' of the initial random Voronoi diagram.
  }
Furthermore, one may notice that the decrease of $Q_{ali,1}$ and $Q_{eq,1}$ is not monotone.
This may be attributed to the facts that (a) Lloyd's algorithm is not designed specifically to minimize these quality
measures and (b) $Q_{ali,1}$ and $Q_{eq,1}$ measure the quality of worst mesh elements.
Overall, we see that $Q_{ali,1}$ and $Q_{eq,1}$ correctly reflect the polygonal
mesh quality under the \revZ{Euclidean} metric.

\begin{figure}[thb]
  \begin{center}
    \includegraphics[width=12cm]{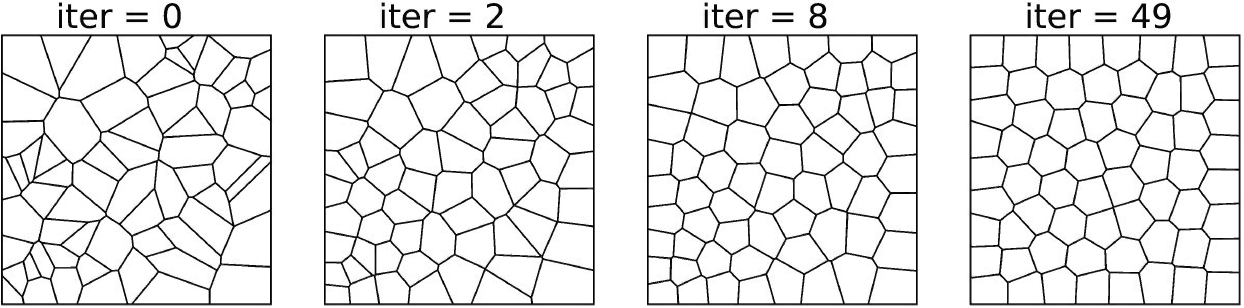}\\
    \includegraphics[width=12cm]{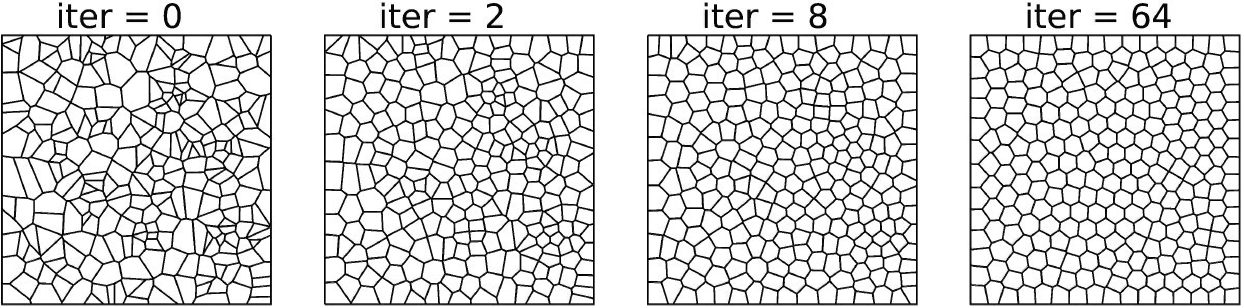}\\
    \includegraphics[width=12cm]{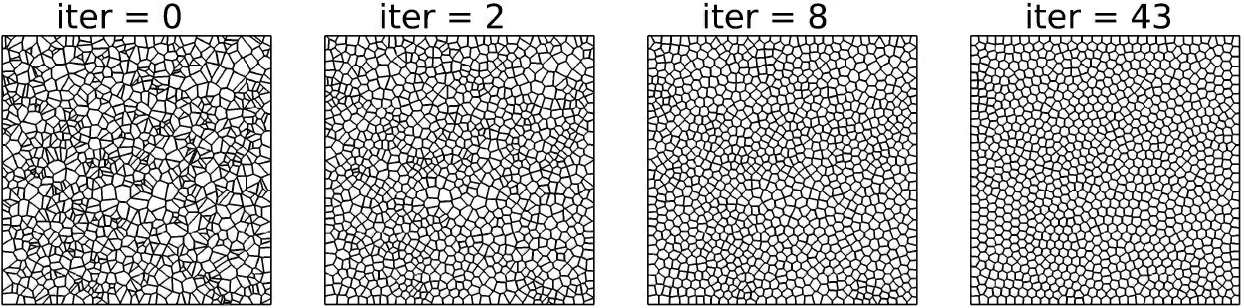}
\end{center}
\caption{Voronoi meshes of $8\times 8$, $16\times 16$, and $32\times 32$ cells in Lloyd's iteration.}
\label{fig:LloydMesh8}
\end{figure}

\begin{figure}[ht]
\begin{center}
\includegraphics[width=4cm]{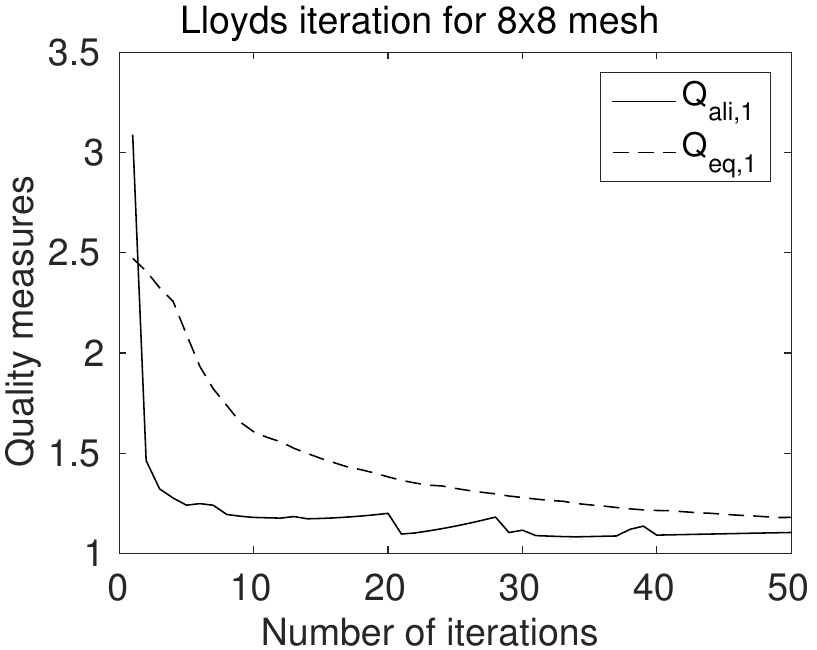} \; \includegraphics[width=4cm]{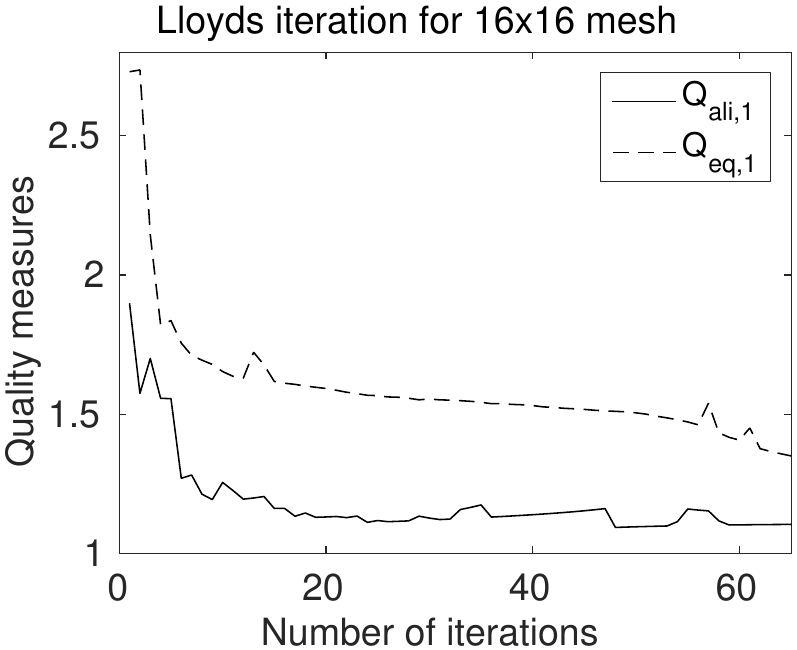} \; \includegraphics[width=4cm]{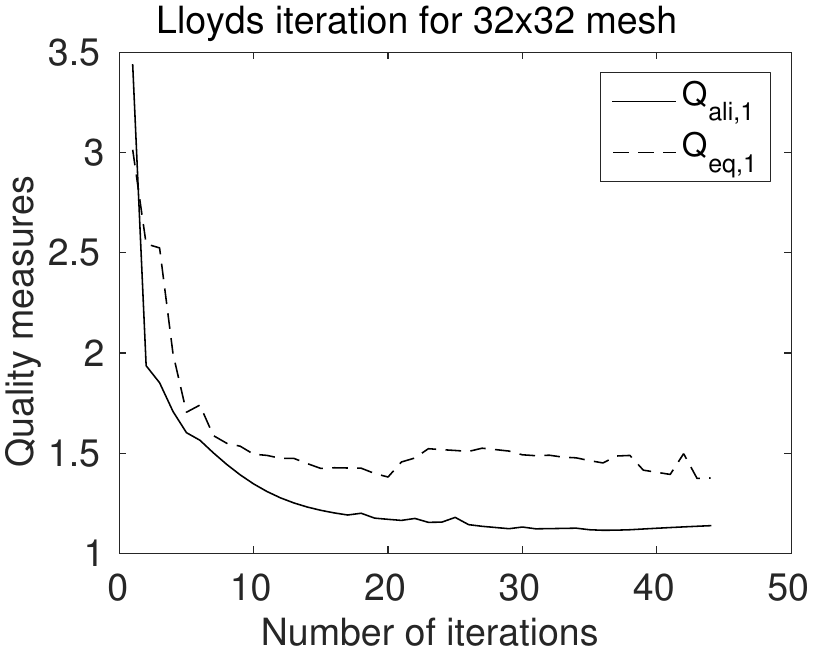}
\caption{History of $Q_{ali,1}$ and $Q_{eq,1}$ in Lloyd's iteration for meshes with $8\times 8$, $16\times 16$ and $32\times 32$ cells.}
\label{fig:LloydQC1}
\end{center}
\end{figure}

\subsection{Approximation 2: use generalized barycentric mappings}
\label{sec:oneReference}

This approximation is similar to Approximation 1 except that here we construct a specific
local mapping $\cF_T: T_C \to T$ using generalized barycentric mappings \cite{Alexa00, FloaterKosinka10, Lipman12, Schneider13}
\revZ{that} are related to generalized barycentric coordinates (GBCs)
\revZ{\cite{Floater03, Floater06, Floater15, Meyer02, Wachspress75}}.

\smallskip
\begin{definition}
\label{def:GBC}
The generalized barycentric coordinates for a given $n$-gon $T$ are the functions
$\lambda_i\,:\, T\to \bbR$, $i=1,\ldots,n$ satisfying
\begin{enumerate}
\item[(i)] (Non-negativity) All $\lambda_i(\vx)$ for $1\le i\le n$ are
non-negative on $T$;
\item[(ii)] (Linear precision) There hold
\[
\sum_{i=1}^n \lambda_i(\vx) = 1, \quad \sum_{i=1}^n \lambda_i(\vx)\vx_i = \vx,\quad
\forall \vx \in T.
\]
\end{enumerate}
\end{definition}
\smallskip

Take a set of GBCs on $T_C$ and denote it by
 $(\lambda_1(\vxi), \lambda_2(\vxi),\ldots, \lambda_n(\vxi))$.
 We can define a mapping $\cF_T$ from $T_C$ to $T$ as
 $$
 \vx = \cF_T(\vxi) \triangleq \sum_{i=1}^n \lambda_i(\vxi) \vx_i, \quad\forall \vxi\in T_C.
 $$
 Such a mapping is called a generalized barycentric mapping in literature
 and has important applications in several fields.
 An important non-trivial question is whether $\cF_T$ defines a bijection or not.
\revA{Fortunately, it has been answered positively for several types of GBCs
   \cite{Alexa00, Chen16, FloaterKosinka10, Lipman12, Schneider13, Schneider15},
including the case of piecewise linear coordinates which is discussed in the following.}


Let $\mathcal{T}_{T_C}$ be a triangulation of $T_C$.
Let $\lambda_i$, $i=1,\ldots, n$, be piecewise linear functions on $\mathcal{T}_{T_C}$ satisfying
$\lambda_i(\vxi_j) = \delta_{ij}$, \revC{where $\vxi_j$ for $j=1,\ldots, \revZ{n}$ are the vertices of $T_C$}, and $\sum_{i=1}^n \lambda_i = 1$. This defines the piecewise linear
GBCs associated with $\mathcal{T}_{T_C}$. Note that a different triangulation of $T_C$ can lead to a different set
of GBCs. Moreover, $\mathcal{T}_{T_C}$ can be a triangulation using exactly the same vertices
of $T_C$, or it can have extra vertices added as long as the extra vertices have their own and distinct
barycentric coordinates specified.
Furthermore, the piecewise linear bijection $\cF_T$ from $T_C$ to $T$ can be obtained in a different but
equivalent way: specify the same type of triangulations for $T_C$ and $T$
 and then use piecewise linear mappings to map each individual triangle in $T_C$ to the corresponding
 triangle in $T$.

\revC{In the construction of mesh quality measures, we consider the piecewise linear
barycentric mapping described above.}
Then, the Jacobian matrix $\J_T$ of the mapping $\cF_T$ is piecewise constant on each \revZ{$T$,
i.e., it is constant on each triangle of $\T_T$.}
If we also use a piecewise constant approximation of $\M$ on $T$, from
(\ref{eq:contMeasurements}) and (\ref{eq:contGlobalMeasurements}) we obtain the second set of
mesh quality measures as
\begin{align}
& Q_{ali,2}  = \max_{T \in \T} \max_{K \in \T_T}
\frac{\trace ((\J_T|_K)^t \M_K \J_T|_K ) }{2 \det ((\J_T|_K)^t \M_K \J_T|_K )^{1/2}},
\label{ali-2} \\
& \revi{ Q_{eq,2} = \max_{T \in \T} \max_{K \in \T_T}
 \frac{\det( \J_T|_K ) \sqrt{\det(\M_K)} }{ \sigma_{h,2}} , }
\label{eq-2}
\end{align}
where $\J_T|_K$ is the restriction of $\J_T$ on $K$,
\begin{equation}
\label{sigma-2}
\revi{ \sigma_{h,2} = \frac{1}{N_{tri}} \sum_{T \in \T} \sum_{K \in \T_T} \det( \J_T|_K ) \sqrt{\det(\M_K)}, }
\end{equation}
and $N_{tri}$ is the total number of the triangles in the mesh.

\revi{Similar to Section \ref{sec:app1}}, we test the second set of the mesh quality measures on the same \revB{Lloyd's iterations}
shown in Fig.~\ref{fig:LloydMesh8}.
Again, each $n$-gon is compared with a \revi{unitary} regular reference $n$-gon. There are two easy ways to define the triangular subdivision on $T$ (as well as on $T_C$):
\revZ{
\begin{itemize}
\item Subdivision (a): connect $\vx_1$ with all other vertices of $T$ and get a triangulation; 
\item Subdivision (b): connect all vertices to the arithmetic center $\vx_T$. The arithmetic center lies inside $T$ when
$T$ is convex.
\end{itemize}
}
We test both subdivisions.
The results are reported in Figs.~\ref{fig:LloydQC21} and \ref{fig:LloydQC22}.

\revi{
\begin{remark}\label{rem:Q2anchor}
In the subdivision (a), instead of choosing $\vx_1$, one can choose any $\vx_i$ and connect it with other vertices.
Similar to Section  \ref{sec:app1}, we call this chosen vertex $\vx_i$ the ``anchor point" of subdivison (a).
But unlike the case of the first set of mesh quality measures in which the choice of anchor point has negligible effects,
for the second set of quality measures with subdivision (a), the effect of the anchor point,
although not large enough to alter the asymptotic orders, is no longer negligible.
Numerical results exploring this phenomenon will be presented in Section \ref{sec:compareMeasures}.
\end{remark}

\begin{remark}
  One may choose other barycentric mappings, associated with different types of GBCs, to define $Q_{ali,2}$ and $Q_{eq,2}$.
  The framework following (\ref{eq:contMeasurements})-(\ref{eq:contGlobalMeasurements}) still works.
  However, the calculation of the Jacobian matrix may not be as simple as for the current choice.
\end{remark}
}

\begin{figure}[ht]
\begin{center}
\includegraphics[width=4cm]{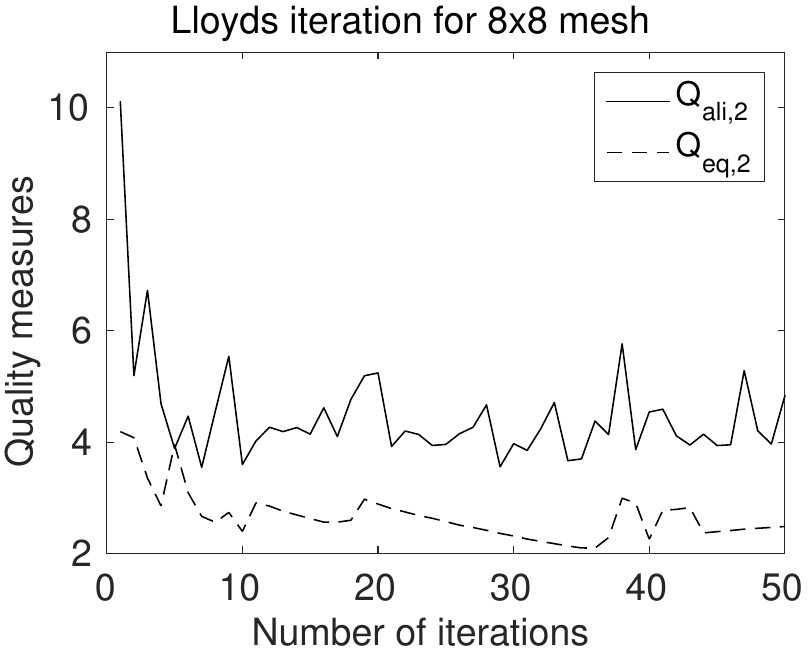} \; \includegraphics[width=4cm]{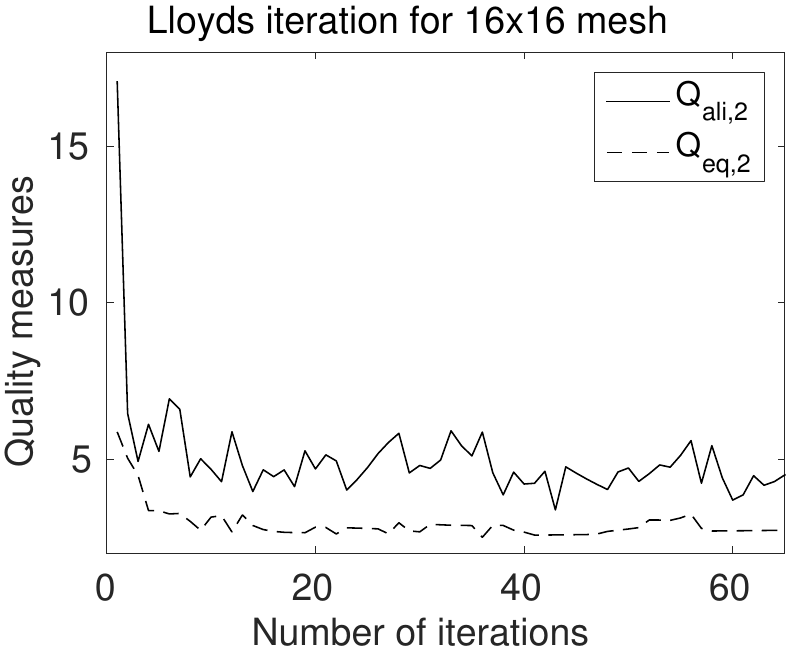} \; \includegraphics[width=4cm]{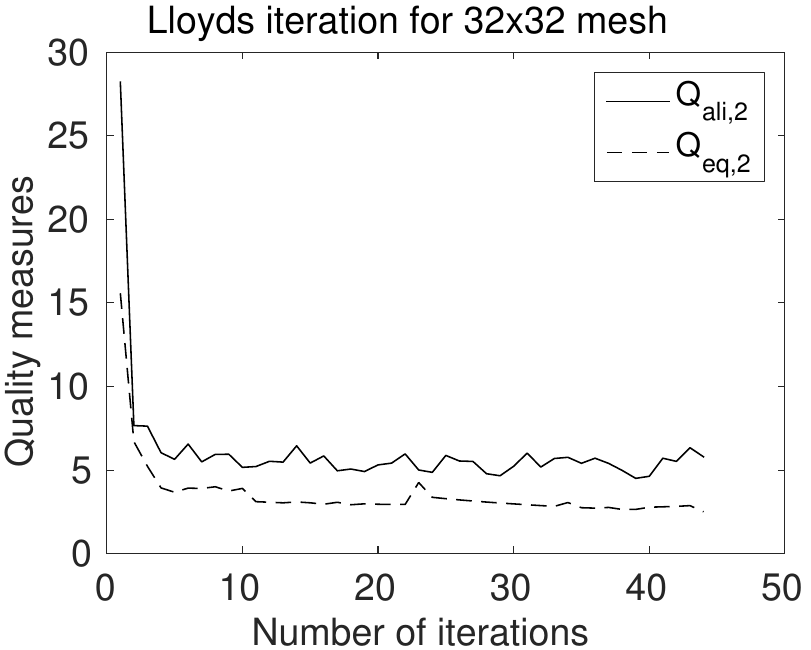}
\caption{History of $Q_{ali,2}$ and $Q_{eq,2}$ in Lloyd's iteration for meshes with $8\times 8$, $16\times 16$ and $32\times 32$ cells. Here, the triangular subdivision (a) is used.}
\label{fig:LloydQC21}
\end{center}
\end{figure}

\begin{figure}[ht]
\begin{center}
\includegraphics[width=4cm]{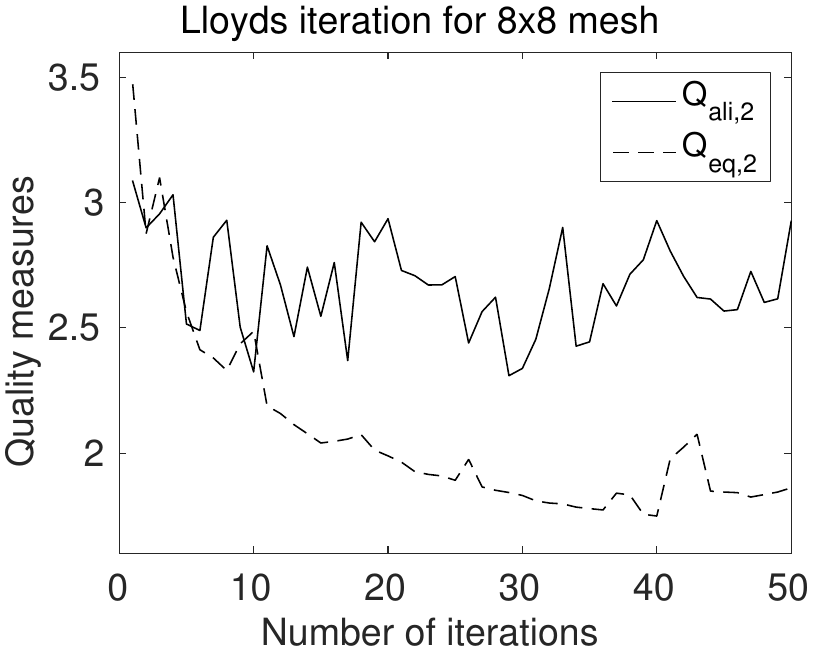} \; \includegraphics[width=4cm]{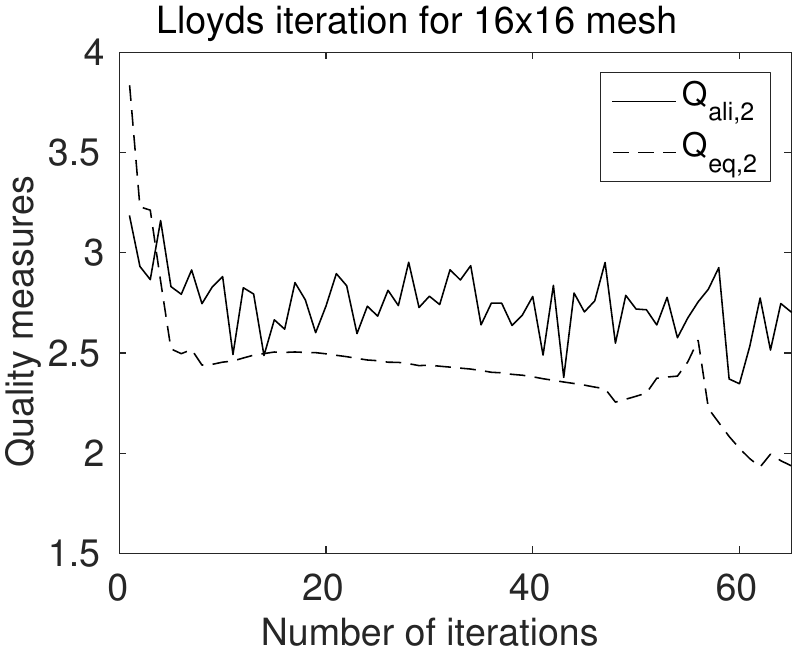} \; \includegraphics[width=4cm]{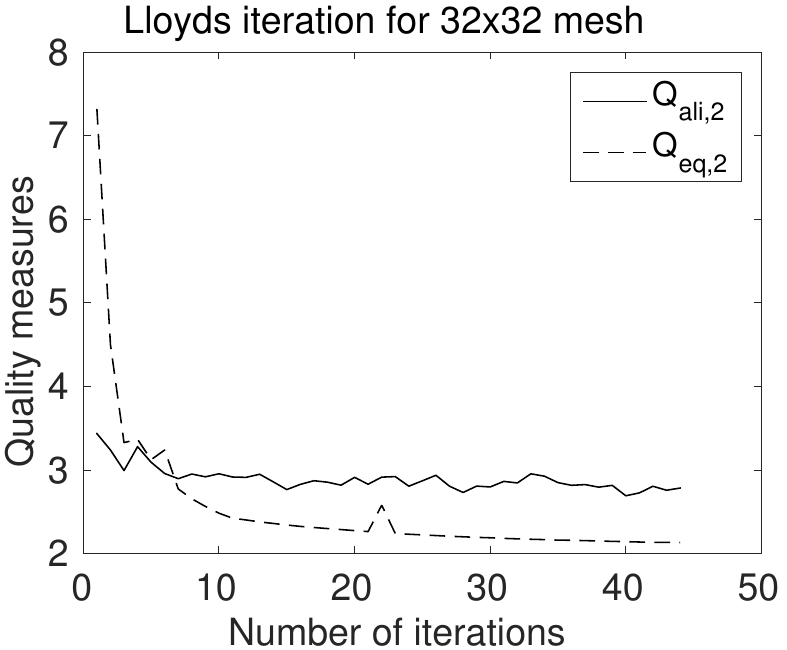}
\caption{History of $Q_{ali,2}$ and $Q_{eq,2}$ in Lloyd's iteration for meshes with $8\times 8$, $16\times 16$ and $32\times 32$ cells. Here, the triangular subdivision (b) is used.}
\label{fig:LloydQC22}
\end{center}
\end{figure}

\revB{In Figs.~\ref{fig:LloydQC21}-\ref{fig:LloydQC22}, $Q_{ali,2}$ and $Q_{eq,2}$ have similar overall
decreases as the mesh quality measures in Approximation 1,  but more dramatic drops at the first few iterations. 
The decreasing is not monotone \revZ{and} can be a bit rough in the case of $Q_{ali,2}$.} 
One may notice that the values of $Q_{ali,2}$ and $Q_{eq,2}$ are generally
greater than $Q_{ali,1}$ and $Q_{eq,1}$ and more sensitive to quality changes (with more and stronger
oscillations). This means that those Voronoi meshes have worse quality 
\revB{when measured in $Q_{ali,2}$, $Q_{eq,2}$ than in $Q_{ali,1}$, $Q_{eq,1}$.}
The reason why $Q_{ali,2}$ and $Q_{eq,2}$ are pickier is due to
their construction: they favor triangulations that look good for the selected piecewise generalized
barycentric mapping. To explain this,  we examine a random $n$-gon $T$ in a CVT,
as shown in Fig.~\ref{fig:piecewisemapping}.
Although $T$ is considered as of good shape by Lloyd's algorithm, i.e., its generator is close to its barycenter,
$T$ may still contain short edges. When both the regular reference $n$-gon and $T$ are cut into triangles which
are compared one-by-one, the triangles associated with the short edges have a bad shape.

It is also interesting to point out that subdivision (b) gives smaller  $Q_{ali,2}$ and $Q_{eq,2}$
than subdivision (a).

\begin{figure}[ht]
\begin{center}
\includegraphics[width=6cm]{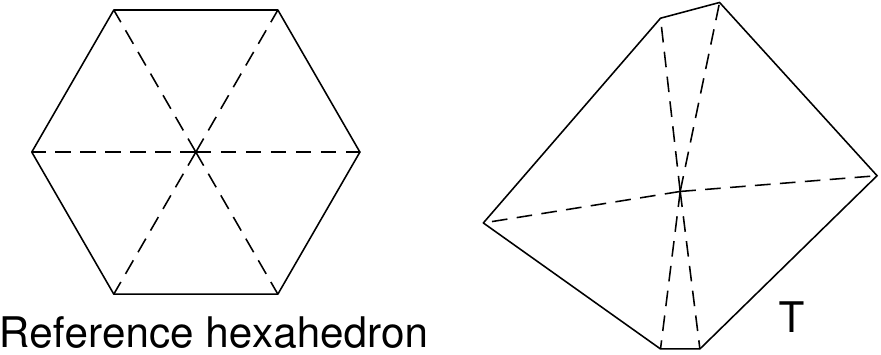}
\caption{Piecewise linear barycentric mapping from a regular reference hexahedron to $T$ using triangular subdivision (b).
Because of the short edges in $T$, some sub-triangles of $T$ can have bad shape, which affects the value of $Q_{ali,2}$.}
\label{fig:piecewisemapping}
\end{center}
\end{figure}

\subsection{\revC{Approximation 3: employ a special design that uses affine $\cF_T$}}
\label{sec:infinitelyManyReferences}

This approximation is different from the previous two.
\revC{
  It relies on the observation that an arbitrary $n$-gon can be viewed as a projection of a high dimensional simplex.
Such a connection has been known in geometry and discrete mathematics
(e.g., see \cite[Chapter 0]{Ziegler}).
Here we use it to analyze polygonal mesh quality.}

Consider a set of GBCs $\lambda_i$, $i=1, ..., n$, on $T\in\T$. Recall that
\begin{equation} \label{eq:gbcToxy}
\vx = \sum_{i=1}^n \lambda_i \vx_i,\quad \forall \vx \in T
\end{equation}
where $\vlambda = [\lambda_1,\lambda_2,\cdots,\lambda_n]^t$ lies in the set
\[
S_{n-1} = \{ [\lambda_1,\,\lambda_2,\ldots,\lambda_n]^t \in \bbR^n\,\bigg|\, \sum_{i=1}^n \lambda_i = 1\textrm{ and }
  \lambda_i\ge 0,\; 1\le i\le n\}.
\]
The set $S_{n-1}$ is a regular $(n-1)$-simplex in $\bbR^n$, meaning that it has
complete symmetry with regard to transitions over vertices, edges, and higher dimensional faces, etc.
\revC{$S_{n-1}$ is contained in an ($n-1$)-dimensional hyperplane $\sum_{i=1}^n \lambda_i = 1$ in $\bbR^n$.}

\revD{Equation (\ref{eq:gbcToxy}) can be viewed as a linear mapping from $S_{n-1}$ onto $T$,
with the center $\vlambda_c \triangleq (\frac{1}{n},\,\frac{1}{n},\ldots,\frac{1}{n})\in S_{n-1}$ mapped to the arithmetic center $\vx_T$ of $T$.
Without loss of generality, we assume that $\vx_T$ is located at the origin. 
If not, by noticing that (\ref{eq:gbcToxy}) can be rewritten as
$$
\vx - \vx_T = \sum_{i=1}^n \lambda_i (\vx_i-\vx_T),
$$
one can easily shift $\vx_T$ to the origin by introducing a coordinate translation $\hat{\vx} = \vx-\vx_T$.
}

The mapping from $S_{n-1}$ to $T$ can also be written as
\begin{equation}
\label{eq:gbcToxy2}
S_{n-1} \stackrel{B_T}{\longrightarrow} T \qquad\textrm{where} \qquad  B_T \vlambda \triangleq
\begin{bmatrix}\vx_1&\vx_2&\cdots&\vx_n\end{bmatrix} \vlambda .
\end{equation}
  Here we conveniently use $B_T$ to denote both the name of the linear mapping
  and the matrix $\begin{bmatrix}\vx_1&\vx_2&\cdots&\vx_n\end{bmatrix}$.
  Similar usage will be employed in the rest of this subsection.
Clearly, one has
$$
\begin{bmatrix} 0\\ 0\end{bmatrix} = \vx_T = B_T \vlambda_c.
$$

When $n=3$, (\ref{eq:gbcToxy2}) together with the constraint $\sum_{i=1}^n \lambda_i = 1$ gives a uniquely solvable $3\times 3$ linear system for non-degenerate $T$,
$$
\begin{bmatrix} B_T \\ \mathbf{1}^t\end{bmatrix} \vlambda \triangleq
\begin{bmatrix}\vx_1&\vx_2& \vx_3 \\ 1 & 1 &  1 \end{bmatrix} \vlambda =
\begin{bmatrix} \vx \\ 1\end{bmatrix},
$$
i.e., the linear mapping is invertible.
But for $n>3$, the mapping $B_T$ is not invertible as there is a non-trivial kernel. 
This can also be explained by the fact that $T$ is a 2-manifold while $S_{n-1}$ is an $(n-1)$-manifold.

Consider the singular value decomposition
of the matrix $B_T  \in \bbR^{2\times n}$,
$$
B_T = U_T\Sigma_T V_T^t,
$$
where
$$
\begin{aligned}
U_T &= [\vu_{1,T}\; \vu_{2,T}]\in \bbR^{2\times 2},\quad
&\Sigma_T = \begin{bmatrix}\sigma_{1,T}& 0 & 0 & \cdots & 0 \\ 0 & \sigma_{2,T} & 0 & \cdots & 0\end{bmatrix} \in \bbR^{2\times n}, \\
  V_T &= [\vv_{1,T}\; \vv_{2,T}\; \cdots\, \vv_{n,T}] \in \bbR^{n\times n},&
  \end{aligned}
  $$
with $\sigma_{1,T}\ge \sigma_{2,T}$ being the singular values, and $U_T$ and $V_T$ being orthogonal matrices.
 The singular values are non-zero
 as long as the polygon $T$ does not degenerate into a line segment.
  We further decompose
  $$ \Sigma_T = \begin{bmatrix}\sigma_{1,T}& 0 \\ 0 & \sigma_{2,T} \end{bmatrix}
    \begin{bmatrix} 1& 0 & 0 & \cdots & 0 \\ 0 & 1 & 0 & \cdots & 0\end{bmatrix}  \triangleq D_T Q_T .
  $$
  Then, the matrix $B_T$ can \revB{be rewritten as}
   $$
   B_T= (U_TD_TU_T^t) (U_TQ_TV_T^t) \triangleq \J_T\, P_T,
   $$
i.e., the linear mapping $B_T$ is decomposed into two linear mappings,
  defined by matrices $P_T$ and $\J_T$, respectively. An illustration is given in Fig.~\ref{fig:SimplexToPolygon},
  with details explained below:

 \begin{figure}[ht]
 \begin{center}
\vcenteredhbox{\includegraphics[width=2.5cm]{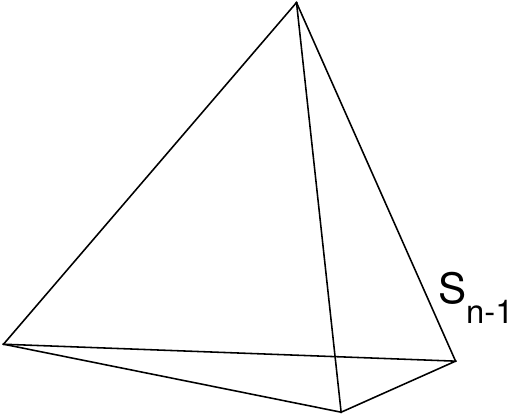}} \;$\xrightarrow{\quad P_T\quad}$\;
\vcenteredhbox{\includegraphics[width=2.5cm]{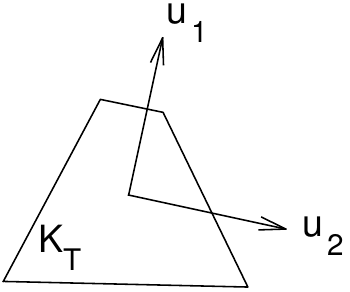}} \;$\xrightarrow{\quad\J_T\quad}$\;
\vcenteredhbox{\includegraphics[width=2.5cm]{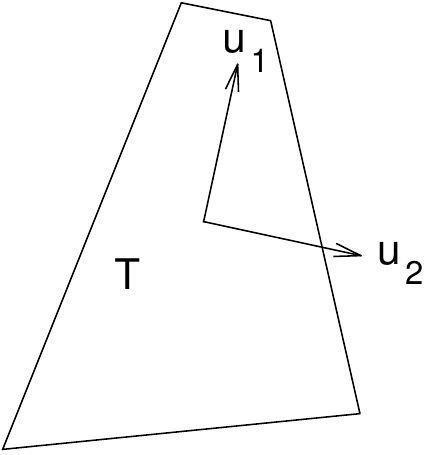}}
 \caption{Illustration of the linear mapping from simplex $S_{n-1}$ to a polygon, which is decomposed into two linear mappings.} \label{fig:SimplexToPolygon}
 \end{center}
 \end{figure}
  
\begin{itemize}
\item Note that 
\[
P_T=U_TQ_TV_T^t = U_T\begin{bmatrix} \vv_{1,T}^t \\ \vv_{2,T}^t \end{bmatrix} .
\]
Obviously, $Q_TV_T^t$ defines an orthogonal projection from $S_{n-1}$  onto the plane spanned by $\vv_{1,T}$ and $\vv_{2,T}$.
Therefore $P_T$ describes an orthogonal projection followed by a 2D rotation/reflection $U_T$.
The matrix $P_T$ maps the simplex $S_{n-1}$ into an $n$-gon (denoted by $K_T$), i.e.,
 \begin{equation}
 \label{KT-1}
 K_T = U_T\begin{bmatrix} \vv_{1,T}^t \\ \vv_{2,T}^t \end{bmatrix} (S_{n-1}) .
 \end{equation}
 We emphasize that $P_T$ consists of only orthogonal projection, rotation and/or reflection, but not scaling.
 Indeed, $P_T$ can be viewed as a {\it flattening of the high-dimensional simplex into 2D}.

\item Next, a linear transformation
$\J_T = U_TD_TU_T^t$ is applied to $K_T$, which maps polygon $K_T$ into a new polygon $T=\J_T (K_T)$.
The matrix $\J_T$ is symmetric and positive definite. Thus it defines an {\it anisotropic scaling},
i.e., scaling by factors $\sigma_{1,T}$ and $\sigma_{2,T}$ in the directions of $\vu_{1,T}$ and $\vu_{2,T}$, respectively.
Polygons $T$ and $K_T$ are affine similar under the anisotropic scaling $\J_T$.
\end{itemize}

It is important to decompose $B_T$ into a {\it flattening} $P_T$ and an {\it anisotropic scaling} $\J_T$,
since we shall show later that $K_T = P_T(S_{n-1})$ has a relatively good shape and can serve as a reference $n$-gon
for polygon $T$. Moreover, since polygons $T$ and $K_T$ are affine similar to each other,
we get an affine mapping $\cF_T(\vxi) \triangleq \J_T \vxi$ from the reference polygon $K_T$ to polygon $T$.

We first list three obvious propositions.
   \medskip

   \begin{proposition}
   \label{prop:KTconvex}
   $T$ is convex if and only if $K_T$ is convex.
   \end{proposition}

   \begin{proposition}
   \label{prop:KTnondegenerate}
   $T$ is non-degenerate if and only if $K_T$ is non-degenerate.
   \end{proposition}

   \begin{proposition} \label{prop:KTcenter}
   The arithmetic center of $K_T$ \revA{is located at the origin.}
   \end{proposition}
   \medskip

Next we show that $K_T$ has a relatively good shape \revi{ in the sense that the ratio between its outer-radius and in-radius is bounded by a modest number}.

   \begin{proposition}
   \label{prop:KTShapeRegular}
   Let $T$ be a convex $n$-gon and $K_T$ be defined in (\ref{KT-1}).
   Then, the in-radius of $K_T$ is greater than or equal to $\sqrt{\frac{1}{n(n-1)}}$,
   and the outer-radius of  $K_T$ is less than or equal to $\sqrt{\frac{n-1}{n}}$.
   \revi{Consequently, the ratio $\rho$ between the outer-radius and the in-radius of $K_T$ is bounded by
   $$
   1 \le \rho \le n-1,
   $$
   which depends on $n$ but not on the shape of $K_T$.}
   \end{proposition}

   \begin{proof}
   Recall that $S_{n-1}$ lies inside the hyperplane $\sum_{i=1}^n \lambda_i = 1$.
   By Proposition \ref{prop:KTcenter} and linearity, one has
   $$\begin{bmatrix}0\\0\end{bmatrix} = P_T \vlambda_c =
    U_T \begin{bmatrix} \vv_{1,T}\cdot\vlambda_c \\   \vv_{2,T}\cdot\vlambda_c\end{bmatrix} ,
    $$
    which implies that $\vv_{1,T}\cdot\vlambda_c = \vv_{2,T}\cdot\vlambda_c = 0$.
    Consequently, vectors in $\text{span}\{\vv_{1,T},\, \vv_{2,T}\}$ are orthogonal to
    $\vlambda_c= [\frac{1}{n},\ldots,\frac{1}{n}]^t$, which is also the normal direction of
    the hyperplane $\sum_{i=1}^n \lambda_i = 0$.
    In other words, the 2D plane spanned by $\vv_{1,T}$ and $\vv_{2,T}$
    lies inside the hyperplane $\sum_{i=1}^n \lambda_i = 0$,
    which is parallel to the hyperplane $\sum_{i=1}^n \lambda_i = 1$.
    This is important as it guarantees that the orthogonal projection of any ball inside the hyperplane
    $\sum_{i=1}^n \lambda_i = 1$  into $\text{span}\{\vv_{1,T},\, \vv_{2,T}\}$
    must be a circular disk, as illustrated in Fig.~\ref{fig:hyperplaneProjection}.

    \begin{figure}[ht]
    \begin{center}
    \includegraphics[width=8cm]{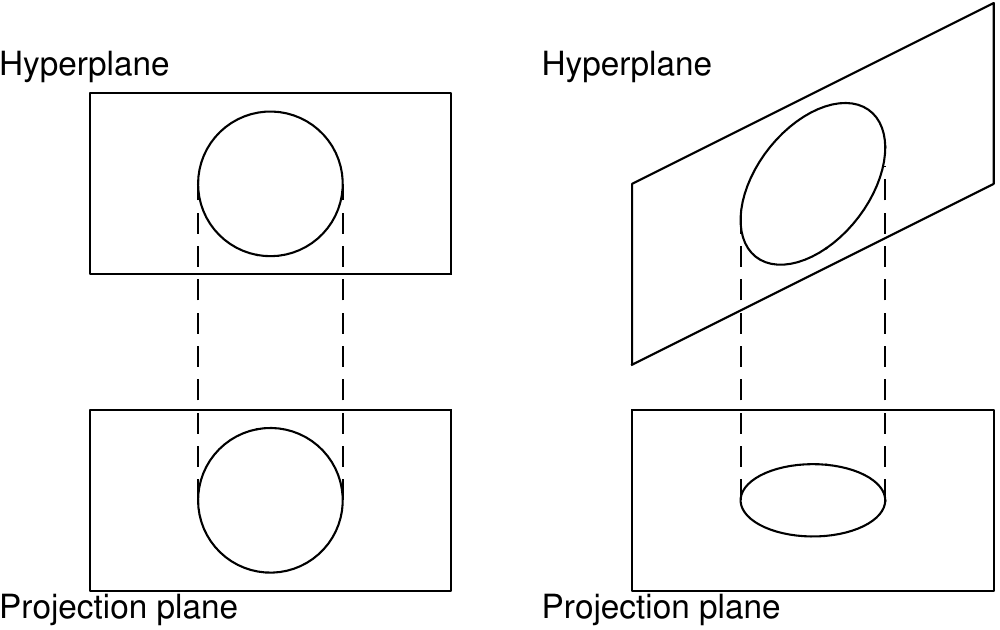}
    \caption{Orthogonal projection from a hyperplane to a 2D plane. Left:
    If the 2D plane lies inside a parallel hyperplane, then projection of any ball in the hyperplane gives a circular disk;
    Right: if not parallel, the projection becomes an ellipse.
    \revA{Because of 3D perspective, the ball in the hyperplane may look like an ellipse in the right panel.}} \label{fig:hyperplaneProjection}
    \end{center}
    \end{figure}

   By Proposition \ref{prop:KTconvex}, polygon $K_T$ is also convex.
   Hence the projection $P_T$ maps the inscribed ball of $S_{n-1}$ into an inner disk of $K_T$, and
   the circumscribed ball of $S_{n-1}$ into an outer disk of $K_T$. One simply needs to compute the radii of the inscribed and
   the circumscribed ball of $S_{n-1}$, which are just $\sqrt{\frac{1}{n(n-1)}}$ and $\sqrt{\frac{n-1}{n}}$.
   This completes the proof of the lemma.
   \end{proof}
    \smallskip

   \begin{remark}
   For $n>3$, the projection of the  inscribed (\revZ{resp.,} circumscribed) ball of $S_{n-1}$ under $P_T$
   is not necessarily the largest disk in $K_T$ (\revZ{resp.,} the smallest disk containing $K_T$).
    \end{remark}
    \smallskip

   Because of the uniform bound for $\rho$ stated in Proposition \ref{prop:KTShapeRegular}, $K_T$
   has a relatively good shape and can thus be used as a reference $n$-gon.
   \smallskip

\begin{figure}[ht]
\begin{center}
\includegraphics[width=8cm]{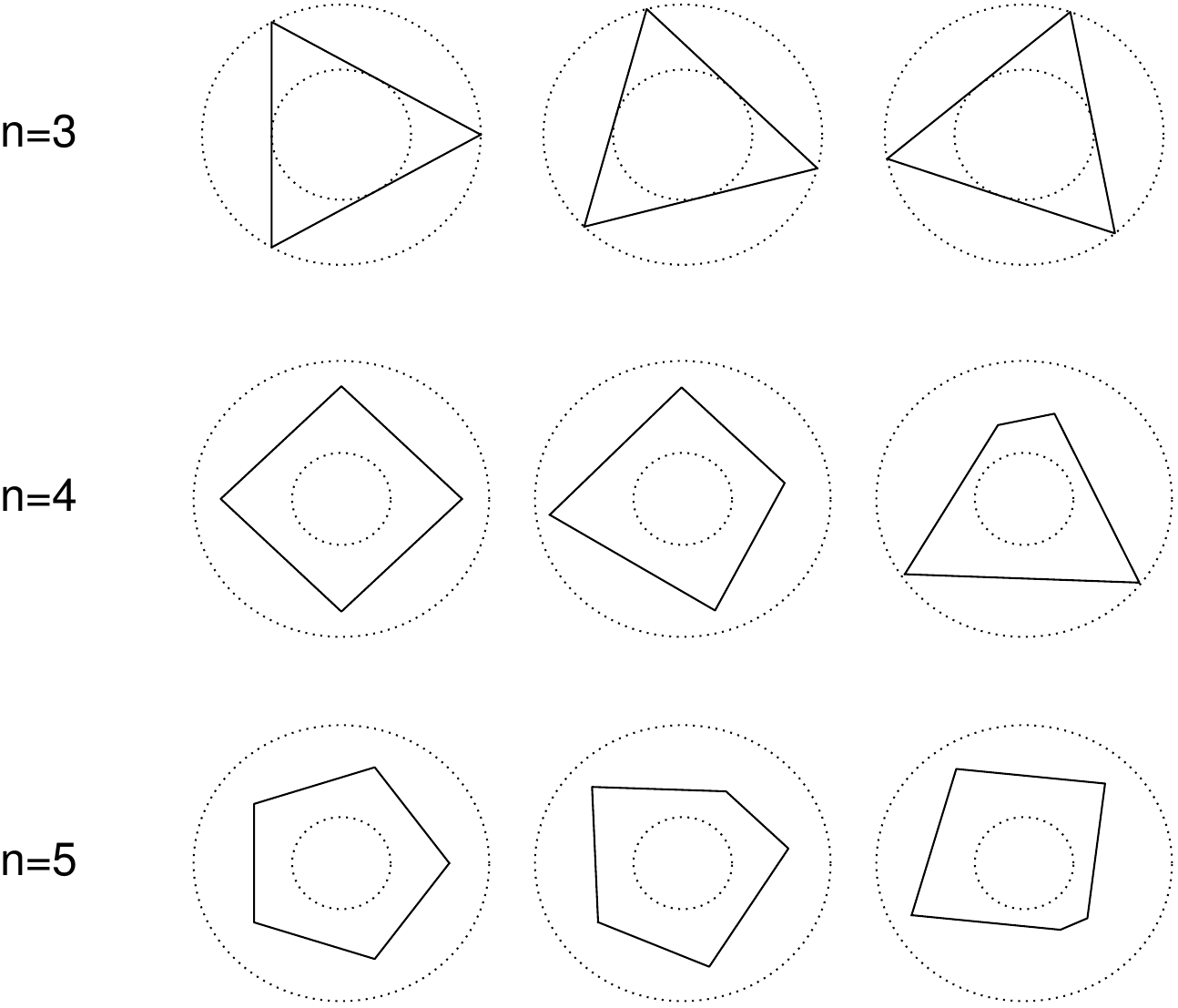}
\caption{Examples of reference polygons. The dotted circles are the projection of the inscribed and circumscribed balls of $S_{n-1}$ under $P_T$.} \label{fig:refPolygons}
\end{center}
\end{figure}

A few possible $K_T$'s are shown in Fig.~\ref{fig:refPolygons}.
   When $n=3$, the simplex $S_2$ is an equilateral triangle lying in the plane
   $\lambda_1+\lambda_2+\lambda_3=1$ and the condition $\vv_1\cdot \mathbf{1} = \vv_2 \cdot\mathbf{1}= 0$
   implies that $\text{span}\{\vv_1,\vv_2\}$ is just the plane $\lambda_1+\lambda_2+\lambda_3 = 0$.
   Thus the reference $3$-gons, or the reference triangles, are simply rotated/reflected equilateral triangles
   with edge length $\sqrt{2}$, as shown in Fig.~\ref{fig:refPolygons}.
   In other words, the  (equilateral) reference triangle is unique up to rotation/reflection.
   When $n>3$, on the other hand, the reference $n$-gons will no longer be unique even after rotation/reflection.
   However, they should all lie between two circles with radii $\sqrt{\frac{1}{n(n-1)}}$ and $\sqrt{\frac{n-1}{n}}$,
   as stated in Proposition \ref{prop:KTShapeRegular}; and they should all be convex and non-degenerate,
   as verified in propositions \ref{prop:KTconvex} and \ref{prop:KTnondegenerate}.

Now we are ready to define the mesh quality measures using the reference $n$-gons.
Recall that each $T$ has a unique $K_T$ and $\T_C$ is the collection of all those $K_T$'s.
(We do not need to compute $K_T$'s when evaluating the mesh quality measures.)
The Jacobian matrix $\J_T$ of $\cF_T$ is defined as
\begin{equation}
\J_T = U_T \begin{bmatrix} \sigma_{1,T} & 0 \\ 0 & \sigma_{2,T} \end{bmatrix} U_T^t ,
\label{JT-3}
\end{equation}
where $U_T$ and $\sigma_{1,T}$ and $\sigma_{2,T}$ are the left
singular vectors and singular values of the matrix $B_T = [\vx_1-\vx_T, ..., \vx_n-\vx_T]$.
(Here we put back the translation just for convenience.)
Using the metric tensor approximation (\ref{eq:def1M}), we can then define the mesh quality measures as
\begin{align}
& Q_{ali,3}  = \max_{T \in \T} \frac{\trace (\J_T^t \M_T \J_T) }{2 \det (\J_T^t \M_T \J_T)^{1/2}},
\label{ali-3} \\
& \revi{Q_{eq,3} = \max_{T \in \T} \frac{\det(\J_T ) \sqrt{\det(\M_T)} }{ \sigma_{h,3}} ,}
\label{eq-3}
\end{align}
where $\J_T$ is defined in (\ref{JT-3}),
\begin{equation}
\label{sigma-3}
\revi{\sigma_{h,3} = \frac{1}{N_p} \sum_{T \in \T} \det(\J_T) \sqrt{\det(\M_T)},}
\end{equation}
and $N_p$ is the number of the polygons in $\T$.

We have tested the third set of mesh quality measures numerically on the same Lloyd iterations shown
in Fig.~\ref{fig:LloydMesh8}. The results are reported in Fig.~\ref{fig:LloydQC3}.
Comparing with the numerical results for $Q_{ali,1}$ and $Q_{eq,1}$ (cf. Fig.~ \ref{fig:LloydQC1}
and Table~\ref{tab:LloydQC}), one may notice that
$Q_{ali,3}$ is slightly smaller than $Q_{ail,1}$ but $Q_{eq,3}$ is slightly larger than $Q_{eq,1}$ for the same mesh.
This is because in the computation for the first approximation, regular $n$-gons are used as reference $n$-gons,
which gives a precise control of the size of reference $n$-gons; while in the third approximation, the size of
reference $n$-gons varies a bit as can be seen in Fig.~\ref{fig:refPolygons}.
Overall, $Q_{ali,3}$ and $Q_{eq,3}$ also provide a good measure on the quality of Voronoi meshes.
A major advantage of this approximation is that all $\cF_T$'s are affine, which will be useful
in error estimation in \revB{the numerical solution of partial differential equations.}


\begin{figure}[ht]
\begin{center}
\includegraphics[width=4cm]{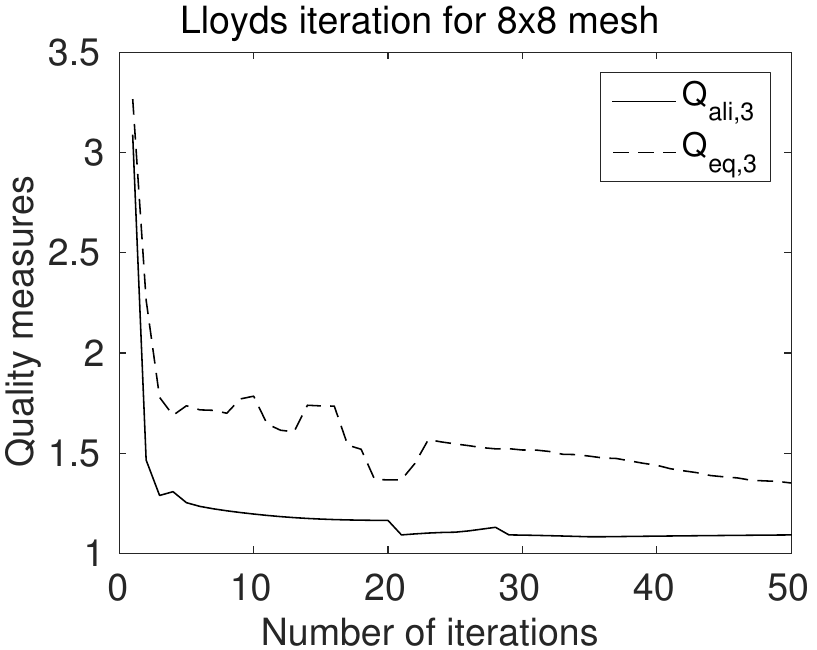} \; \includegraphics[width=4cm]{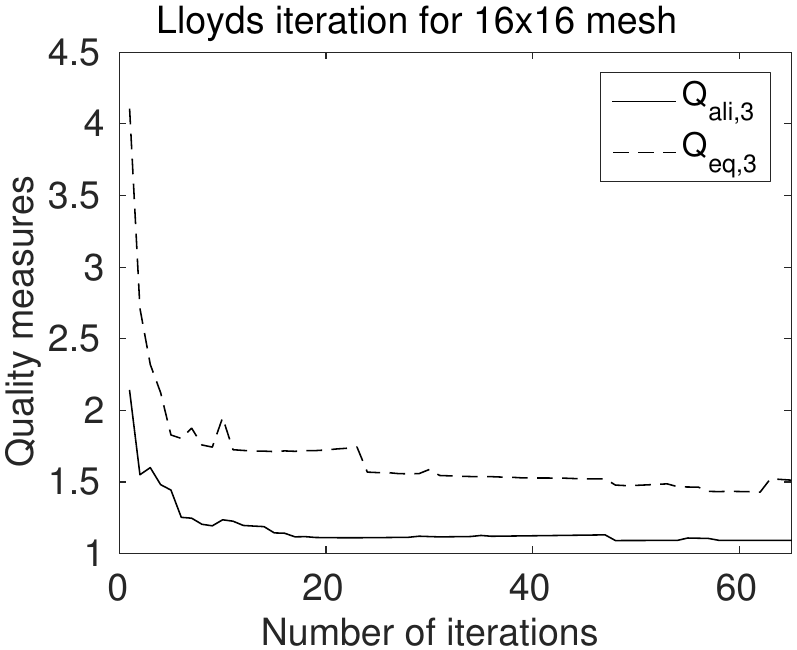} \; \includegraphics[width=4cm]{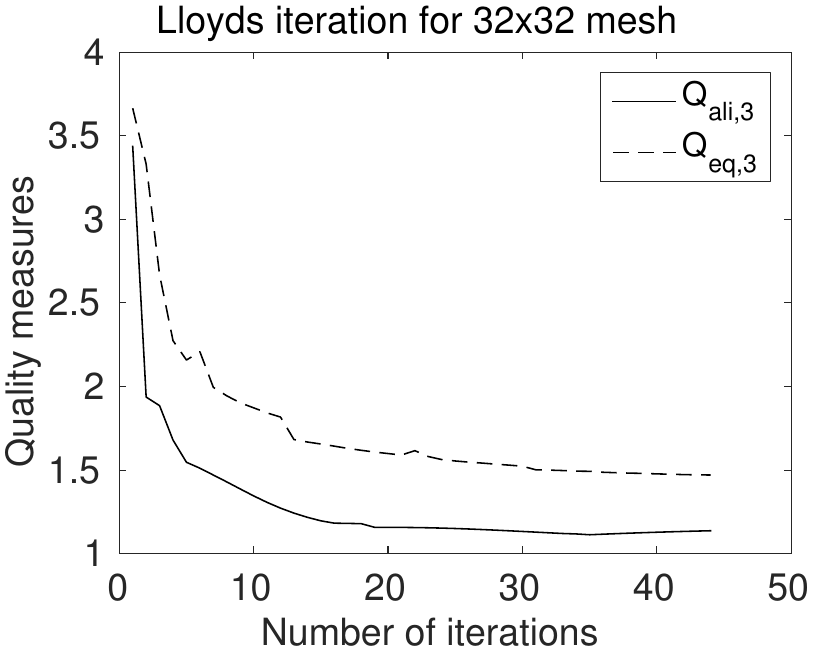}
\caption{History of $Q_{ali,3}$ and $Q_{eq,3}$ in Lloyd's iteration for meshes with $8\times 8$, $16\times 16$ and $32\times 32$ cells.}
\label{fig:LloydQC3}
\end{center}
\end{figure}

\revi{
\subsection{More numerical comparison of three mesh quality measures}
\label{sec:compareMeasures}
It is unclear to us at the moment how to compare the performance of the three sets of mesh quality measures theoretically.
For this reason, we present more numerical comparison results in this subsection.
Previously, we have compared the numerical behavior of the quality measures on the Lloyd's iteration, applied to the entire mesh.
This time the quality measures are compared on individual polygons, in order to reveal the subtle but rich differences among them.
On a single polygon $T$, following \eqref{ali-1}-\eqref{eq-1}, \eqref{ali-2}-\eqref{eq-2} and \eqref{ali-3}-\eqref{eq-3}, 
the quality measures can be written as
$$
\begin{aligned}
q_{ali,1} &= \frac{\trace (A_T^{t}\M_TA_T) } {2 \det (A_T^{t}\M_TA_T)^{1/2}}, \quad & q_{eq,1} &= \det(A_T)\sqrt{\det(\M_T)}, \quad &&\textrm{where } A_T\textrm{ is defined in \eqref{J-1} }, \\
q_{ali,2} &= \max_{K \in \T_T}\frac{\trace ( (\J_T|_K)^{t}\M_K \J_T|_K) } {2 \det ( (\J_T|_K)^{t} \M_K\J_T|_K)^{1/2}}, \quad & q_{eq,2} &= \max_{K \in \T_T}\det( \J_T|_K)\sqrt{\det(\M_K)}, 
    \quad &&\textrm{where }\J_T\textrm{ is piecewisely defined}, \\
q_{ali,3} &= \frac{\trace (\J_T^{t}\M_T\J_T) } {2 \det (\J_T^{t}\M_T\J_T)^{1/2}}, \quad & q_{eq,3} &= \det(\J_T)\sqrt{\det(\M_T)}, \quad &&\textrm{where }\J_T\textrm{ is defined in \eqref{JT-3} }.
\end{aligned}
$$
In the above we have conveniently set $\sigma_{h,1}=\sigma_{h,2}=\sigma_{h,3}=1$.
For the first two sets of quality measures, the reference polygon is set to be the unitary regular $n$-gon.

The first test is designed to check the robustness of the quality measures with respect to the choice of the ``anchor point", 
which is used in the construction of the first two quality measures; see remarks \ref{rem:Q1anchor} and \ref{rem:Q2anchor}.
To this end, we generate some random polygons using the following steps.
\begin{enumerate}
\item Generate a random convex polygon by calculating the convex hull of $20$ random points in $[0,1]\times[0,1]$. Then shift the arithmetic center of the polygon to the origin.
\item Generate a random $2\times 2$ orthogonal matrix $U$. Apply a linear transformation with the coefficient matrix
\begin{equation} \label{eq:randomU}
U^t \begin{bmatrix} a&0\\0&1\end{bmatrix} U 
\end{equation}
to the random polygon. This stretches/compresses the polygon with factor $a$ in the direction of the first row vector of $U$.
\end{enumerate}
When $a<<1$, the above process generates a random anisotropic convex polygon, aligned with a random direction and with size of roughly $a\times 1$.
In the left panel of Fig. \ref{fig:randPolygons}, we draw a random polygon with $a=0.1$. 
It can have relatively short edges, but its diameter is of $O(1)$.

\begin{figure}[ht]
\begin{center}
\includegraphics[width=4cm]{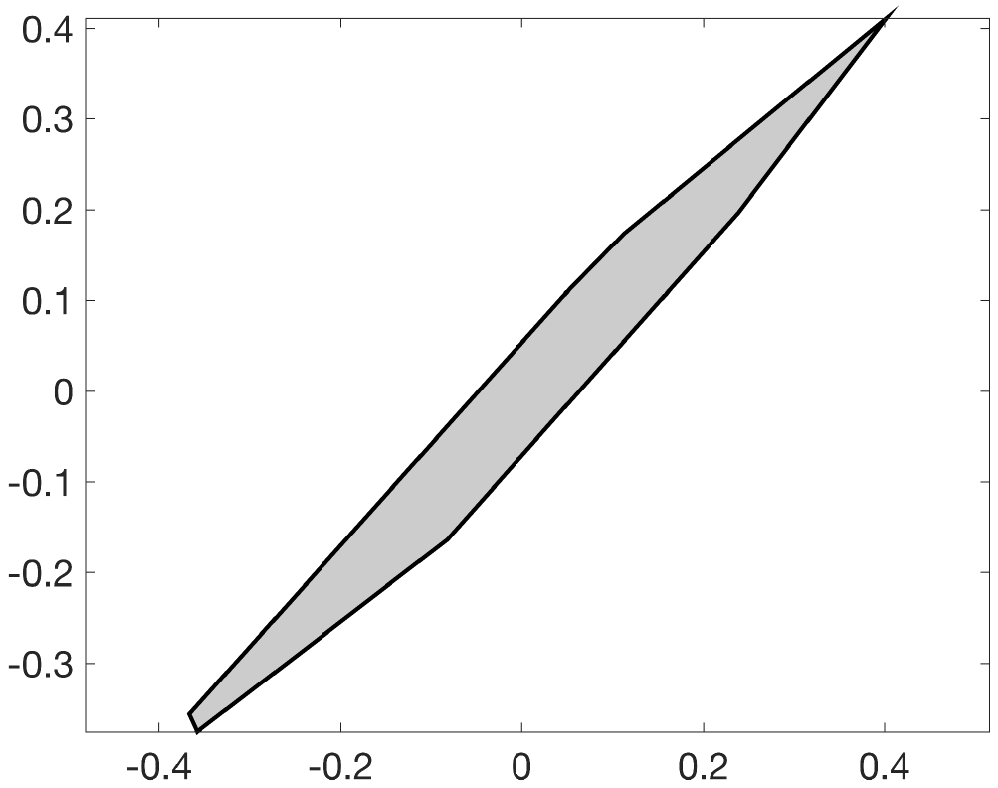} \; \includegraphics[width=4cm]{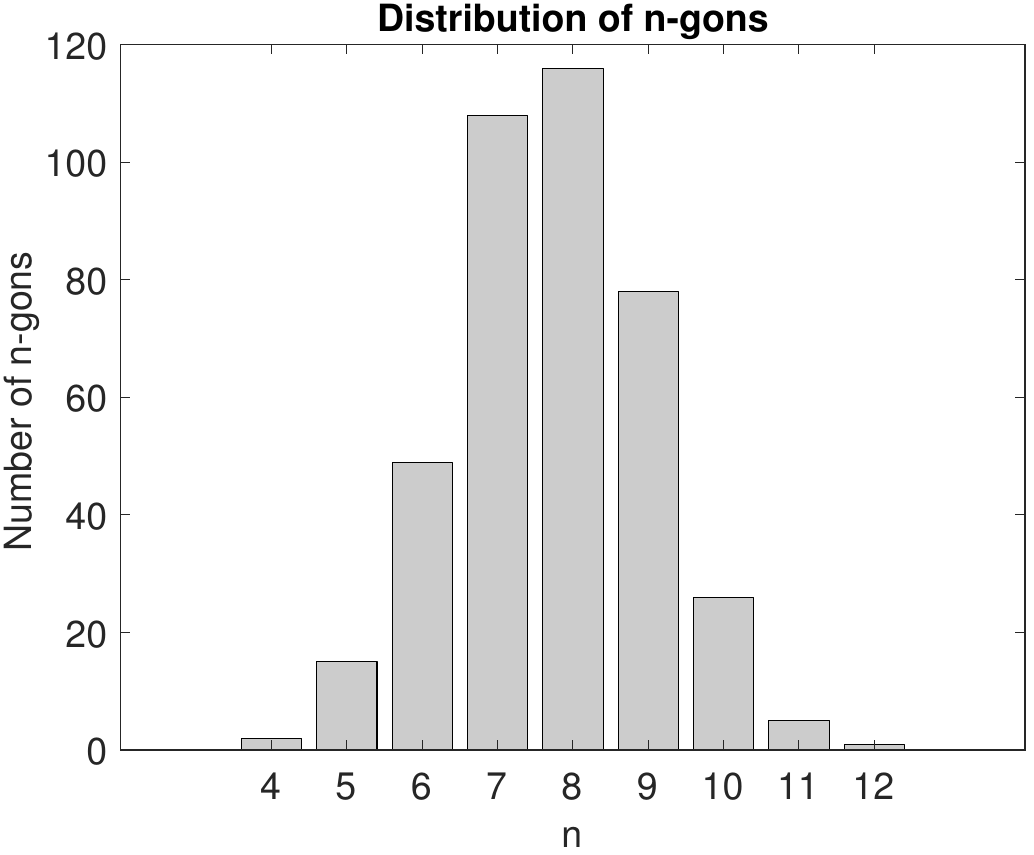} \; \includegraphics[width=4cm]{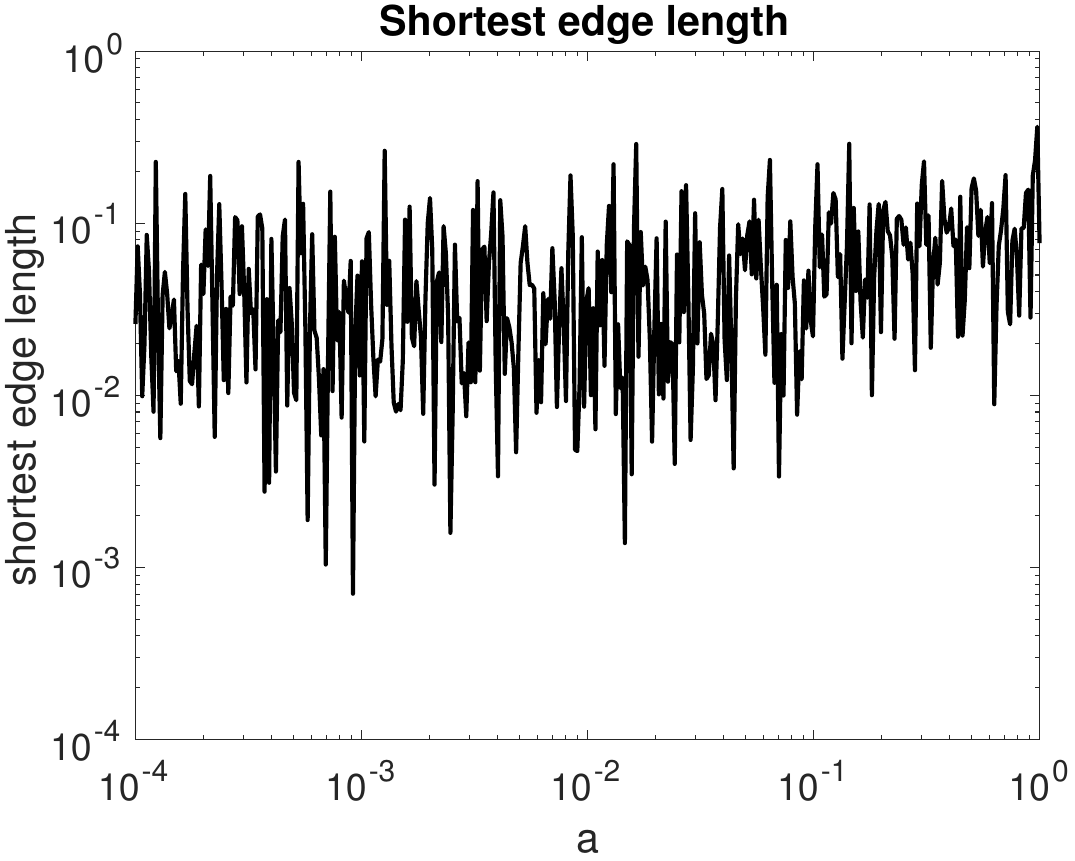}
\caption{Left: a random polygon with $a=0.1$. Middle: distribution of $n$-gons in the first test. Right: The shortest edge length vs. the value of $a$ in the first test.}
\label{fig:randPolygons}
\end{center}
\end{figure}

In the test run, we generate $400$ polygons using the process described above, with various $a\in [10^{-4},1]$. 
In the middle panel of Fig.~\ref{fig:randPolygons}, the distribution of $n$-gons appear to be normal with respect to $n$, which is reasonable because 
the polygons are generated using convex hulls.
We plot the shortest edge length of all test polygons in the right panel of Fig. \ref{fig:randPolygons}.
The majority of the shortest edge length appears to lie in $[10^{-2},10^{-1}]$, with a few going down below $10^{-3}$. However, the shortest edge length seems to have no or very weak relation to $a$.

The test polygons are highly anisotropic when $a$ is small. Therefore when measured under the Euclidean metric, they are considered having a bad shape. 
But these polygons are of good shape if measured under the correct metric
\begin{equation} \label{eq:test1metric}
\M_T=U^t\begin{bmatrix} \frac{1}{a^2}&0 \\ 0 & 1\end{bmatrix}U ,
\end{equation}
with the same $U$ used in \eqref{eq:randomU} for each random polygon $T$.
In Figs. \ref{fig:randAli}-\ref{fig:randEq} and \ref{fig:randMAli}-\ref{fig:randMEq}, we report the mesh quality measures computed using $\M_T = I$ and the metric defined in \eqref{eq:test1metric}, respectively.
In the experiment, we compute the first and second sets of mesh quality measures
with the anchor point rotating among all vertices.
The maximum and minimum values of the resulting quality measures are plotted in Figs. ~\ref{fig:randAli}-\ref{fig:randMEq}, with the regions between them shaded in gray. 
Then we compute the first quality measures using the arithmetic center as the anchor point, the second quality measures using subdivision (b), and the third quality measures.
The results are plotted in dark curves.

\begin{figure}[ht]
\begin{center}
\includegraphics[width=4cm]{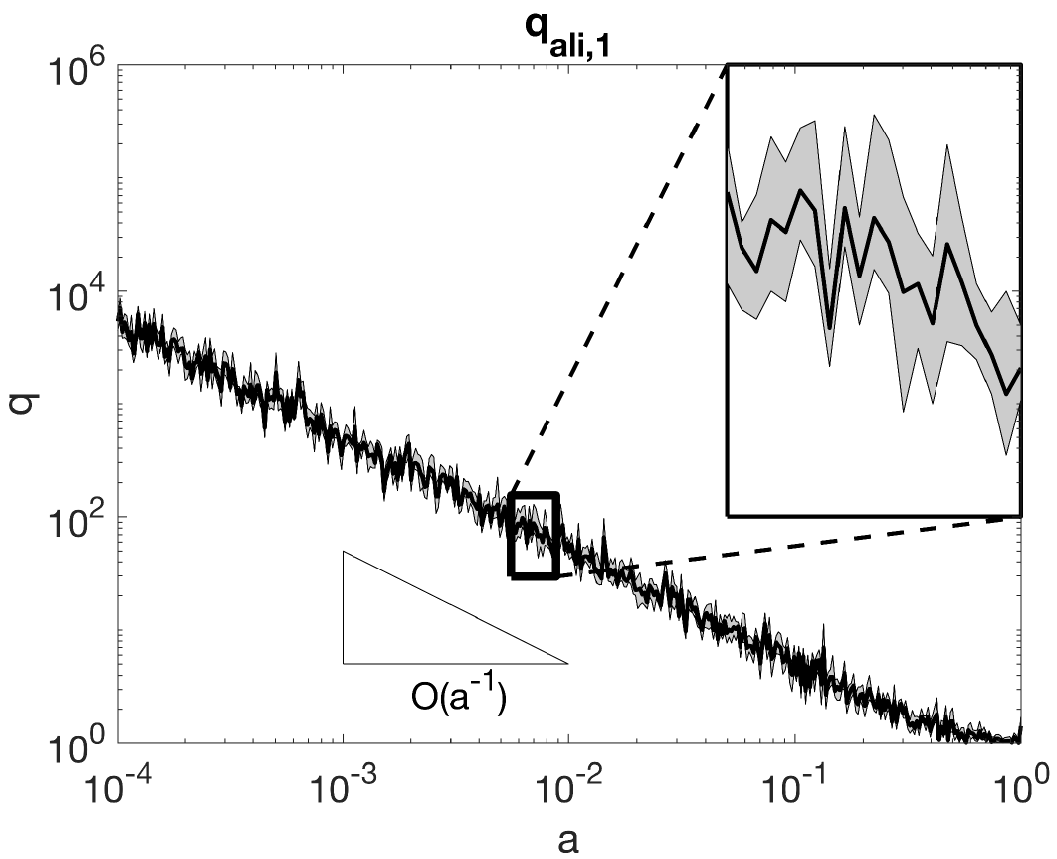} \; \includegraphics[width=4cm]{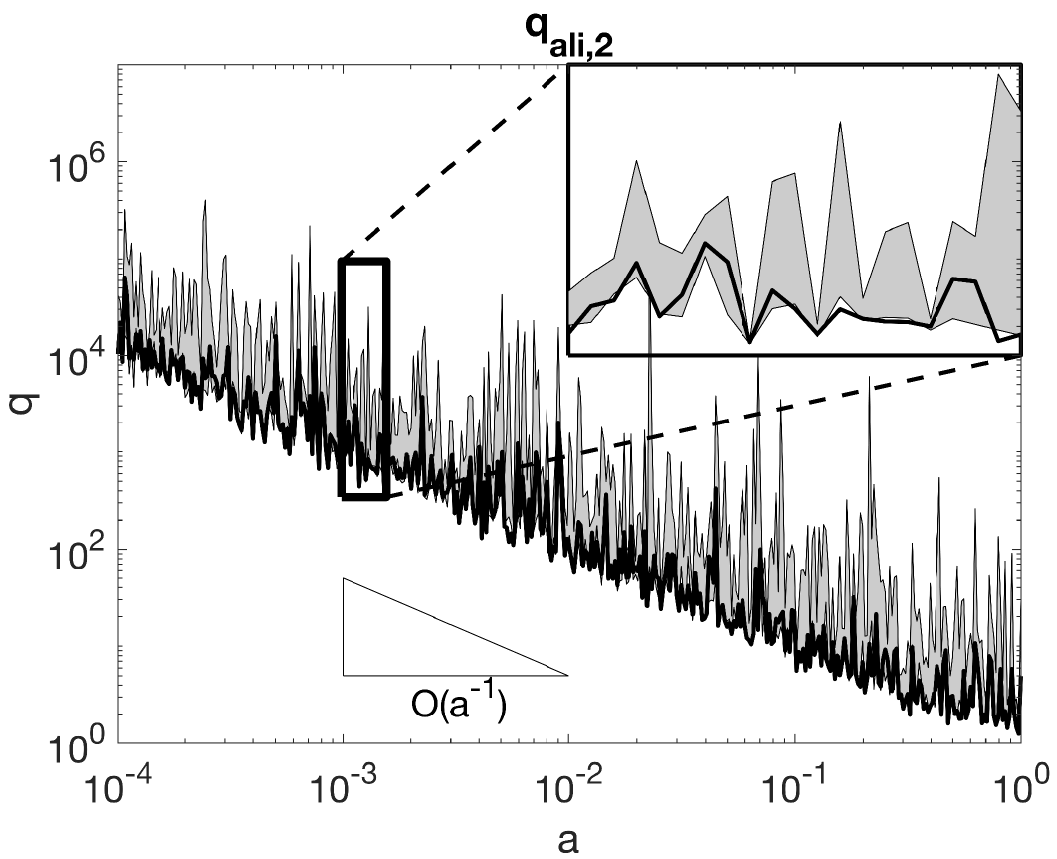} \; \includegraphics[width=4cm]{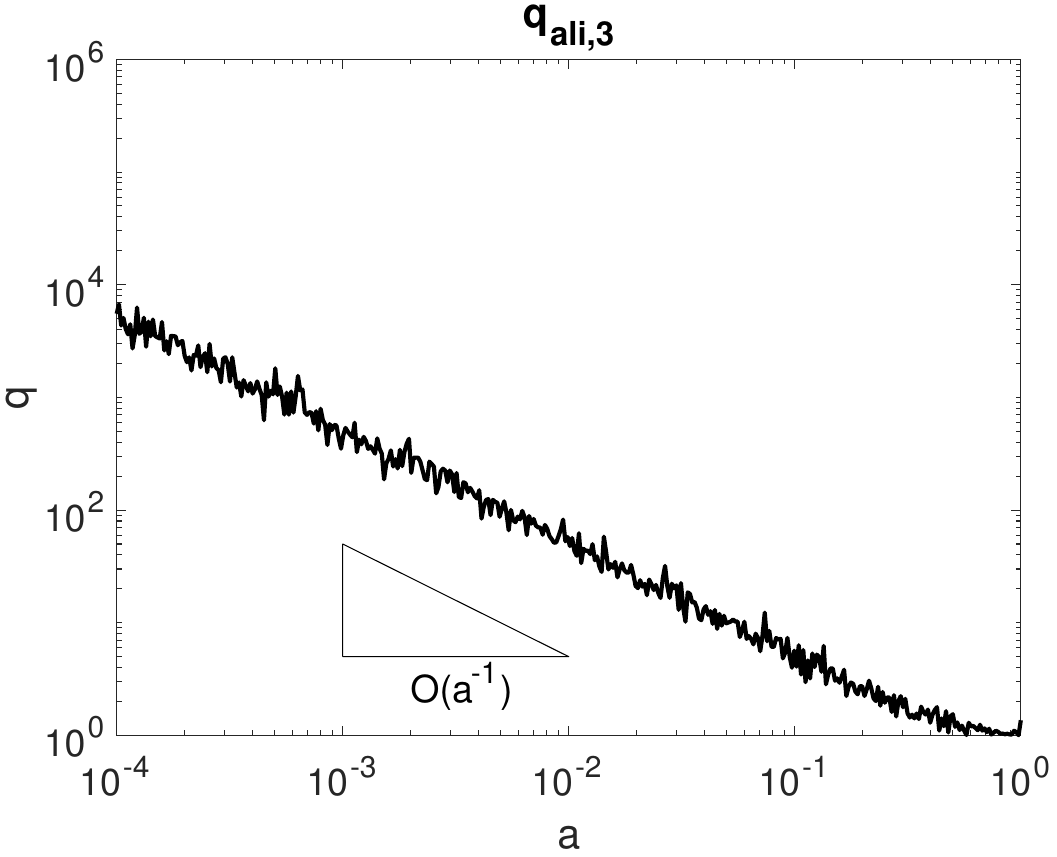}
\caption{Alignment measure vs. $a$ in the first test, computed with $\M_T=I$. In the left panel, the shaded region denotes the range of $q_{ali,1}$ with the anchor point rotating among all vertices,
while the dark curve denotes $q_{ali,1}$ using the arithmetic center as the anchor point. 
\revAA{An inset showing a magnified region is also presented.}
In the middle panel, the shaded region denotes the range of $q_{ali,2}$ using subdivision (a)
with the anchor point rotating among all vertices, while the dark curve denotes $q_{ali,2}$ using subdivision (b).
\revAA{In the right panel, the dark curve denotes $q_{ali,3}$.}
}
\label{fig:randAli}
\end{center}
\end{figure}

\begin{figure}[ht]
\begin{center}
\includegraphics[width=4cm]{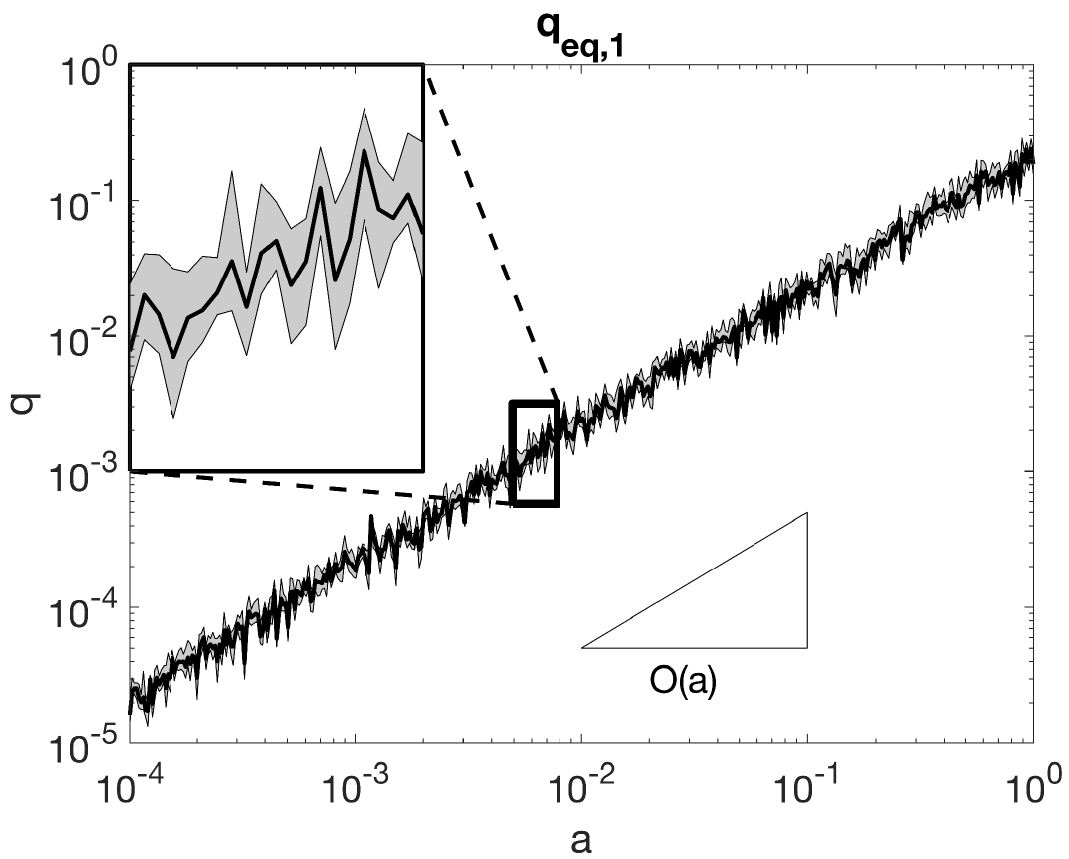} \; \includegraphics[width=4cm]{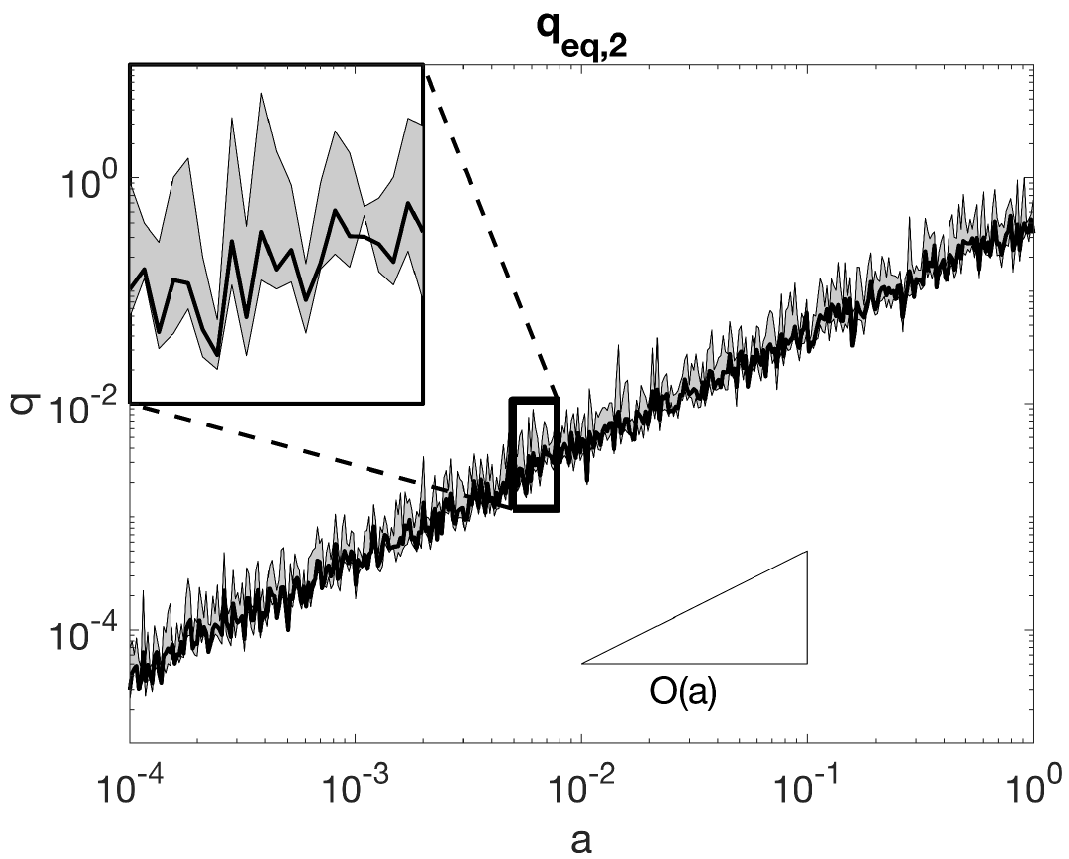} \; \includegraphics[width=4cm]{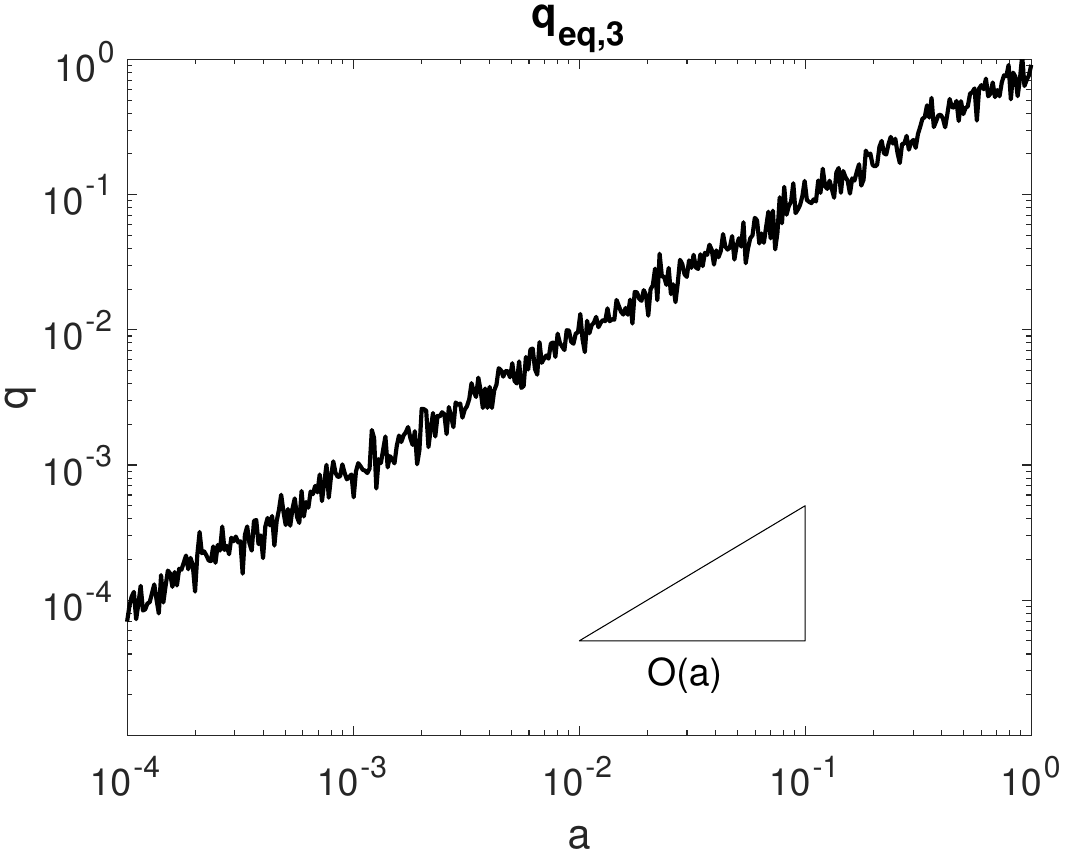}
\caption{Equidistribution measure vs. $a$ in the first test, computed with $\M_T=I$. See the caption of Fig. \ref{fig:randAli} for more details.}
\label{fig:randEq}
\end{center}
\end{figure}

\begin{figure}[ht]
\begin{center}
\includegraphics[width=4cm]{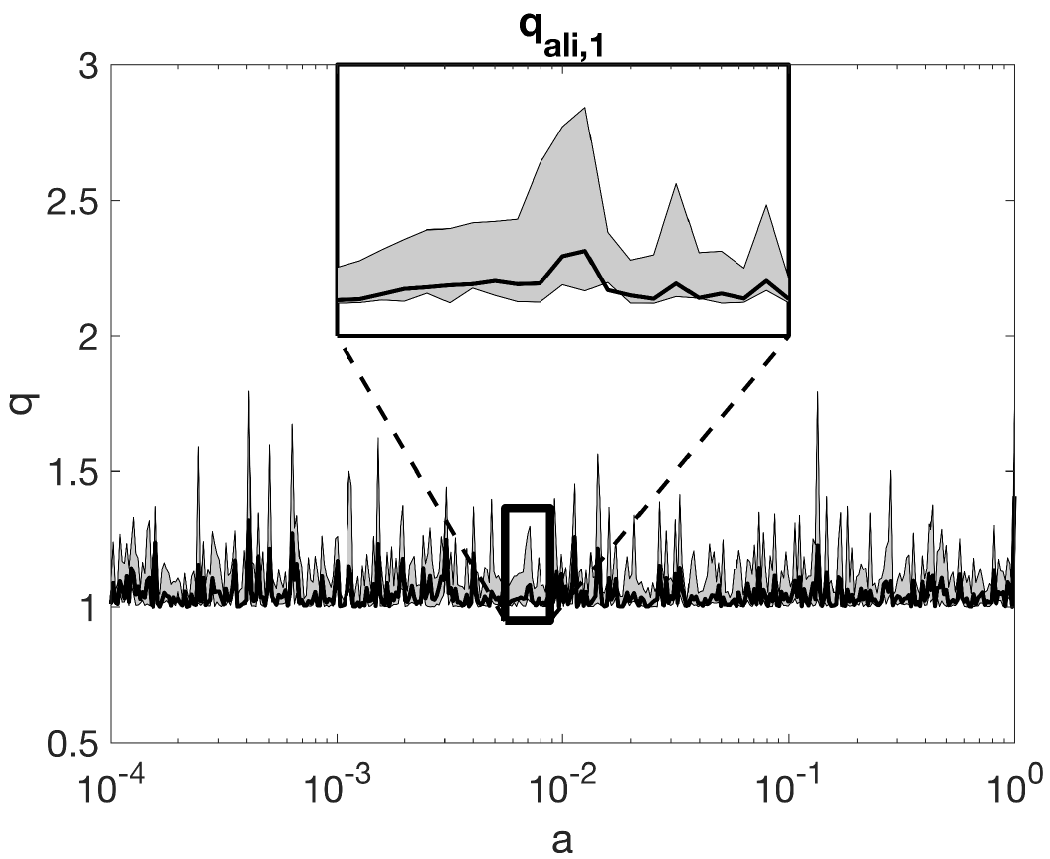} \; \includegraphics[width=4cm]{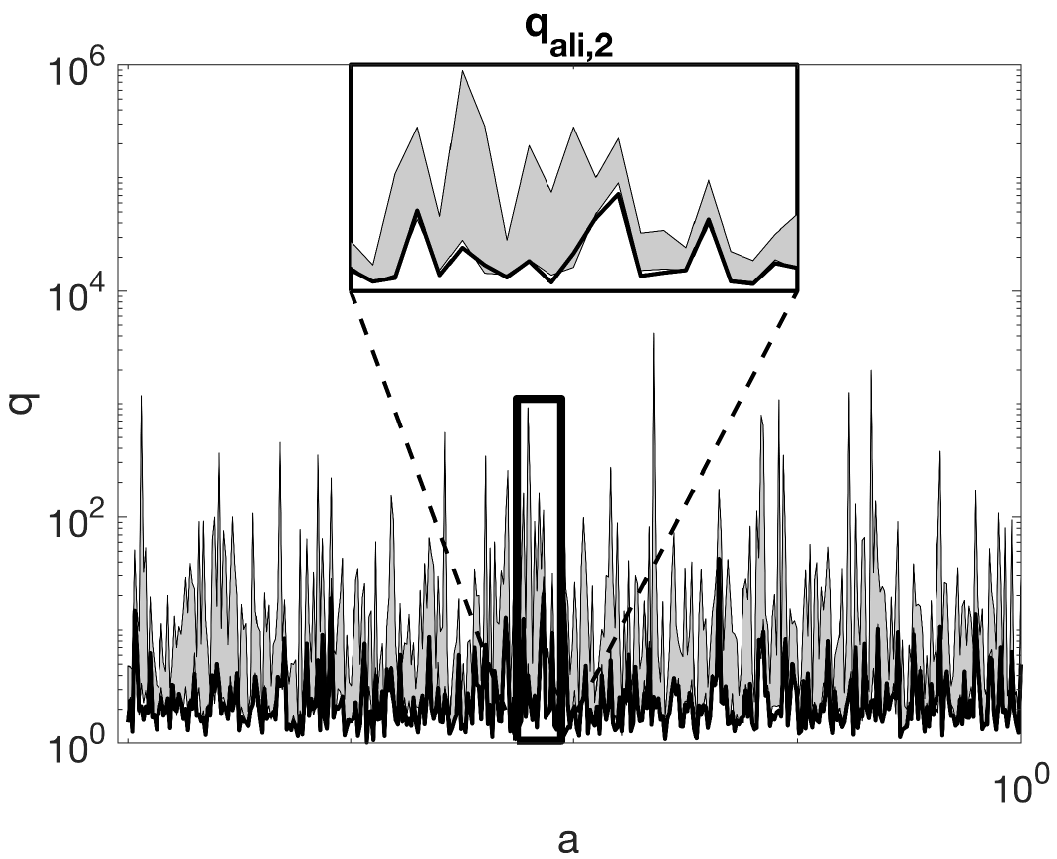} \; \includegraphics[width=4cm]{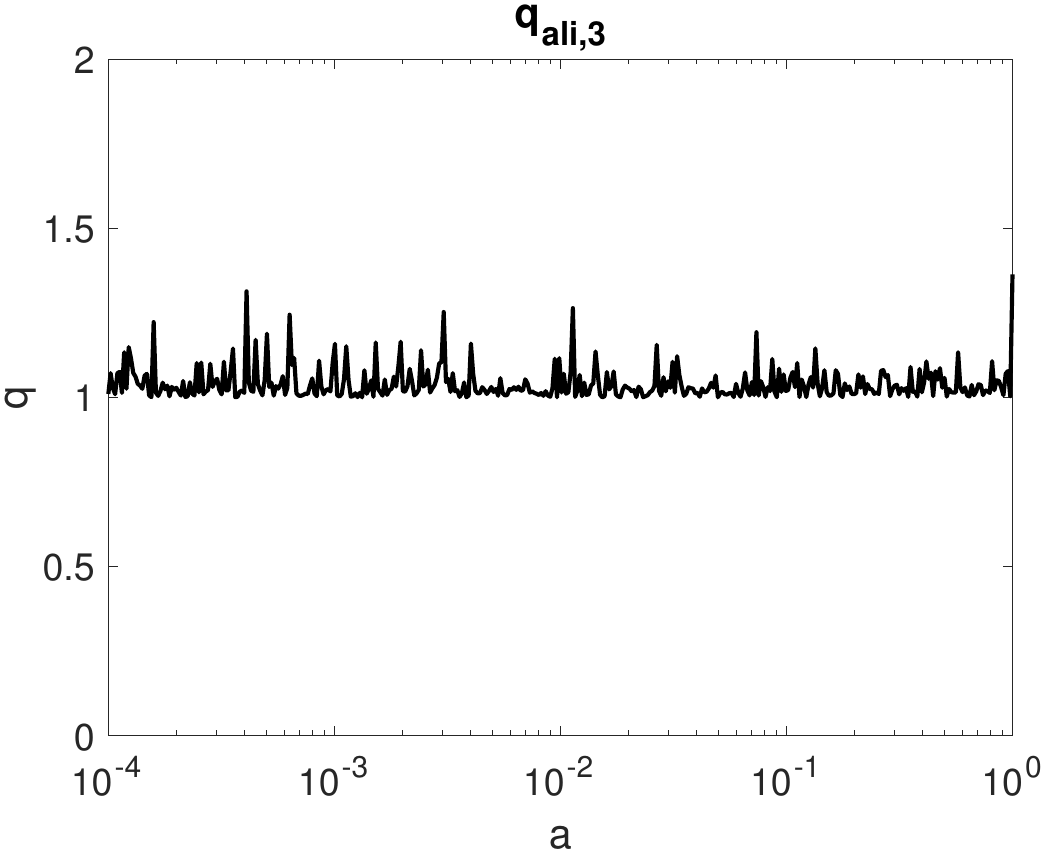}
\caption{Alignment measure vs. $a$ in the first test, computed with the metric in \eqref{eq:test1metric}. See the caption of Fig. \ref{fig:randAli} for more details.}
\label{fig:randMAli}
\end{center}
\end{figure}

\begin{figure}[ht]
\begin{center}
\includegraphics[width=4cm]{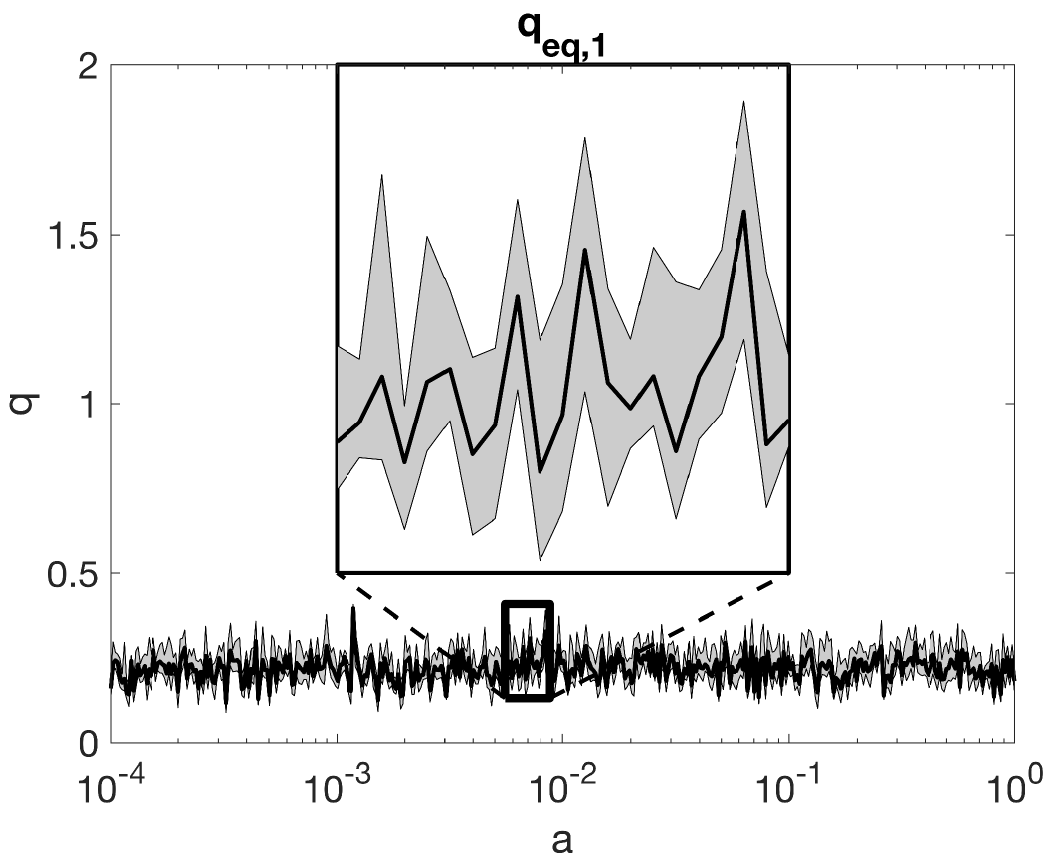} \; \includegraphics[width=4cm]{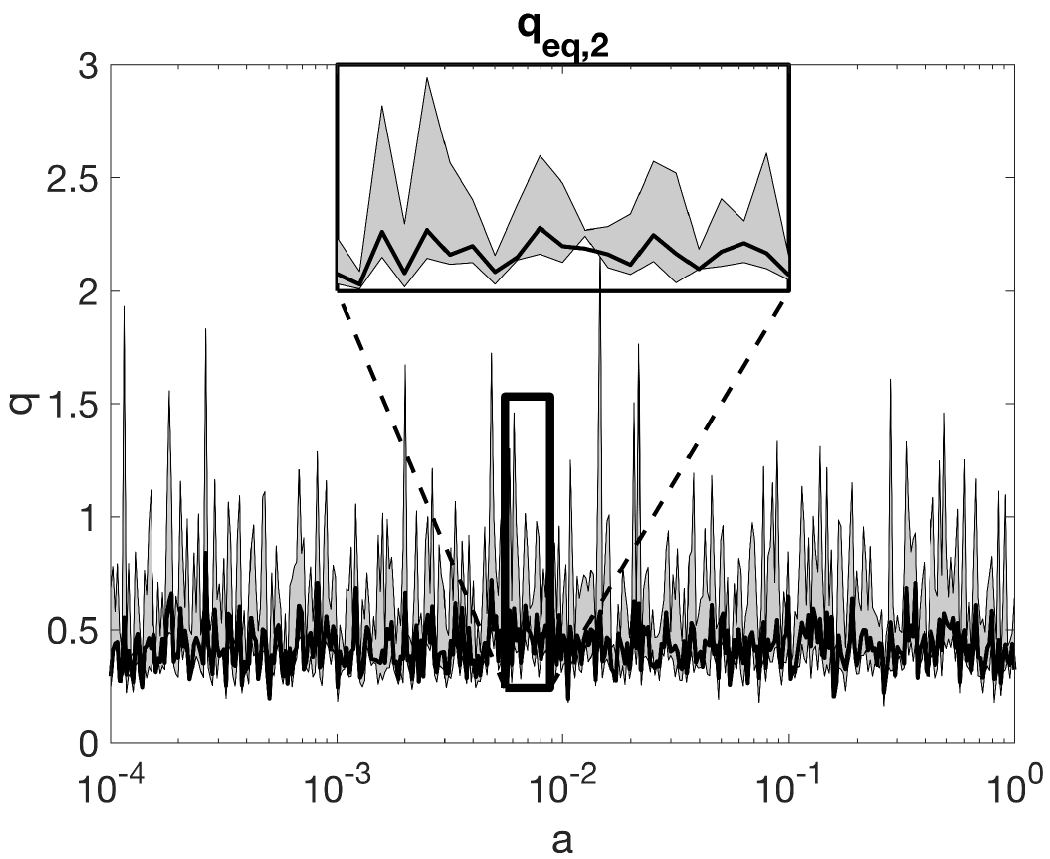} \; \includegraphics[width=4cm]{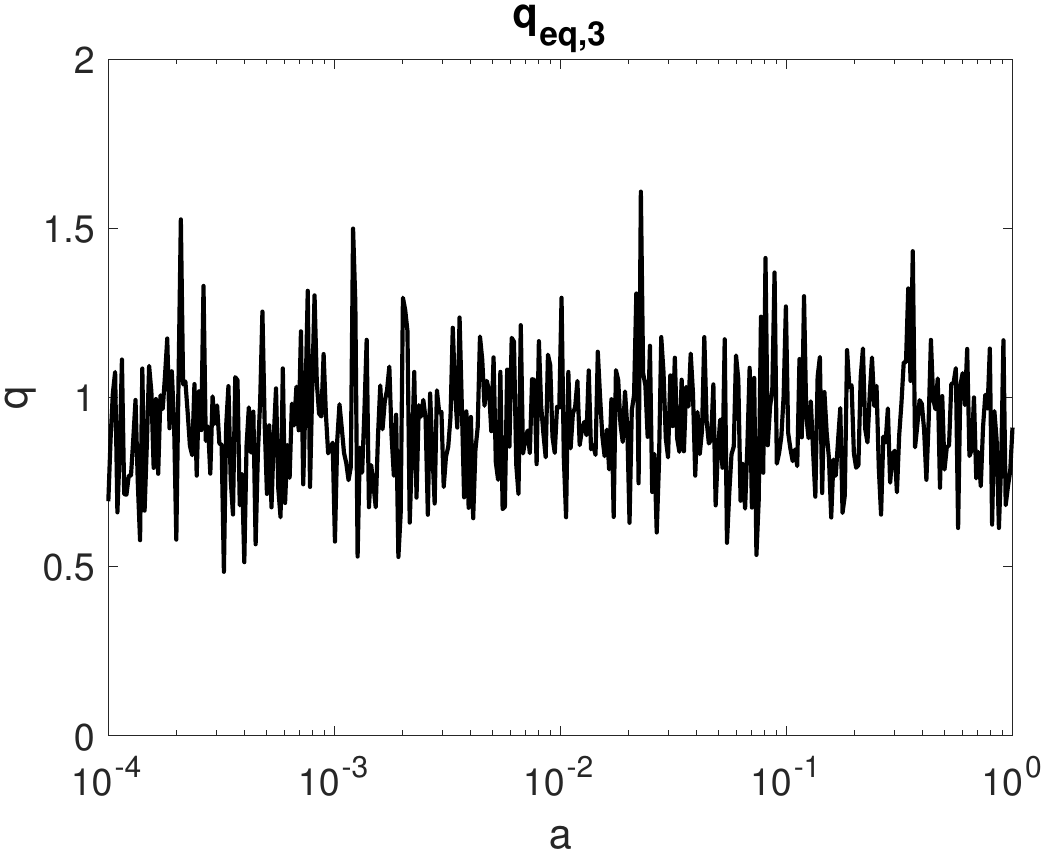}
\caption{Equidistribution measure vs. $a$ in the first test, computed with the metric in \eqref{eq:test1metric}. See the caption of Fig. \ref{fig:randAli} for more details.}
\label{fig:randMEq}
\end{center}
\end{figure}
 
 As expected, the quality measures under the Euclidean metric deteriorates as $a$ gets smaller, indicating that the quality of the polygons decreases. On the other hand,
the quality measures under the metric \eqref{eq:test1metric} appear to be independent of $a$. Moreover, we have the following interesting observations.
\begin{itemize}
\item Under the Euclidean metric, all three alignment measures have order $O(a^{-1})$ while all three equidistribution measures have order $O(a)$.
Here we emphasize that all the quality measures give reasonable {\em quantitative} measure of how ``bad" the polygon is.

\item The choice of anchor points has negligible impact on $q_{ali,1}$ and $q_{eq,1}$, whereas it does affect the values of $q_{ali,2}$ and $q_{eq,2}$ although the effect does not seem to alter the asymptotic order with respect to $a$. 

\item $q_{ali,2}$ and $q_{eq,2}$ computed with subdivision (b) have values comparable to the minimum value of $q_{ali,2}$ and $q_{eq,2}$ computed with subdivision (a),
i.e., the lower bound of the gray region in Figs. \ref{fig:randAli}-\ref{fig:randMEq}. Though they are not identical. 
Moreover, $q_{ali,2}$ and $q_{eq,2}$ computed with subdivision (b) have similar behavior as the other two sets of quality measures.

\item It is also interesting to notice that the dark curves, i.e., the first quality measure using the arithmetic center as the anchor point, the second using subdivision (b), and the third ones,
have relatively narrow numerical ranges for the test polygons. All three alignment measures in Fig. \ref{fig:randAli} are in approximately the range of $[1,10^{4}]$, 
while all three equidistribution measures in Fig. \ref{fig:randEq} in $[10^{-5},1]$. 
In Fig. \ref{fig:randMAli},  all three alignment measures lie approximately in the range of $[1,10]$, indicating good alignment property.
In Fig. \ref{fig:randMEq}, there is a small discrepancy among the values of equidistribution measures, i.e., $q_{eq,1}\approx 0.25$, $q_{eq,2}\approx 0.5$, and $q_{eq,3}\approx 1$.
This is because the random polygons are first generated in $[0,1]\times [0,1]$ and hence have original size of approximately $1\times 1$, 
while the unitary regular $n$-gon used as reference polygon for the first two sets of measures is defined using vertices $[\cos(2\pi i/n),\sin(2\pi i/n)]$ for $i=1,\ldots,n$
and hence has size of approximately $2\times 2$.
\end{itemize}

From the first test, we see that all quality measures, except for the second type using subdivision (a), provide {\em accurate and robust quantitative measures}.
The relatively less robust performance of $q_{ali,2}$ and $q_{eq,2}$ is probably related to the presence of short edges in the polygons. 
This may not be a bad thing in the sense that all other quality measures fail to capture the presence of short edges while the
second set of measures can. To further explore this, we design the second test,
\revZ{where the polygons are generated} as follows:
\begin{enumerate}
\item Randomly pick $n\in [3,10]$ and consider the unitary regular $n$-gon centered at the origin, with vertex $\vx_1$ located at $(1,0)$.
\item Insert a new vertex between $\vx_n$ and $\vx_1$ at $(1,-a)$. When $a$ is small, this generates an $(n+1)$-gon with the shortest edge length $a$. 
\end{enumerate}
We compute the three sets of mesh quality measures, all calculated under the Euclidean metric $\M_T=I$. 
In the test run, $200$ polygons are generated, with $n+1$ evenly distributed in $[4,11]$ and $a\in [10^{-5},10^{-1}]$.
The results are reported in Fig. \ref{fig:shortEdgeAli}-\ref{fig:shortEdgeEq}. 

\begin{figure}[ht]
\begin{center}
\includegraphics[width=4cm]{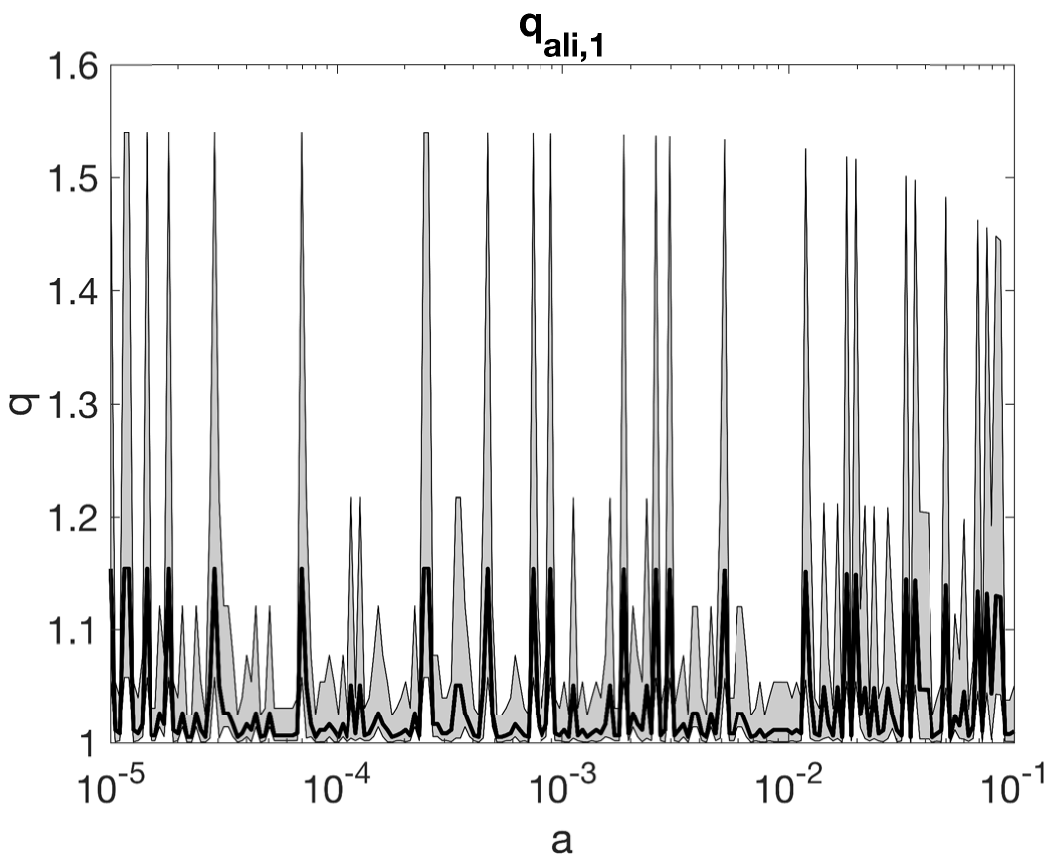} \; \includegraphics[width=4cm]{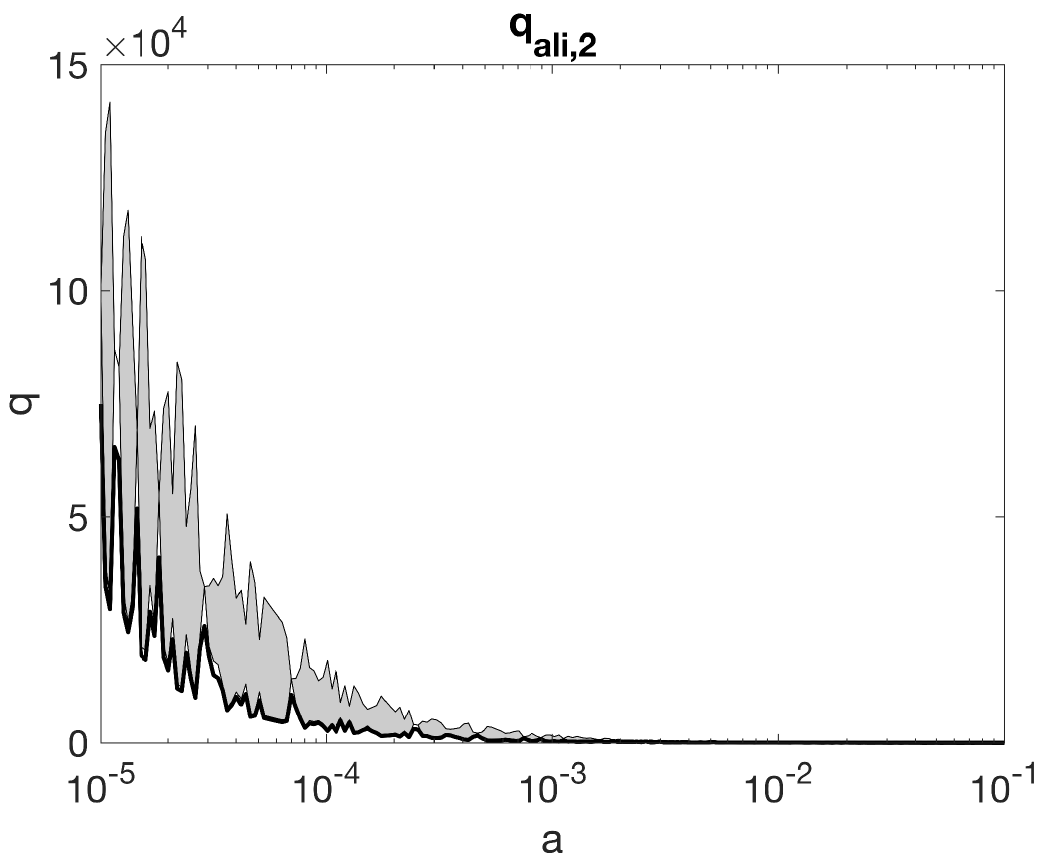} \; \includegraphics[width=4cm]{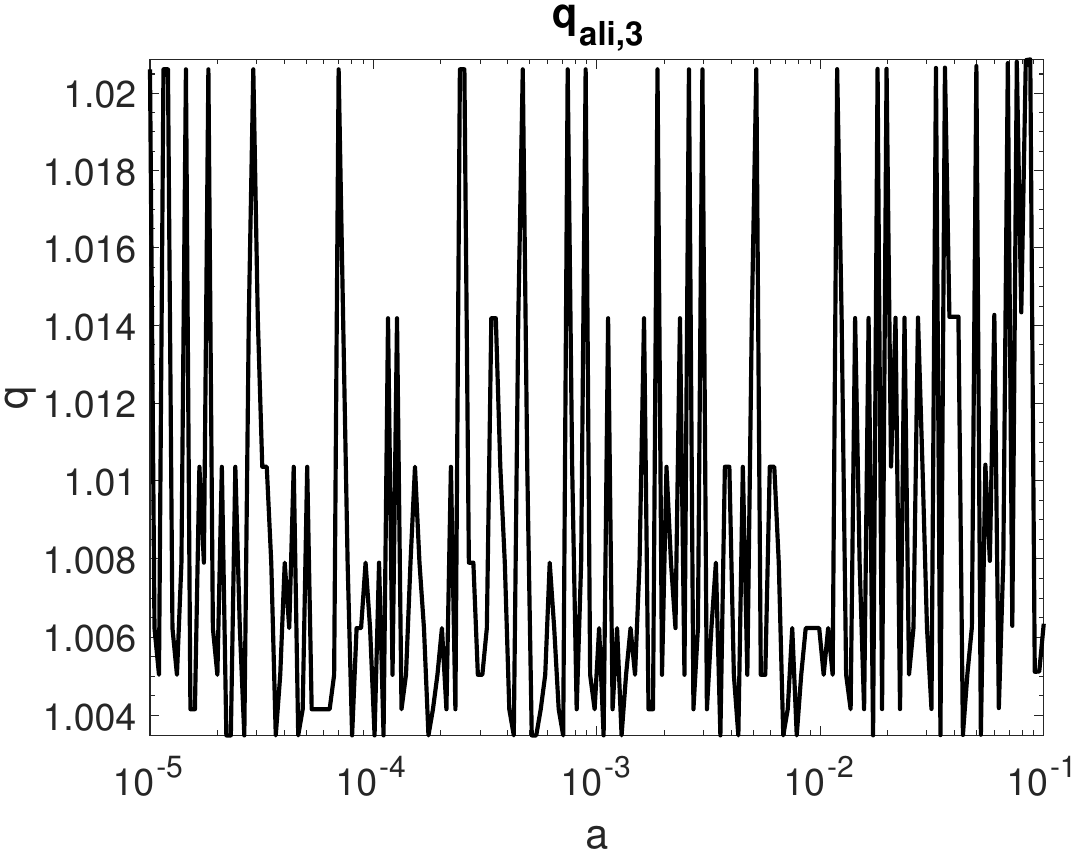}
\caption{Alignment measure vs. $a$ in the second test. See the caption of Fig. \ref{fig:randAli} for the meaning of shaded regions and dark curves.}
\label{fig:shortEdgeAli}
\end{center}
\end{figure}

\begin{figure}[ht]
\begin{center}
\includegraphics[width=4cm]{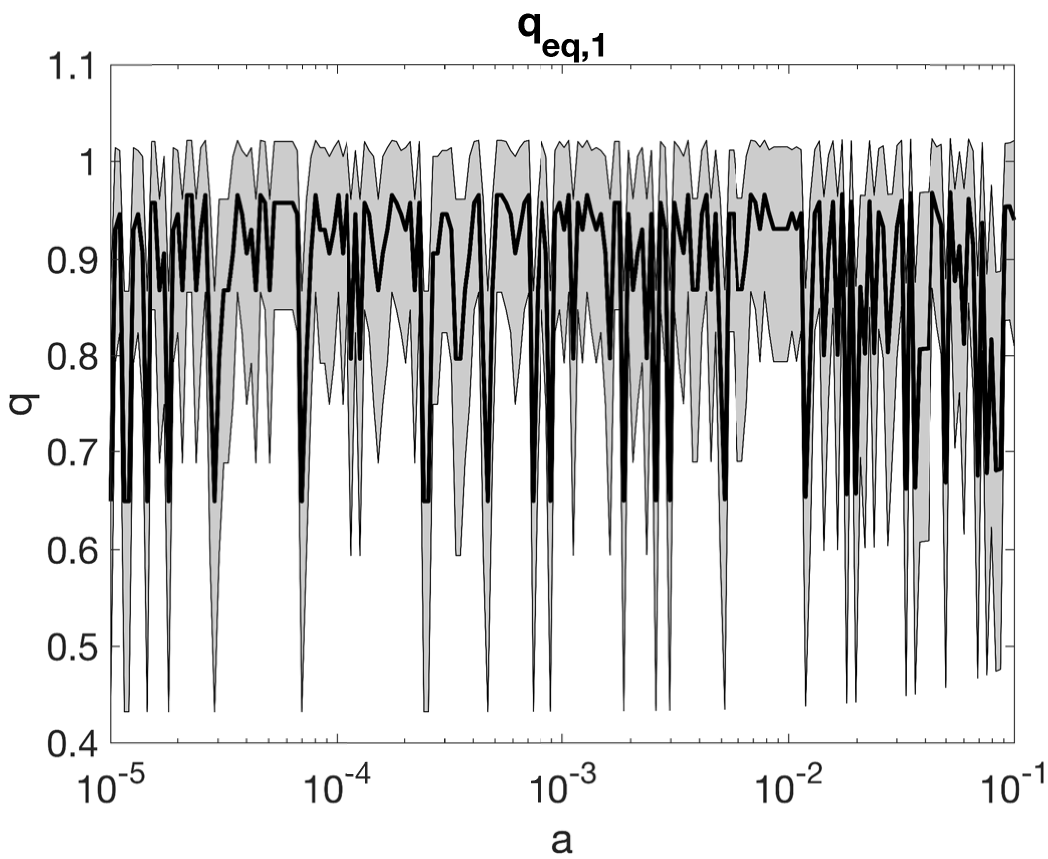} \; \includegraphics[width=4cm]{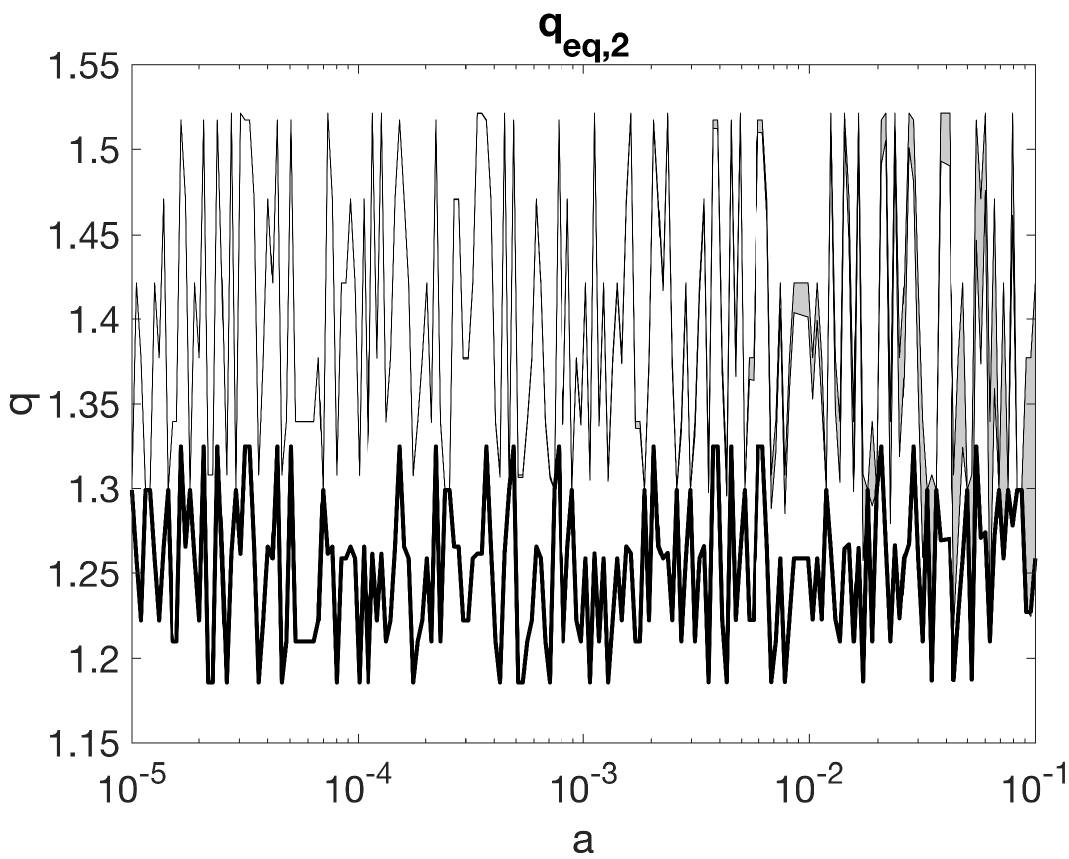} \; \includegraphics[width=4cm]{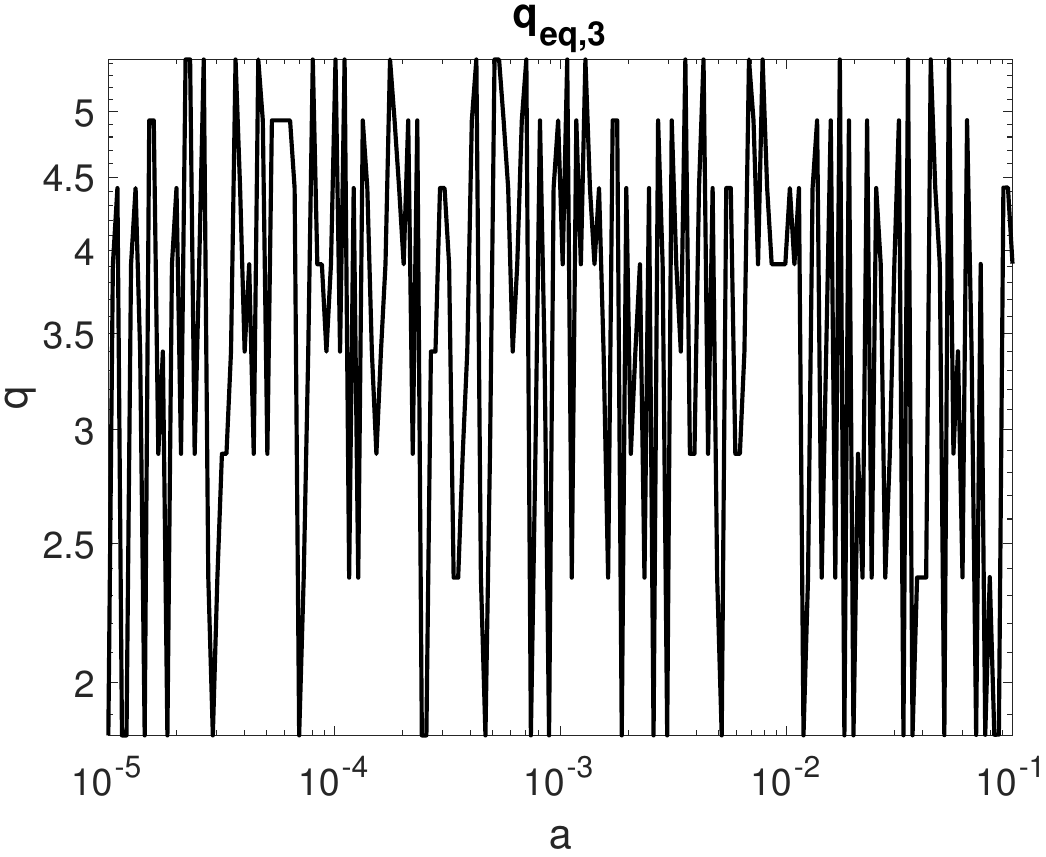}
\caption{Equidistribution measure vs. $a$ in the second test. See the caption of Fig. \ref{fig:randAli} for the meaning of shaded regions and dark curves.
\revAA{In the middle panel, the shaded region is thin and almost looks like a gray curve.}
}
\label{fig:shortEdgeEq}
\end{center}
\end{figure}

From Fig. \ref{fig:shortEdgeAli}-\ref{fig:shortEdgeEq}, we see that all quality measures except $q_{ali,2}$ appear to be (asymptotically) insensitive to the short edge. 
This also means that they cannot detect short edges.
The exception is $q_{ali,2}$, which is understandable since it is defined using the maximum value on all sub-triangles.
This can be a useful feature since the presence of short edges is considered harmful in many numerical methods.

A simple solution for now is try to avoid short edges in given polygonal meshes. 
One way is to combine the vertices into one if they are too close to each other.
For example, all CVT meshes used in this paper has been processed this way to ensure  that there \revZ{are} no short edges.
From the numerical results of the first and second tests, it is probably safe to claim 
that the effect of short edges on $q_{ali,2}$, with subdivison (b), is relatively small and negligible when the edge length
is controlled within around $10^{-2}\times$(diameter of polygon). 
This gives a practical threshold for the vertex merging process.
One may even raise the  threshold to $10^{-1}\times$(diameter of polygon) for better performance.
}

\subsection{Summary}
\label{sec:summary}

The three sets of mesh quality measures discussed in the previous subsections can be
viewed as discretizations or numerical approximations of the continuous alignment and equidistribution
quality measures defined in (\ref{eq:contGlobalMeasurements}). Moreover, they all adopt the idea of
evaluating the quality of a polygonal mesh by comparing its elements to their counterparts
in the computational mesh which can be a conventional mesh or a collection of reference polygons.
Numerical experiment shows that they
all provide correct measures for the quality of Voronoi meshes generated by Lloyd's algorithm.
\revi{They also provide accurate and robust quantitative measures on randomly generated polygons.
All measures except for $Q_{ali,2}$ appear to be insensitive to the presence of short edges.
}
Some of the main features of these mesh quality measures are elaborated in the following.

\begin{itemize}
\item $Q_{ali,3}$ (\ref{ali-3}) and $Q_{eq,3}$ (\ref{eq-3})
	are fully determined by the mesh $\T$. Specifically,
	the computational mesh $\T_C$ is a collection of $K_T$'s which are shown to have
	reasonably good quality and determined by
	the coordinates of the vertices of $T$. The mapping $\cF_T$ is an affine mapping
	from $K_T$ to $T$, which is also determined by the coordinates of the vertices of $T$.
\item $Q_{ali,1}$ (\ref{ali-1}) and $Q_{eq,1}$ (\ref{eq-1}) are not fully determined
	by the mesh $\T$. The computational mesh $\T_C$ can be chosen by the user
	to be a mesh or a collection of reference $n$-gons. The measurement of the quality of $\T$
	depends on the choice of $\T_C$. The mapping $\cF_T$ is approximated by
	an affine mapping that is obtained by least squares fitting to the correspondence
	between the vertices of $T\in \T$ and those of $T_C\in \T_C$ and thus fully determined by
	$\T$ and $\T_C$.
\item $Q_{ali,2}$ (\ref{ali-2}) and $Q_{eq,2}$ (\ref{eq-2}) are not fully determined
	by the mesh $\T$. The computational mesh $\T_C$ can be chosen by the user
	to be a mesh or a collection of reference $n$-gons. \revi{The mapping $\cF_T$ is specified using
	generalized barycentric mappings, for example, \revZ{the piecewise linear barycentric mapping}.}
	\revZ{This} measurement of the quality of $\T$ depends on the choice of $\T_C$
	as well as the choice of $\cF_T$. 
	\revi{
	Numerical tests show that $Q_{ali,2}$ is sensitive to the presence of short edges.
	But when the shortest edge length is controlled within around $10^{-2}\times$(diameter of polygon),
	$Q_{ali,2}$ and $Q_{eq,2}$, using subdivision (b), have similar performance \revZ{as} the other two \revZ{sets} of quality measures.
	}
\end{itemize}

\begin{table}[ht]
  \caption{$Q_{ali}$ and $Q_{eq}$ at selected Lloyd's iteration steps for the $32\times 32$ mesh.
  Subdivision (b) was used for $Q_{ali,2}$ and $Q_{eq,2}$.}
  \label{tab:LloydQC}
\begin{center}
  \begin{tabular}{|c||c|c||c|c||c|c|}
    \hline
    Iter. & $Q_{ali,1}$ & $Q_{eq,1}$ & $Q_{ali,2}$ & $Q_{eq,2}$ & $Q_{ali,3}$ & $Q_{eq,3}$ \\ \hline
    0    & 3.4362	 & 3.0130	 & 3.4362	 & 7.3202	 & 3.4362	 & 3.6641 \\ \hline
    2    & 1.8519	 & 2.5239	 & 2.9940	 & 3.3286	 & 1.8851	 & 2.6679 \\ \hline
    8    & 1.3927	 & 1.5355	 & 2.9190	 & 2.5647	 & 1.3869	 & 1.9022 \\ \hline
    43 & 1.1394	 & 1.3771	 & 2.7847	 & 2.1333	 & 1.1370	 & 1.4703 \\ \hline
  \end{tabular}
\end{center}
\end{table}

\revA{
Finally, we briefly discuss how to extend the three sets of mesh quality measures into three dimensions (3D).
For the first set, the extension itself is straightforward since the \revZ{least-squares} fittings work readily in 3D.
The difficulty lies in how to obtain reference polyhedra.
  In 3D, mappings between two polyhedra rely not only on their number of vertices, but also on their topological
  type, i.e., \revB{the relation of vertices, edges and faces.} This is different from the 2D case, where a regular $n$-gon
  serves as a good reference element for all convex $n$-gons. Constructing 3D ``regular'' reference polyhedra
  is not trivial in general. Nevertheless, the first set of mesh quality measures is still useful especially
  in cases where a reference mesh $\T_C$ with the same topological structure as $\T$ is available.
  An example is moving mesh algorithms where meshes are different only in location of their vertices.

  For the second set of mesh quality measures, in addition to well-defined reference polyhedra,
  one also needs well-defined generalized barycentric mappings.
  For the piecewise linear barycentric mapping, we know that it exists for convex polyhedra.
  \revC{Note that in a Voronoi diagram, all polyhedra are convex.}
  Some non-convex polyhedra, for example the Sch\"{o}nhardt polyhedron, may not have a simplicial subdivision.
  Hence the second set of mesh quality measures cannot be extended to such polyhedra.

  The third set of mesh quality measures can be extended to 3D similarly using the singular value decomposition
  of the matrix $B_T$. Moreover, one does not need to worry about the reference polyhedron since it is defined
  automatically by $P_T$.
  
  Finally, the above discussion only points out that an extension to 3D is possible. 
  The actual extension is highly non-trivial and many details remain to be explored in the future.
}

\section{Anisotropic polygonal mesh adaptation}
\label{sec:MMPDE}

In this section, we study the anisotropic adaptation for polygonal meshes
through a moving mesh method
based on the MMPDE (moving mesh PDE) \cite{HRR94a,HR11}.
Notice that a mesh is completely determined by two data structures:
the coordinates and connectivity of vertices that form polygons.
In a certain range, one can move the vertices without changing
mesh topology or tangling the mesh.
The moving polygonal mesh method uses this idea and implements it in an iterative manner.
It should be pointed out that the MMPDE method is only a method for generating adaptive,
anisotropic polygonal meshes. Other methods, such as those based on refinement, \revB{can also be used}.

The procedure for the moving polygonal mesh method \revi{for solving a given PDE problem is presented below.}
\begin{enumerate}
\item Initialization: Given an initial physical mesh $\T^{(0)}$ for $\Omega$;
\item Outer iteration ($k=0, 1, ...$):
  \begin{description}
	\item{\revA{2a}} Update the metric tensor $\M^{(k)}$ based on the information
		available at the current iteration.
		The information includes the current mesh $\T^{(k)}$ and the physical solution $u^{(k)}$
		that is obtained by solving the underlying PDE on the current mesh $\T^{(k)}$.
		\revi{For example, the metric tensor $\M^{(k)}$ can be computed using various formulae involving the reconstructed Hessian of $u^{(k)}$; e.g., see \cite{HS03}.}
	\item{\revA{2b}} Find a way to move the vertices of the physical mesh so that the new mesh $\T^{(k+1)}$
		has a better quality under the metric $\M^{(k)}$,
		\revi{which will be further explained below.}
		\label{step-2b}
	\end{description}
\end{enumerate}

\revi{
We first briefly describe the main idea behind Step~2b.
When starting Step~2b, the metric tensor $\M^{(k)}$ is given and we assume that the mesh topology will not change, i.e., $\T^{(k+1)}$ has
the same topological structure as $\T^{(k)}$. Therefore, to compute $\T^{(k+1)}$ is the same as to compute the coordinates of vertices for $\T^{(k+1)}$.
Note that if a reference computational mesh is given, the mesh quality measures of $\T^{(k+1)}$, as defined in Section \ref{sec:qualitymeasures}, depend
solely on the Jacobian matrix and the metric tensor, which in turn depend solely on the coordinates of vertices of $\T^{(k+1)}$. 
This turns the problem of finding $\T^{(k+1)}$ into an optimization problem: 
finding the coordinates of vertices in order to minimize the alignment measure, the equidistribution measure, or a combination of both.
Standard techniques for solving the optimization problem such as the gradient descent method or the gradient flow approach can then be applied.

In this work, the moving mesh strategy \revZ{is} based on $Q_{ali,2}$ and $Q_{eq,2}$
(with piecewise \revZ{linear} barycentric mappings and subdivision (b)).
There are several advantages of using the second set of mesh quality measures.
First of all, $Q_{ali,2}$ and $Q_{eq,2}$ are associated with a triangulation of each
polygon in the mesh and the union of those triangulations forms a triangular mesh itself.
This allows us to reuse part of the code previously developed in \cite{HK2014}
for moving triangular meshes. \revB{As an additional benefit, the moving mesh
method of \cite{HK2014} has been shown both theoretically and numerically \cite{HK2017}
to maintain the nonsingularity of the triangular mesh during the mesh movement,
i.e., there is no mesh tangling as long as the initial mesh is not tangled.
Thus we expect that the polygonal moving mesh method based on the sub-triangulation can avoid mesh tangling
as long as it has a nonsingular sub-triangulation, which is true if the mesh elements remain convex.}
Moreover, from Table~\ref{tab:LloydQC} one can see that $Q_{ali,2}$ and $Q_{eq,2}$
have \revZ{larger} values than the first and third sets of quality measures. In this sense,
they can be viewed as the toughest measure among the three. In addition,
$Q_{ali,2}$ is sensitive to the presence of short edges. Thus, algorithms based on
(and minimizing) $Q_{ali,2}$ and $Q_{eq,2}$ more likely produce meshes of better quality and
particularly avoid short edges than algorithms based on the other sets of quality measures.

We should emphasize that in principle, the same procedure can also be applied \revZ{to} the first and third sets
of mesh quality measures.
However, since these two sets of measures deal directly with polygonal
meshes, there are at least two things that cannot be achieved as for $Q_{ali,2}$ and $Q_{eq,2}$
that are associated with triangular meshes. The first is that there is no theoretical guarantee at least so far
that the minimization process can avoid mesh tangling \revDD{even if} the initial mesh is nonsingular and the mesh elements remain convex.
The other is that \revZ{an} effective minimization process typically requires the use of analytic gradients of
a meshing function that can be cumbersome to get for general polygonal meshes.
Analytic formulae for the gradients of meshing functions have been derived 
in compact matrix form for triangular meshes in \cite{HK2017}.
One may argue that numerical approximations of the gradients can be used.
However, our limited numerical experience shows that mesh adaptation can lead to
very small mesh spacing and numerical gradients can lead to inaccurate location of vertices
and thus \revZ{the resulting algorithms are less reliable}.
}

\revi{
The quality measures $Q_{ali,2}$ and $Q_{eq,2}$ are defined by taking maximum values, which do not make suitable objective functions
for an optimization problem. We take a variational approach that is commonly used in the moving mesh community.}
Assume that a reference computational mesh $\hat{\T}_C$ has been chosen
for the mesh movement purpose \revDD{such that $\hat{\T}_C$ has the same topology as $\T^{(k+1)}$
and there is a one-to-one correspondence between mesh elements in these two meshes.} 
We will discuss how to set up $\hat{\T}_C$ later.
Denote the triangular meshes resulting from the triangulation associated
with $Q_{ali,2}$ and $Q_{eq,2}$ for $\hat{\T}_C$ and $\T^{(k+1)}$
by $\T_{\hat{\T}_C}$ and $\T_{\T^{(k+1)}}$, respectively. 
\revDD{Since $\hat{\T}_C$ and $\T^{(k+1)}$ are assumed to have the same topology
and subdivision (b) (which triangulates each polygonal element
into triangular elements by connecting the arithmetic center (the anchor point) to all vertices of the element)
is used for elements of both $\hat{\T}_C$ and $\T^{(k+1)}$,
$\T_{\hat{\T}_C}$ and $\T_{\T^{(k+1)}}$ have the same topology and there is also a one-to-one correspondence
between their triangular elements.}
Consider the function
\begin{equation}
\label{Ih-1}
\revDD{\hat{I}_h( \{\hat{\vxi}_i\}, \{ \vx_i^{(k+1)}\})  = \sum_{K \in \T_{\T^{(k+1)}}} |K| G(\hat{J}_K, \det(\hat{J}_K),\hat{\M}_K) = \sum_{T \in \T^{(k+1)}} \sum_{K\in \T_T} |K| G(\hat{J}_K, \det(\hat{J}_K),\hat{\M}_K),}
\end{equation}
where $\{\hat{\vxi}_i\}$ and $\{ \vx_i^{(k+1)}\}$ denote the coordinates of the vertices of
$\T_{\hat{\T}_C}$ and $\T_{\T^{(k+1)}}$, respectively,
\revi{$\hat{\M}_K$} is the average of $\M^{(k)}$ on $K$,
\revi{$\hat{J}_K$} is the inverse of the Jacobian
matrix $\J_K$ of the affine mapping from $K_C \in \T_{\hat{\T}_C}$ to $K \in \T_{\T^{(k+1)}}$,  and
\revi{
\begin{align}
G(\hat{J}_K, \det(\hat{J}_K), \hat{\M}_K)  & =
\frac{1}{3} \sqrt{\det(\hat{\M}_K)} \left( \text{trace}(\hat{J}_K \hat{\M}_K^{-1} \hat{J}_K^T) \right)^2
 + \frac{4}{3}  \sqrt{\det(\hat{\M}_K)} \left( \frac{\det(\hat{J}_K)}{\sqrt{\det(\hat{\M}_K)}} \right)^2 .
\label{G}
\end{align}
}
The function (\ref{Ih-1}) is a \revi{Riemann sum} of a continuous meshing functional \cite{Hua01b,HK2014}
and it is known that minimizing \revi{$\hat{I}_h$}
will tend to make the mesh to satisfy the alignment and equidistribution conditions
associated with $\M^{(k)}$
and thus to have a better quality under the metric $\M^{(k)}$.

\revi{
The reference mesh $\hat{\T}_C$ is assumed to have the same topological structure as $\T^{(k+1)}$. 
Recall that the moving mesh iteration does not alter the topological structure of the mesh, i.e., 
$\T^{(0)},\,\T^{(1)},\, \ldots, \, \T^{(k+1)}$ all have the same topological structure.
One can simply choose $\T^{(0)}$ as a CVT and then take $\hat{\T}_C = \T^{(0)}$.
Now, since $\T_{\hat{\T}_C}$ is known, $\hat{I}_h$ is a function of $\{ \vx_i^{(k+1)}\}$.
Thus, $\{ \vx_i^{(k+1)}\}$ or $\T_{\T^{(k+1)}}$ can be obtained by minimizing $\hat{I}_h$.
Solving the minimization problem usually involves an iterative method because $\hat{I}_h$ is highly nonlinear.
This forms the inner iteration.
}

\revi{
Now we have a complete and implementable Step~2b. In the following we shall elaborate the inner iteration.}
Recall that $\M^{(k)}$ is defined on $\T^{(k)}$. During the inner iteration, the metric tensor
needs to be updated constantly (through interpolation) on approximate meshes of $\T^{(k+1)}$
that will have different vertex locations than $\T^{(k)}$.
This can be expensive even if linear interpolation is used.

\revDD{
To avoid this difficulty, we use an indirect approach to find $\T^{(k+1)}$.
Since the computational mesh $\hat{\T}_C$ is a conventional mesh,
we denote by $\Omega_C$ the domain that $\hat{\T}_C$ occupies.
We notice that $\hat{I}_h$ defined in (\ref{Ih-1}) is a function of the triangulation $\T_{\hat{\T}_C}$
for $\Omega_C$ and the triangulation $\T_{\T^{(k+1)}}$ for $\Omega$.
If we replace $\T_{\T^{(k+1)}}$ by $\T_{\T^{(k)}}$ and $\T_{\hat{\T}_C}$ by $\T_{\T_C}$ in $\hat{I}_h$,
where $\T_C$ is a polygonal mesh for $\Omega_C$ with the same topology
as $\T^{(k)}$, we obtain a function
\begin{equation}
\label{lh-1a}
I_h ( \{\veta_i\}, \{ \vx_i^{(k)}\}) = \sum_{K \in \T_{\T^{(k)}}} |K| G(J_K, \det(J_K), \M_K) = \sum_{T \in \T^{(k)}} \sum_{K\in\T_T} |K| G(J_K, \det(J_K), \M_K), 
\end{equation}
where $\{\veta_i\}$ are the vertices of $\T_{\T_C}$, $J_K$ is the inverse of the Jacobian matrix $\J_K$ of the affine mapping from $K_C \in \T_{\T_C}$
to $K \in \T_{\T^{(k)}}$, $\M_K$ is the average of $\M^{(k)}$ on $K$, and function $G$ is defined as in \eqref{G}.
If we minimize $I_h$ with respect to $\{\veta_i\}$ while fixing $\{ \vx_i^{(k)}\}$ (and fixing the mesh topology),
we can obtain an optimal solution
$\{\veta_i\}$ or $\T_{\T_C}$. A piecewise linear mapping $\cF_h$ from $\Omega_C$ to $\Omega$, satisfying
\begin{equation}
\{ \vx_i^{(k)}\} = \cF_h (\{\veta_i\}) ,
\label{linear-mapping}
\end{equation}
can be defined using the affine mappings between the elements in $\T_{\T_C}$ and their counterparts
in $\T_{\T^{(k)}}$.
We can then define the new mesh $\T_{\T^{(k+1)}}$ (and then $\T^{(k+1)}$) by setting
\[
\{ \vx_i^{(k+1)}\} = \cF_h (\{\hat{\vxi}_i\}) ,
\]
which can be computed through (\ref{linear-mapping}) by linear interpolation.

To justify this indirect approach, we recall that $\hat{I}_h$ defined in (\ref{Ih-1}) is a Riemann sum
of a continuous meshing functional \cite{Hua01b,HK2014}. Minimizing (\ref{Ih-1}) with respect
to $\{ \vx_i^{(k+1)}\}$ for fixed $\{\hat{\vxi}_i\}$ leads to a new physical mesh $\T_{\T^{(k+1)}}$
which, together with $\T_{\hat{\T}_C}$, defines a piecewise linear mapping (denoted by $\tilde{\cF}_h$)
from $\Omega_C$ to $\Omega$. It is reasonable to expect that $\tilde{\cF}_h$ is an approximation
to the minimizer of the continuous meshing functional when the mesh is sufficiently fine.
On the other hand, $I_h$ defined in (\ref{lh-1a}) is also a Riemann sum of the continuous
meshing functional. Similarly, we can expect that $\cF_h$ is also an approximation to
the minimizer of the continuous functional. As a result, we can expect that $\cF_h$ is close to $\tilde{\cF}_h$
when the mesh is sufficiently fine.
}

\revi{We choose the gradient flow approach for minimizing $I_h( \{\veta_i\}, \{ \vx_i^{(k)}\})$.}
\revD{The mesh equation is presented below without derivation.}
The interested reader is referred to \cite{HK2014} for detailed derivation.
\begin{equation}
   \begin{cases}
   \frac{d \revi{\V{\eta}_i}}{d t}
      = \frac{P_i}{\tau} \sum\limits_{K \in \omega_i} |K| \V{v}_{i_K}^K,
      & \quad i = 1, \dotsc, N_v, \quad t > 0,
   \\
   \revi{\veta_i}(0) = \hat{\V{\xi}}_i,
   & \quad i = 1, \dotsc, N_v,
   \end{cases}
   \label{mmpde-1}
\end{equation}
\revC{where $N_v$ is the number of vertices of the sub-triangulation $\T_{\T^{(k)}}$},
$\omega_i$ is the patch of triangles associated with vertex $\vx_i^{(k)}$ in $\T_{\T^{(k)}}$,
$i_K$ is the local index of $\vx_i^{(k)}$ in $K$, $\V{v}_{i_K}^K$ is the local mesh velocity associated
with the $i_K^{\text{th}}$ vertex of $K$,
$\tau > 0$ is a constant parameter used to adjust
the time scale of mesh movement, and $P= (P_1, \dotsc, P_{N_v})$ is a positive function used to
make the mesh equation to have desired invariance properties.
The local velocities are given by
\begin{equation}
   \begin{bmatrix} (\V{v}_1^K)^t\\  (\V{v}_2^K)^t \end{bmatrix}
      = - E_K^{-1} \frac{\partial G}{\partial J_K}
         - \frac{\partial G}{\partial \det(J_K)}
            \frac{\det(E_{K_c})}{\det(E_K)} E_{K_c}^{-1},
   \qquad
   \V{v}_0^K =  - \sum_{j=1}^2 \V{v}_j^K ,
   \label{mmpde-2}
\end{equation}
where
\begin{align*}
\frac{\partial G}{\partial J_K}  & =  \frac{4}{3} \sqrt{\det(\M_K)}  \,
      {\text{trace}(J_K \M_K^{-1} J_K^T )} \;  \M_K^{-1} J_K^T,
\\
\frac{\partial G}{\partial \det(J_K)} & = \frac{8}{3}  \frac{\det{(J_K)}}{\sqrt{\det(\M_K)}} .
\end{align*}
The balancing function in (\ref{mmpde-1}) is chosen to be $P_i = \det(\M(\vx_i^{(k)}))^{\frac{1}{2}}$
such that (\ref{mmpde-1}) is invariant under the scaling transformation $\M \to c \M$.

The mesh equation (\ref{mmpde-1}) should be modified properly for boundary vertices.
For fixed boundary vertices, the corresponding equations should be replaced by
\[
\frac{d \revi{\V{\eta}_i}}{d t}  = 0.
\]
For boundary vertices on a curve represented by $ \phi(\revi{\V{\eta}}) = 0$,
the corresponding mesh equations should be modified such that its normal component along
the curve is zero, i.e.,
\[
   \nabla \phi (\revi{\V{\eta}_i}) \cdot \frac{d \revi{\V{\eta}_i}}{d t} = 0.
\]

The mesh equation (\ref{mmpde-1}), along with proper modification for boundary vertices,
can be solved by any ODE solver. (Matlab's ODE solver ode15s is used in our computation.)
In principle, it should be integrated until a steady state is obtained. Since finding $\revi{\veta_i}$
just represents one step of the outer iteration, we integrate (\ref{mmpde-1}) only up to
$t = 1$ (with $\tau = 1/300$) to save CPU time.

We now discuss the choice and computation of the metric tensor.
We choose the metric tensor based on optimizing the $L^2$ norm of error for piecewise
linear interpolation \cite{Hua05b,HS03}. The main reason for this choice is that it is simple,
problem independent, and effective. It reads as
\begin{equation}
\M = {\det \left(\alpha_h I +  |H(u_h)| \right)}^{- \frac{1}{6}}
\left [ \alpha_h I  +  |H(u_h)| \right ] ,
\label{M-1}
\end{equation}
where $u_h$ is an approximate solution, $H(u_h)$ is the recovered Hessian of $u_h$, $|H(u_h)|$
is the eigen-decomposition of $H(u_h)$ with the eigenvalues being replaced by their absolute values,
and the regularization parameter $\alpha_h > 0$ is chosen such that
\begin{equation}
\label{alpha-1}
\int_\Omega \sqrt{\det(\M)} d \V{x}
= 2 \int_\Omega {\det\left( |H(u_h)|\right)}^{\frac{1}{3}} d \V{x} .
\end{equation}
It has been shown in \cite{KaHu2013} that when the recovered Hessian
satisfies a closeness assumption, a linear finite element solution of an elliptic boundary value problem on
a simplicial mesh computed using the moving mesh algorithm converges at a second order rate
as the mesh is refined.
Numerical examples presented later show that the same strategy seems to work well for polygonal meshes too.
We use a Hessian recovery method based
on a least squares fit. More specifically, a quadratic polynomial is constructed locally for each vertex
via least squares fitting to neighboring nodal function values and an approximate Hessian at the vertex
is then obtained by differentiating the polynomial.

\revDD{
It is worth pointing out that, while the current polygonal moving mesh algorithm is essentially the (triangular)
moving mesh algorithm \cite{HK2014} applied to a sub-triangulation of the polygonal mesh,
there is a difference in the computation of the metric tensor in the current algorithm
(referred to as the polygonal MMPDE algorithm) and a purely triangular moving mesh
method (referred to as a triangular MMPDE algorithm). In the latter case, the metric tensor
is commonly computed as a piecewise constant function
on the triangular elements. On the other hand, in each MMPDE outer iteration of the current polygonal MMPDE
algorithm, the metric is computed as a piecewise constant function on the polygonal elements. More specifically,
the metric is first computed at all vertices of the polygonal mesh $\T^{(k)}$.
Then, $\M^{(k)}|_T$ for each polygonal element $T\in \T^{(k)}$ is defined by taking average
of the metric at all vertices of $T$, and $\M_K$ is set to be equal to $\M^{(k)}|_T$
for any sub-triangle $K$ in the triangulation of $T$. In this way, each polygonal element
is considered as a whole. On the contrary, the metric tensor is generally discontinuous
on each polygonal element if it is computed as piecewise constant on sub-triangular elements.
Therefore, while these two approaches lead to similar meshes especially when the mesh is sufficiently fine,
they can result in polygonal meshes with different properties (as to be shown below).


Another difference between polygonal and triangular MMPDE algorithms lies in
the choice of the reference element(s). A triangular MMPDE algorithm typically uses
a single reference element $\hat{K}$. Then, by minimization, the algorithm tempts to make every physical element $K$
(measured in the metric $\M_K$) as similar to $\hat{K}$ as possible. Thus, the shape of each individual element
basically is independent of the shape of other elements.
On the other hand, for any polygon $T \in \T^{(k+1)}$, the polygonal MMPDE algorithm described above uses
the triangles in $\T_{\hat{T}_C}$ as the reference elements for their counterparts in $\T_{T}$,
and by minimization, the algorithm tempts to make each triangle (measured in the metric $\M_T$) in $\T_{T}$
as similar to its corresponding triangle in $\T_{\hat{T}_C}$ as possible. Since $\M_T$ is constant on $T$, this means
that $T$ is made to be similar to $\hat{T}_C$ as possible. In this sense, the polygonal MMPDE algorithm is actually
polygon-oriented.


The polygonal MMPDE algorithm inherits the non-tangling feature of the triangular algorithm.
By design, the inner iteration of the current algorithm is the same as a triangular MMPDE algorithm
on the sub-triangulation.  According to \cite{HK2017}, the triangular mesh generated by the MMPDE algorithm
during the inner iteration, $\T_{\T^{(k+1)}}$, will not tangle as long as $\T_{\T^{(k)}}$
does not tangle. Recall that the non-tangling feature of $\T_{\T^{(k+1)}}$ implies that none
of its edges crosses other edges at an interior point. Since the edges of $\T^{(k+1)}$ are also
edges of $\T_{\T^{(k+1)}}$, we know that none of the edges of $\T^{(k+1)}$ crosses any other edges
at an interior point and thus $\T^{(k+1)}$ does not tangle.
Recalling that the sub-triangulation $\T_{\T^{(k)}}$ is obtained from $\T^{(k)}$
via subdivision (b) (which triangulates each polygonal element
into triangular elements by connecting the arithmetic center (the anchor point) to all vertices of the element),
$\T_{\T^{(k)}}$ is guaranteed not to tangle when all of the polygonal elements of $\T^{(k)}$ are convex.
This, on one hand, shows the importance to maintain the convexity of the polygonal elements.
On the other hand, we should emphasize that the polygonal MMPDE algorithm works
as long as the sub-triangulation of each polygonal element (that is not necessarily convex)
does not tangle. Subdivision (b) works when polygons are not very far from being convex.
We can also use more sophisticated polygon partitioning algorithms that work for general polygons;
e.g., see \cite{deBerg2000}. This will be an interesting research topic for the near future.

We now explain the difference between the polygonal and triangular MMPDE algorithms by a numerical example.
For simplicity, assume that there is a continuous metric tensor $\M$ given in $\Omega$,
so that in each outer iteration the discrete metric tensor is computed by interpolating $\M$ instead of solving a PDE.
In Fig. \ref{fig:convexity}, we plot the mesh for the polygonal MMPDE algorithm
after 10 MMPDE outer iterations and meshes for the triangular MMPDE algorithm after 2 and 10 outer iterations.
Notice that for the latter case we only plot the polygonal counterpart instead of the sub-triangulation
in order to compare with the former case.
After 2 outer iterations, the triangular MMPDE starts to generate non-convex polygons,
shaded in gray in Fig. \ref{fig:convexity}.
The number of non-convex polygons increases as the iteration goes on.
On the other hand, the polygonal MMPDE does not generate any non-convex polygons for this example.
Although only the mesh after 10 iterations is plotted here, we have tested up to 100 iterations and
all polygons have remained convex.
Indeed, the mesh movement has become very small after about 5 iterations,
indicating a stationary solution has been reached.

The difference shown in Fig. \ref{fig:convexity} is understandable, and even expected. 
The triangular MMPDE moves individual triangles with their own metric,
while the polygonal MMPDE tends to move each individual polygon as a whole.
Because of this, the polygons in a polygonal MMPDE appear to be more `rigid'.
In Fig. \ref{fig:convexity}, we observe that the triangular MMPDE is slightly more flexible
as it generates thinner (better aligned) elements while the polygonal MMPDE leads to polygons
of convex shape. 

\begin{figure}[ht]
\begin{center}
\includegraphics[width=3cm]{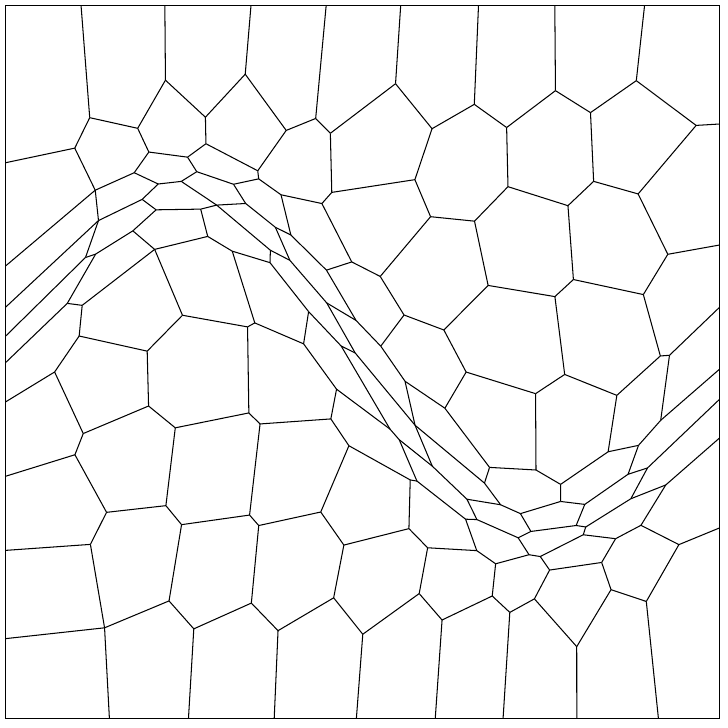}\quad
\includegraphics[width=3cm]{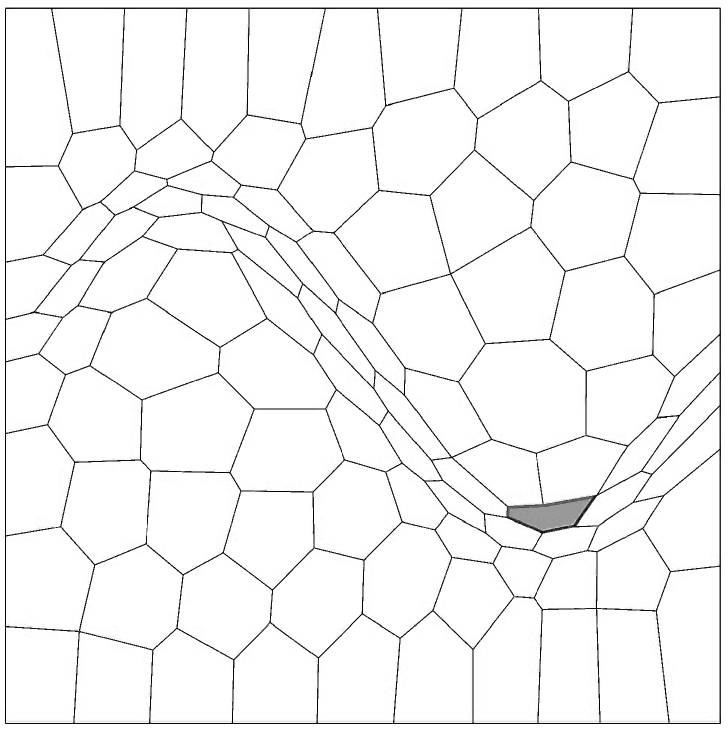}
\includegraphics[width=3cm]{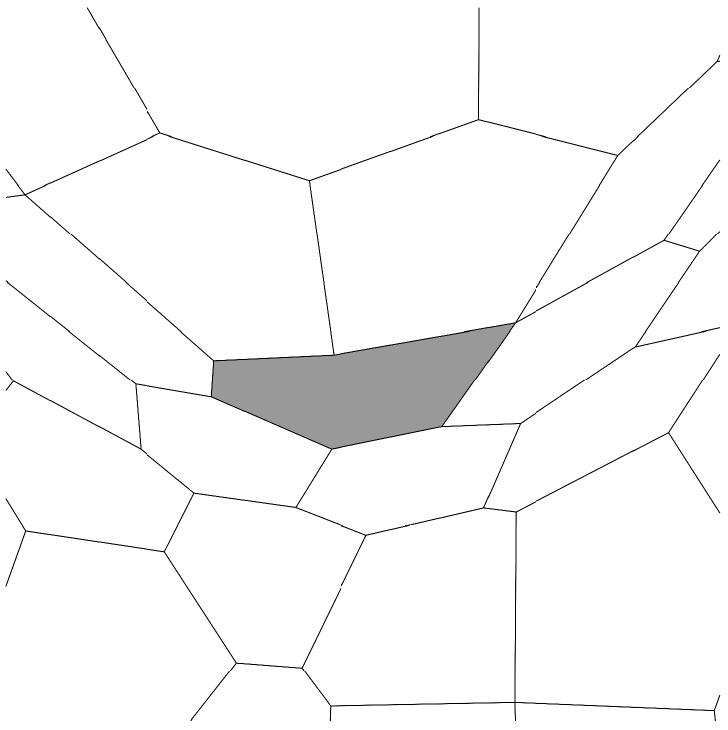}\quad
\includegraphics[width=3cm]{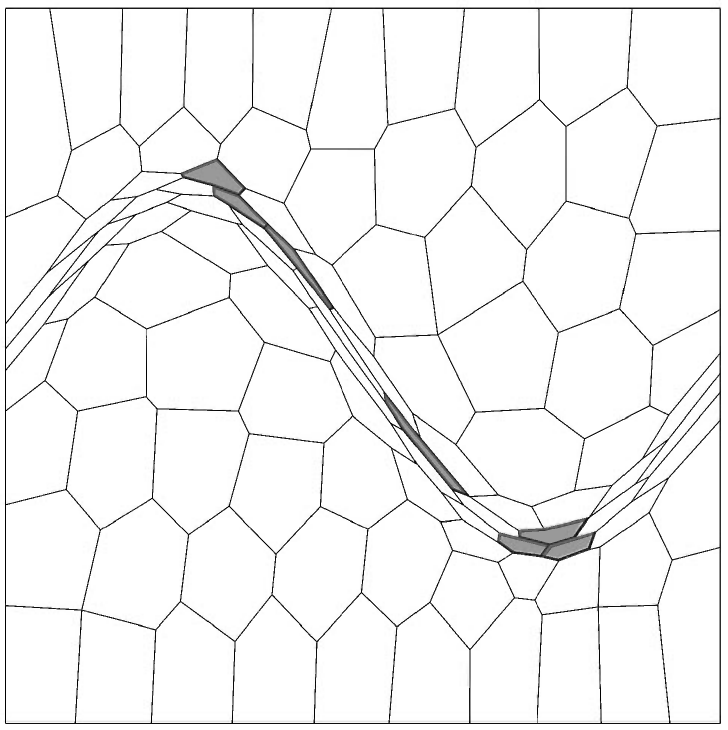}
\caption{The difference between the polygonal MMPDE algorithm and the triangular one applied to the sub-triangulation of the same initial mesh.
The left-most panel: mesh after 10 polygonal MMPDE outer iterations, with all polygons remaining convex;
the middle two panels: mesh after 2 triangular MMPDE outer iterations starts to have non-convex polygons (shaded in gray);
the right-most panel: mesh after 10 triangular MMPDE outer iterations has more non-convex polygons.}
\label{fig:convexity}
\end{center}
\end{figure}

Having all polygons remaining convex in a polygonal mesh is a much appreciated property.
Some numerical tools such as the Wachspress coordinates can only be defined on convex polygons.
Moreover, as pointed out earlier, the sub-triangulation of an all convex polygonal mesh is nonsingular
and hence it is guaranteed \cite{HK2017} that the mesh will not tangle in the next inner iteration.
However, currently there is no built-in mechanism in the polygonal MMPDE algorithm to prevent mesh elements from becoming non-convex,
although numerical results shown in Fig. \ref{fig:convexity} seem promising.
Such issues, as well as the convergence of the algorithm,
are generally difficult to analyze theoretically \cite{BuddHuangRussell09}.
On the other hand, the continuous meshing functional corresponding to
$I_h( \{\revi{\veta_i}\}, \{ \vx_i^{(k)}\})$ is known coercive and polyconvex and has a minimizer \cite{HR11}.
Moreover, numerical examples to be given in Section \ref{sec:numericalResults}
show that the algorithm is efficient and robust. 
}

\revi{
  As mentioned at the end of Section \ref{sec:compareMeasures}, we eliminate short edges in the initial CVT mesh $\T^{(0)}$.
  This ensures that the reference mesh $\hat{\T}_C$ is free of short edges.
  In the subsequent moving mesh iterations, one should not eliminate short edges since this will alter the mesh topology.
  Besides, short edges in highly anisotropic meshes may not be ``short'' when measured under the correct anisotropic measure.
  They need to be distinguished from the true ``short edges'', which are still ``short'' even when measured under the anisotropic measure.
  This work is left to $Q_{ali,2}$, which is able to detect the true ``short edges'' as shown in Section \ref{sec:compareMeasures}.
  We expect that by keeping $Q_{ali,2}$ and $Q_{eq,2}$ small, the mesh will contain few true ``short edges''.
  So far there is no theoretical analysis of this, but numerical results to be presented in Section \ref{sec:numericalResults}
  \revZ{show} that the polygonal moving mesh algorithm appears to handle the short edge issue well.
  }

Finally, we mention that the entire moving mesh algorithm can be extended to 3D
  straightforwardly. To this end, one simply needs to combine the discussion on 3D polyhedral mesh quality measures
  (presented at the end of Section \ref{sec:qualitymeasures}) and a 3D MMPDE framework (see for example \cite{HK2014}).

\section{Numerical results} \label{sec:numericalResults}
In this section, we present numerical results obtained with the moving mesh algorithm given in Section \ref{sec:MMPDE} for a \revZ{selection} of five examples.
We first consider 
the Poisson's equation $-\Delta u = f$ in $\Omega = (0,1)\times (0,1)$ subject to the Dirichlet boundary condition.
Two examples with the following exact solutions are tested,
\begin{align*}
& \text{\bf Example 1: }\qquad  u=\tanh(40y-80x^2) - \tanh(40x-80y^2);\\
&\text{\bf Example 2: } \qquad u =\sqrt{0.5(r-x)} - 0.25 r^2,\quad r=\sqrt{x^2+y^2}.
\end{align*}
These examples are solved on polygonal meshes using the Wachspress finite element method \cite{Wachspress75},
an $H^1$ conforming finite element method using the Wachspress barycentric coordinates as basis functions.
It is known that for smooth exact solutions and sufficiently fine shape-regular quasi-uniform polygonal meshes,
the Wachspress finite element method has the asymptotic convergence order $O(h)$ in $H^1$ semi-norm
and $O(h^2)$ in $L^2$ norm, where $h$ is the characteristic size of the mesh.
If one considers a quasi-uniform polygonal mesh with $N\times N$ polygons, the characteristic mesh size $h$
roughly equals $N^{-1}$. Consequently, the optimal asymptotic order of the approximation error
for the Wachspress finite element method can be expressed into $O(N^{-1})$ in $H^1$ semi-norm
and $O(N^{-2})$ in $L^2$ norm.

\begin{figure}[ht]
\begin{center}
\includegraphics[width=4cm]{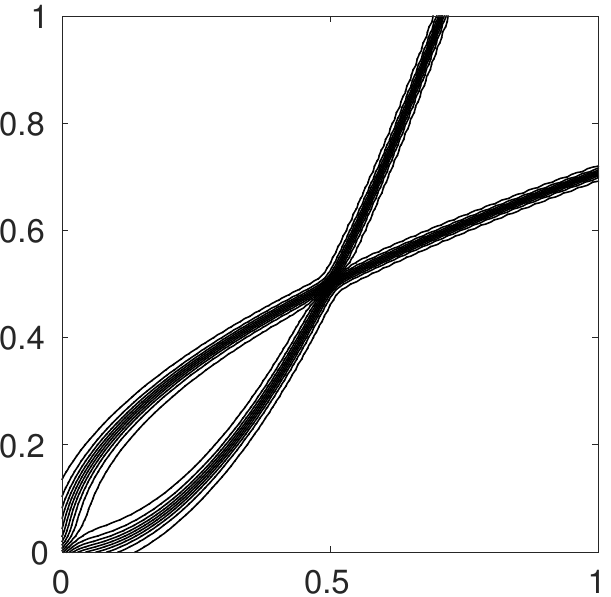}
\caption{Contour plot of the exact solution for Example 1.}
\label{fig:ex1contour}
\end{center}
\end{figure}

Let us briefly explain the reason why we pick these two examples.
For Example 1, the value of $u$ ranges from $-2$ to $2$ in $\Omega$.  An equidistant contour plot of $u$
is shown in Fig.~\ref{fig:ex1contour}. Apparently, $u$ changes rapidly near the curves
$y-2x^2 = 0$ and $x-2y^2 = 0$, while it is almost constant \revZ{elsewhere}. It exhibits a strong anisotropic behavior in the gradient and Hessian near those curves.
When a uniform isotropic mesh is used to discretize Example 1, the number of the elements has to be very large
to resolve the rapid changes of $u$ near the curves.
On the other hand, less elements can be used for the same accuracy when using an adaptive anisotropic mesh,
which is dense around the regions with the rapid changes of $u$ and coarse in places where $u$ is almost constant.
Here, we will show that the moving mesh algorithm described in Section
\ref{sec:MMPDE} is capable of generating high quality anisotropic adaptive polygonal meshes optimized for this problem in terms of
the $L^2$ norm of the approximation error.

Example 2 is a well-known example with a corner singularity at $(0,0)$, and the exact solution $u$ is in $H^{\frac{3}{2}-\varepsilon}(\Omega)$
for arbitrarily small $\varepsilon >0$.
When discretized using quasi-uniform meshes, the best approximation error that can be reached
has asymptotic order $O(h^{0.5})=O(N^{-0.5})$ in $H^1$ semi-norm and $O(h^{1.5})=O(N^{-1.5})$ in $L^2$ norm.
We emphasize that, no matter how fine the uniform mesh is, the asymptotic order of approximation error
cannot be improved due to the intrinsic low regularity of the exact solution.
Later on we will show that adaptive meshes generated by the moving mesh method
not only can lead to smaller errors but also improves the convergence order to the optimal $O(N^{-2})$
for the $L^2$ norm.

For both examples, we set $\Omega_C = \Omega$ and use CVTs generated on $\Omega_C$
by Lloyd's algorithm as the reference computational mesh $\hat{\T}_C$.
\revi{\revZ{This} mesh is processed so that short edges are eliminated by combining nearby vertices into one, 
as discussed in the last paragraph of Section \ref{sec:compareMeasures}.}
According to the numerical results given in Section \ref{sec:qualitymeasures},
the mesh $\hat{\T}_C$ has good quality under the Euclidean metric.
We take the initial physical mesh $\T^{(0)} = \hat{\T}_C$.

We start from a reference computational mesh $\hat{\T}_C$ with $32\times 32$ polygons and apply the moving mesh algorithm with 10 outer iterations
to Example 1. The initial mesh and the physical meshes after $1$ and $10$ outer iterations,
i.e., $\T^{(0)}$, $\T^{(1)}$ and $\T^{(10)}$, are shown in Fig.~\ref{fig:ex1-32-mesh}.
One can see that the meshes correctly capture the rapid changes of the solution (cf. Fig.~\ref{fig:ex1contour}).
In addition, a close view of $\T^{(10)}$ near $(0.5, 0.5)$ (shown in Fig.~\ref{fig:ex1-32-mesh})
clearly shows the anisotropic behavior of the mesh elements.
The history of alignment and equidistribution measures is reported in Fig.~\ref{fig:ex1-32-QQhistory}.
Here, $Q_{ali,1}$, $Q_{ali,2}$ and $Q_{eq,1}$, $Q_{eq,2}$ are computed by comparing
the physical mesh with the reference computational mesh $\hat{\T}_C$,
while $Q_{ali,3}$ and $Q_{eq,3}$ are computed by comparing each polygon $T$ in the physical
mesh with its own reference $K_T$.
One can see that the moving mesh algorithm reduces the mesh quality measures
although the reduction is not monotone. Moreover,
$Q_{ali,2}$ is slightly bigger than $Q_{ali,1}$ and $Q_{ali,3}$
and $Q_{eq,3}$  is slightly bigger than $Q_{eq,1}$.
These are consistent with what we have observed in Section \ref{sec:qualitymeasures}.
Fig.~\ref{fig:ex1-32-QQhistory} also shows that all three sets of mesh quality measures
have very similar evolution patterns.


\begin{figure}[ht]
\begin{center}
\includegraphics[width=3cm]{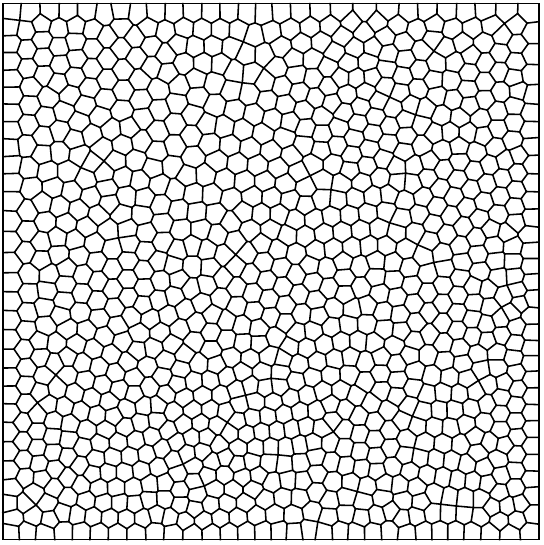}
\includegraphics[width=3cm]{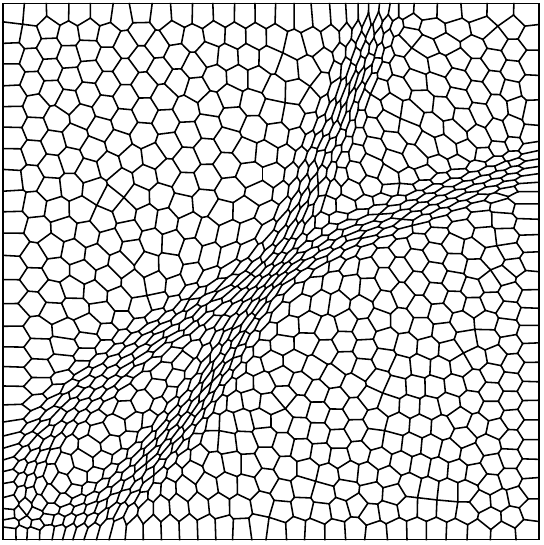}
\includegraphics[width=3cm]{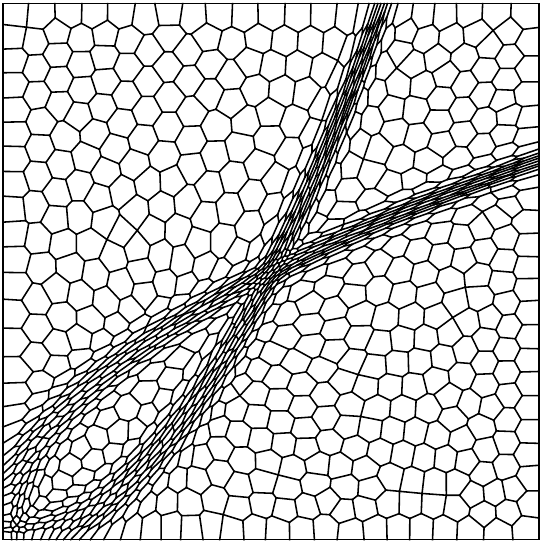}
\includegraphics[width=3cm]{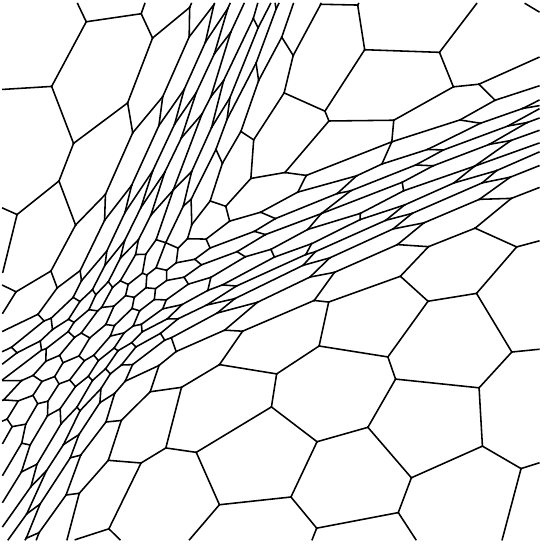}
\caption{Example 1, mesh with $32\times 32$ cells. From left to right:
$\T^{(0)}$, $\T^{(1)}$, $\T^{(10)}$, a close view of $\T^{(10)}$ near (0.5, 0.5).}
\label{fig:ex1-32-mesh}
\end{center}
\end{figure}

\begin{figure}[ht]
\begin{center}
\includegraphics[width=4.5cm]{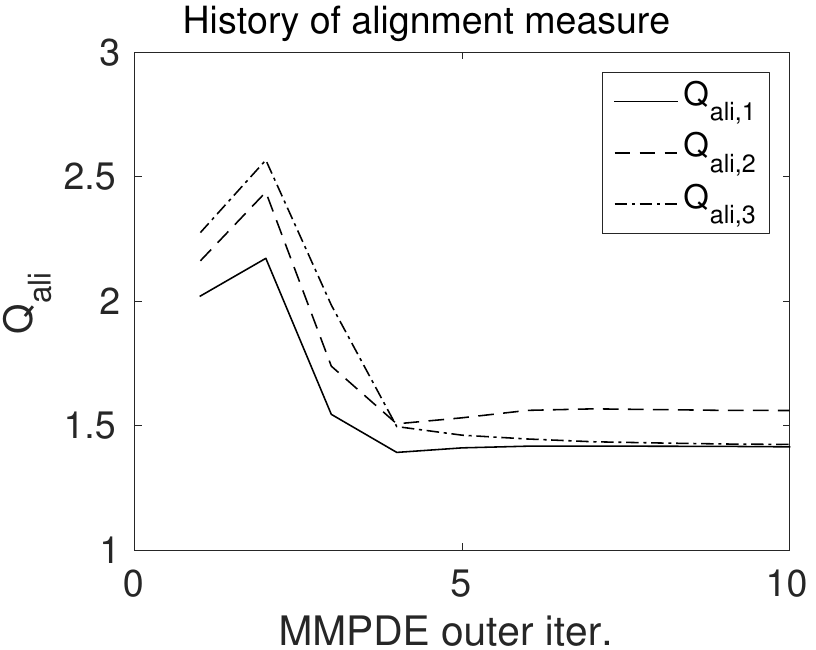}\quad\quad
\includegraphics[width=4.5cm]{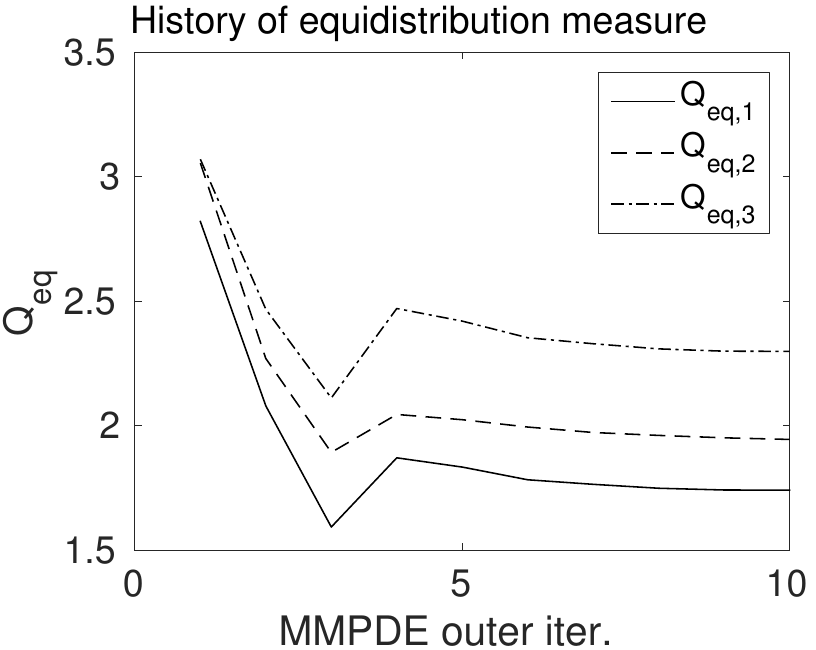}
\caption{Example 1, mesh with $32\times 32$ cells. History of $Q_{ali}$ and $Q_{eq}$.
Subdivision (b) was used for computing $Q_{ali,2}$ and $Q_{eq,2}$.} \label{fig:ex1-32-QQhistory}
\end{center}
\end{figure}

\begin{figure}[ht]
\begin{center}
\includegraphics[width=4.5cm]{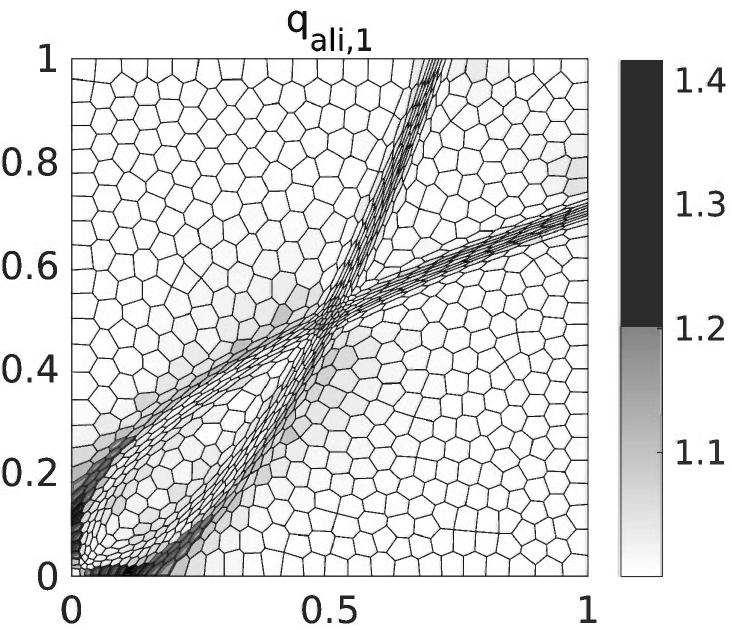}\quad\quad
\includegraphics[width=4.5cm]{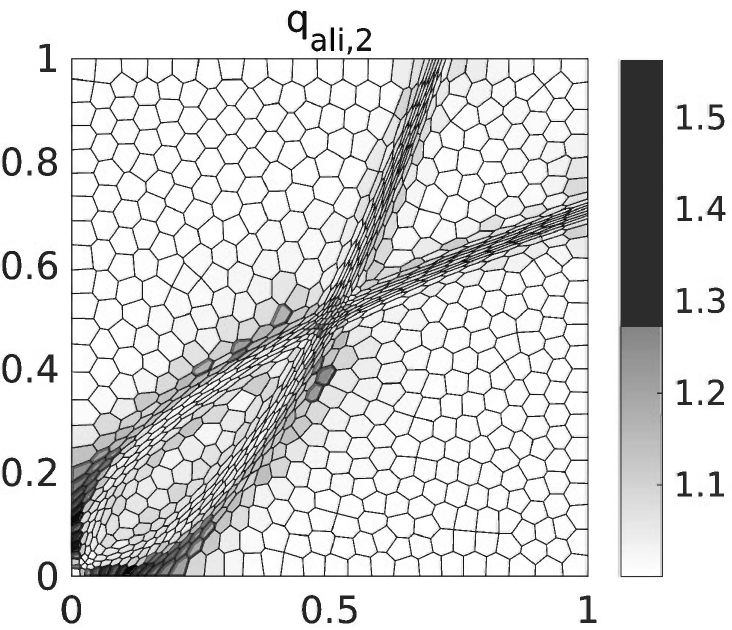}\quad\quad
\includegraphics[width=4.5cm]{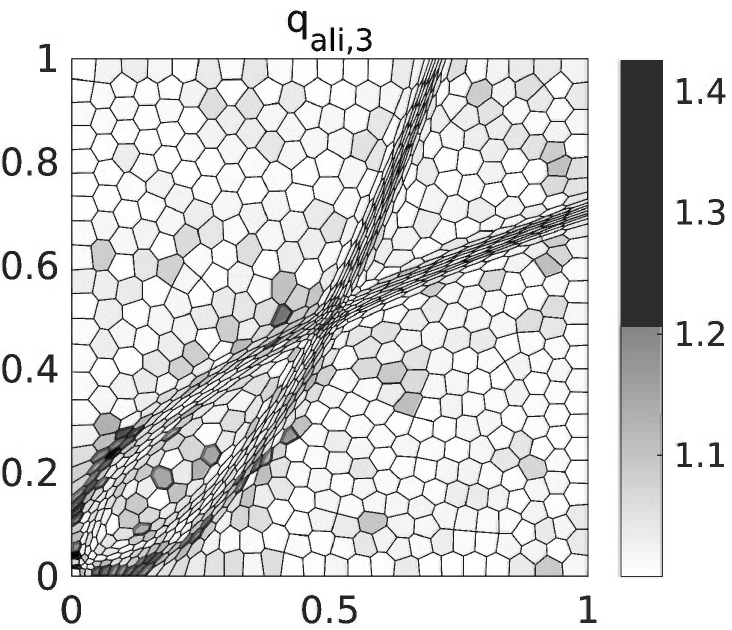}
\caption{Example 1, distribution of $q_{ali,1}$, $q_{ali,2}$ and $q_{ali,3}$ one mesh with $32\times 32$ cells after 10 MMPDE outer iterations.}
\label{fig:ex1-32-qali-dist}
\end{center}
\end{figure}

\begin{figure}[ht]
\begin{center}
\includegraphics[width=4.5cm]{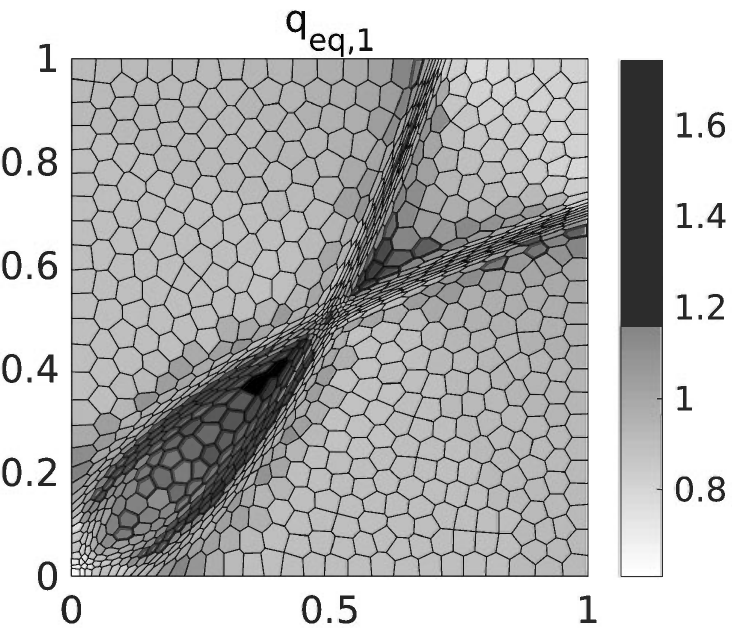}\quad\quad
\includegraphics[width=4.5cm]{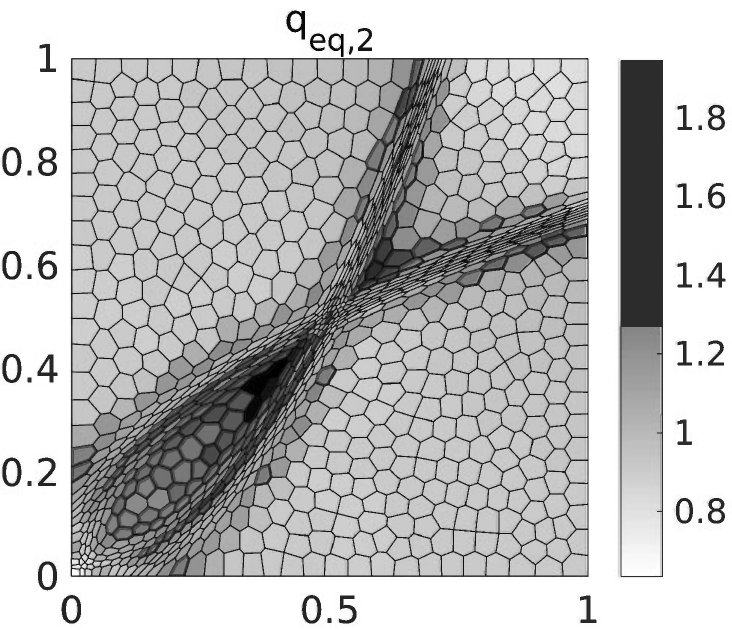}\quad\quad
\includegraphics[width=4.5cm]{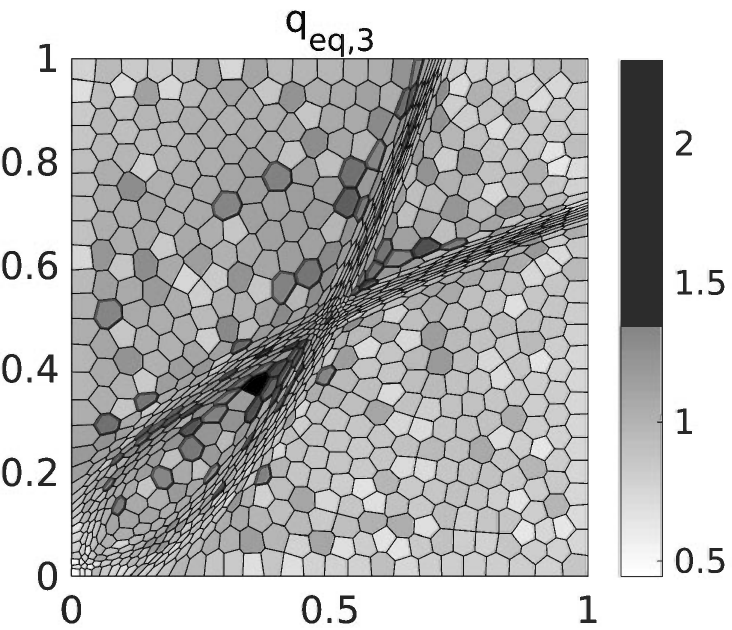}
\caption{Example 1, distribution of $q_{eq,1}$, $q_{eq,2}$ and $q_{eq,3}$ one mesh with $32\times 32$ cells after 10 MMPDE outer iterations.}
\label{fig:ex1-32-qeq-dist}
\end{center}
\end{figure}

\begin{figure}[ht]
\begin{center}
\includegraphics[width=4.5cm]{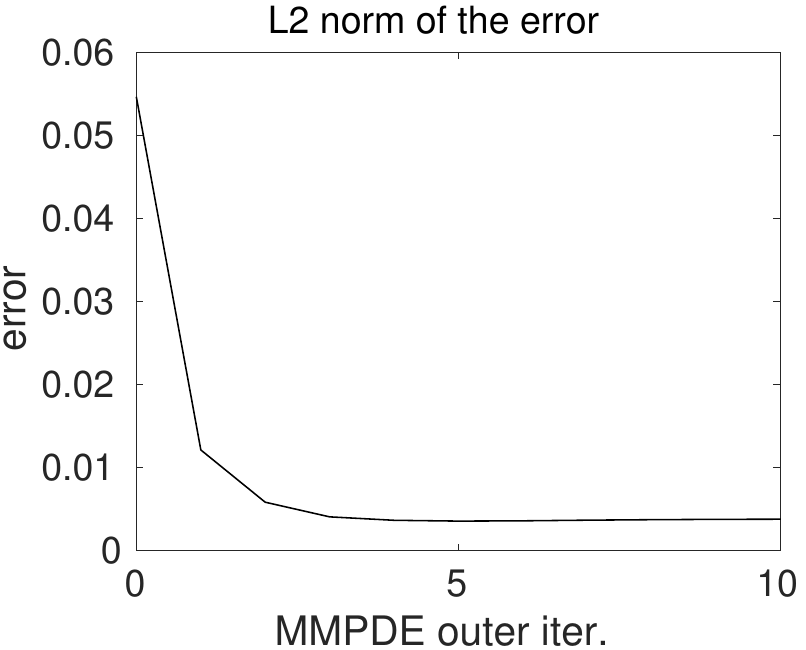}\quad\quad
\includegraphics[width=4.5cm]{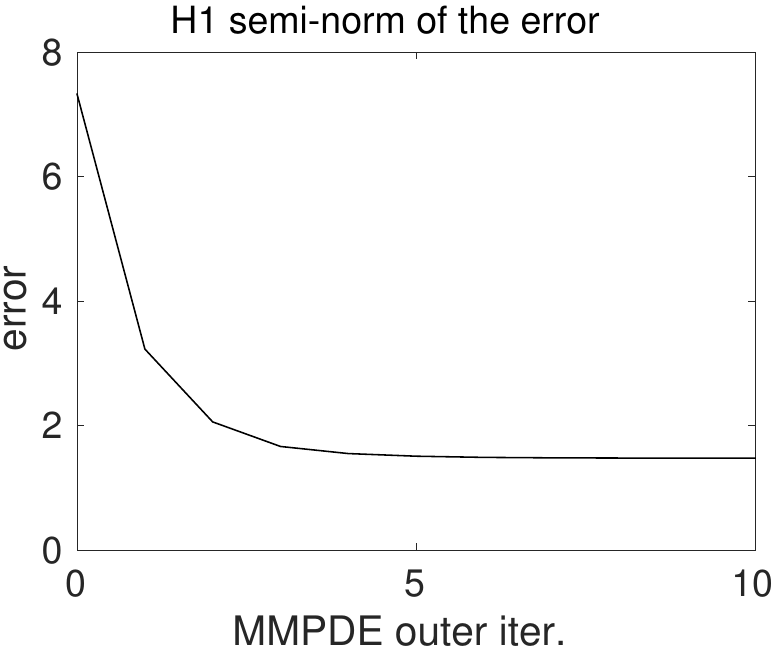}
\caption{Example 1, a mesh with $32\times 32$ cells. History of $L^2$ norm and $H^1$ semi-norm of the error $u-u_h$
is plotted as a function of the outer iteration number.}
\label{fig:ex1-32-errhistory}
\end{center}
\end{figure}

\revi{
  Although we have eliminated short edges from the initial CVT mesh,
  one may wonder whether \revZ{or not} new short edges, measured under the anisotropic metric \eqref{M-1}, are generated during the moving mesh iterations.
  The answer is negative because \revZ{$Q_{ali,2}$},
  which according to Section \ref{sec:compareMeasures} is sensitive to short edges, remains small \revZ{in Fig.~\ref{fig:ex1-32-QQhistory}}.
  To further illustrate this and also to compare three sets of mesh quality measures,
  we compute the piecewise constant functions $q_{ali,k}$ and $q_{eq,k}$, for $k=1,2,3$,
  on the $32\times 32$ mesh after $10$ outer iterations.
  Here
$$
\begin{aligned}
q_{ali,1} &= \frac{\trace ([E_T E_{T_C}^t (E_{T_C} E_{T_C}^t)^{-1}]^t \M_T [E_T E_{T_C}^t (E_{T_C} E_{T_C}^t)^{-1}]) }{2 \det ([E_T E_{T_C}^t (E_{T_C} E_{T_C}^t)^{-1}]^t \M_T [E_T E_{T_C}^t (E_{T_C} E_{T_C}^t)^{-1}])^{1/2}}, \\
q_{eq,1} &= \frac{\det(E_T E_{T_C}^t (E_{T_C} E_{T_C}^t)^{-1} ) \sqrt{\det(\M_T)} }{ \sigma_{h,1}},
\end{aligned}
$$
with $\sigma_{h,1}$ defined as in (\ref{sigma-1}),
and $q_{ali,k}$ and $q_{eq,k}$ for $k=2,3$ are defined similarly.
Note that the ranges of $q_{ali,k}$ and $q_{eq,k}$ are $[1,\infty)$ and $(0, \infty)$, respectively,
and $Q_{ali,k}$, $Q_{eq,k}$ are the maximum norm of $q_{ali,k}$, $q_{eq,k}$.
We report the distributions of $q_{ali,k}$ and $q_{eq,k}$ in Figs.~\ref{fig:ex1-32-qali-dist} \revZ{and} \ref{fig:ex1-32-qeq-dist}.
Clearly $q_{ali,2}$ is bounded above by $1.6$.
The mesh after 10 iterations contains no short edges, measured under the anisotropic metric \eqref{M-1}.
We also notice that $q_{ali,2}$ and $q_{eq,2}$ have patterns very similar to $q_{ali,1}$ and $q_{eq,1}$,
while $q_{ali,3}$ and $q_{eq,3}$ are slightly different.
This is an interesting phenomenon worthy of investigation in the future.
Based on the above observation, in the rest of this section we will only plot $q_{ali,1}$ and $q_{eq,1}$ in the distribution graphs.
}

A more important question is whether \revZ{or not} these adaptive meshes actually help reduce the approximation error.
To examine this, we computed the $L^2$ norm and the $H^1$ semi-norm of $u-u_h$, where $u_h$ is the finite element solution using the Wachspress
finite element method, on these adaptive meshes. The results are reported in Fig.~\ref{fig:ex1-32-errhistory}.
The figure clearly shows the effectiveness of the moving mesh method in reducing the approximation error,
as both the $L^2$ norm and the $H^1$ semi-norm of the error \revB{are reduced to about one-tenth} after applying 10 outer iterations.
Interestingly, although the algorithm does not reduce $Q_{ali}$ or $Q_{eq}$ monotonically, it seems to reduce
the $L^2$ norm and the $H^1$ semi-norm of the error monotonically for Example 1.
We also notice that the majority of reduction occurs within the first few outer iterations.

With the above observations, we continue testing the method for Example 1 \revZ{on} meshes with different numbers of cells.
Consider meshes with $N\times N$ polygonal cells, for $N=8$, 16, 32, 64, and 128.
Again, the reference computational meshes are taken as CVTs generated by Lloyd's algorithm.
Optimal rates of convergence, $O(N^{-1})$ in $H^1$ semi-norm and $O(N^{-2})$ in $L^2$ norm,
usually cannot be achieved when the mesh is not fine enough to resolve all detail of the solution.
In Table \ref{tab:ex1-conv}, it can be seen that on quasi-uniform mesh $\T^{(0)}$, asymptotic order
of errors in $H^1$ semi-norm
is less than $O(N^{-1})$ and improves as $N$ increases. The asymptotic order of the error in $L^2$ norm  on $\T^{(0)}$
appears to be larger than $O(N^{-2})$ at the beginning, which is indeed an indication that very coarse meshes cannot
fully resolve the detail of the solution and thus result in bad approximations. Again, when $N$ increases, the asymptotic order of the error in $L^2$ norm
improves until it reaches $O(N^{-2})$.

\begin{figure}[ht]
\begin{center}
\includegraphics[width=4.5cm]{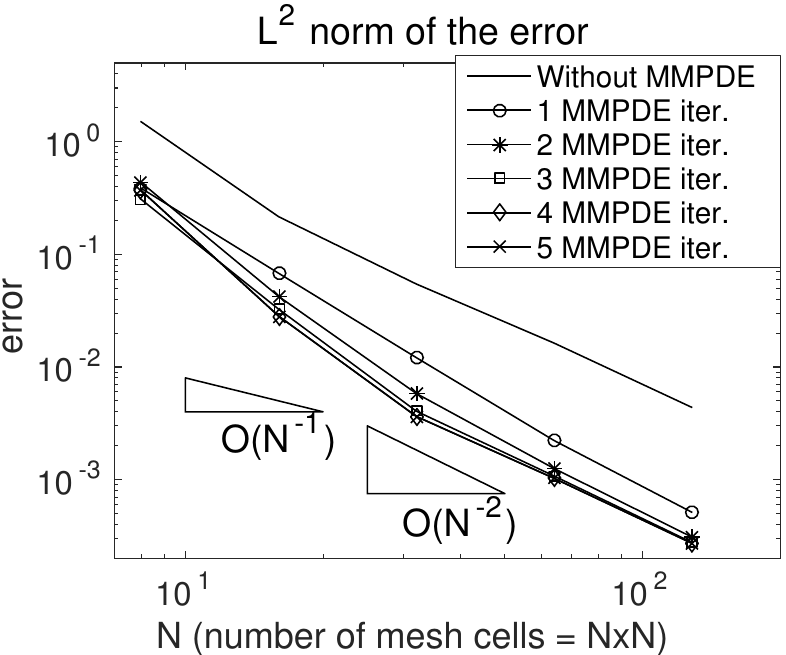}\quad\quad
\includegraphics[width=4.5cm]{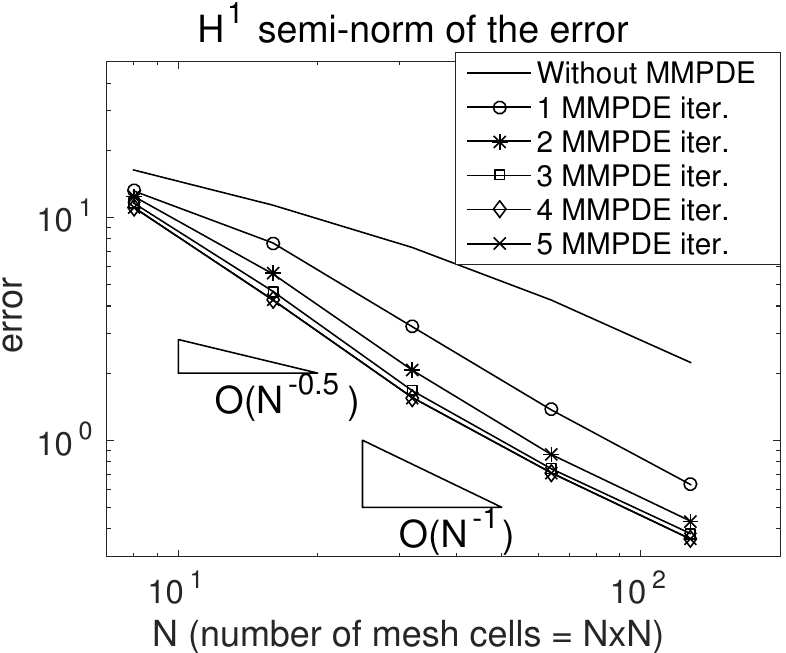}
\caption{Example 1, mesh with $N\times N$ cells, for $N=8$, 16, 32, 64, and 128.
The $L^2$ norm and the $H^1$ semi-norm of the error $u-u_h$ after different numbers of MMPDE outer iterations
are plotted as functions of $N$.} \label{fig:ex1-conv}
\end{center}
\end{figure}

\begin{table}[ht]
  \caption{Example 1, mesh with $N\times N$ cells, for $N=8$, 16, 32, 64, and 128.
The $L^2$ norm and the $H^1$ semi-norm of the error on $\T^{(0)}$, i.e., no MMPDE iteration, and $\T^{(5)}$, i.e., $5$ MMPDE iterations.
Here, the asymptotic order of the error is computed using two consecutive meshes with respect to $\frac{1}{N}$.}
\label{tab:ex1-conv}
  \begin{center}
    \begin{tabular}{|c|c|c|c|c|c|c|c|c|}
      \hline
      & \multicolumn{4}{|c|}{On mesh $\T^{(0)}$} & \multicolumn{4}{|c|}{On mesh $\T^{(5)}$} \\ \hline
      & \multicolumn{2}{|c|}{$L^2$ norm} & \multicolumn{2}{|c|}{$H^1$ semi-norm} & \multicolumn{2}{|c|}{$L^2$ norm} & \multicolumn{2}{|c|}{$H^1$ semi-norm} \\ \hline
      $N$ & error & order & error & order & error & order & error & order \\ \hline
      8  & 1.50e+0 &&  1.63e+1 && 4.15e-1 && 1.16e+1 &  \\ \hline
      16 & 2.16e-1 & 2.8 & 1.13e+1 & 0.5 & 2.65e-2 & 4.0 & 4.17e+0 & 1.5 \\ \hline
      32 & 5.45e-2 & 2.0 & 7.32e+0 & 0.6 & 3.54e-3 & 2.9 & 1.51e+0 & 1.5 \\ \hline
      64 & 1.63e-2 & 1.7 & 4.25e+0 & 0.8 & 1.01e-3 & 1.8 & 7.04e-1 & 1.1 \\ \hline
      128& 4.39e-3 & 1.9 & 2.23e+0 & 0.9 & 2.73e-4 & 1.9 & 3.51e-1 & 1.0 \\ \hline
    \end{tabular}
  \end{center}
\end{table}

Both the $H^1$ semi-norm and $L^2$ norm of the approximation error are significantly improved
by the anisotropic adaptive mesh algorithm.
In Fig.~\ref{fig:ex1-conv}, the approximation error is reported for different mesh sizes and
different numbers of MMPDE outer iterations.
We can see that the approximation error can be effectively reduced by
performing just a few MMPDE outer iterations.
To better compare the quantity of the error, we also list the results on the initial mesh
and on the physical mesh after 5 MMPDE outer iterations in Table~\ref{tab:ex1-conv}.
Asymptotic orders in terms of $\frac{1}{N}$ are reported in the table too.
Improvements in both the $H^1$ semi-norm and the $L^2$ norm and in convergence order
can be observed.

\revA{
  In the implementation, since there is no general quadrature rule on polygons,
  we split the polygons into sub-triangles and use Gaussian quadrature on the sub-triangles to evaluate integrals.
  The Gaussian quadratures are exact for polynomials of corresponding degrees.
  But the Wachspress finite element method uses rational functions as basis functions, which cannot be integrated exactly as polynomials.
  It is not clear how much the inexact numerical integration affects the convergence rates.
  This is a known problem to the community working on generalized barycentric coordinates, and is waiting to be investigated.
  So far, numerical results presented in \revZ{Table} \ref{tab:ex1-conv} appear to \revZ{meet} the expectations.
  }

Next, we test Example 2 under the same settings as for Example 1.
We start from an initial mesh with $32\times 32$ cells and apply the MMPDE algorithm.
The initial mesh and physical meshes after $1$ and $10$ outer iterations of MMPDE are shown in
Fig.~\ref{fig:ex2-32-mesh}.
The history of $Q_{ali}$ and $Q_{eq}$ and the $L^2$ norm and $H^1$ semi-norm of the error
is reported in Figs.~\ref{fig:ex2-32-QQhistory} and \ref{fig:ex2-32-errhistory}.
One may notice that $Q_{ali,1}$, $Q_{ali,2}$, and $Q_{ali,3}$ stay almost constant
while $Q_{eq,1}$, $Q_{eq,2}$, and $Q_{eq,3}$ increase slightly during the MMPDE iterations.
We first point out that the corner singularity in this example is essentially isotropic.
Therefore, any isotropic mesh has good alignment under any metric based
on the recovered Hessian of the solution. As a result, we expect that
$Q_{ali}$ remains small and constant during the MMPDE outer iterations.
\revi{To better illustrate this, in Fig.~\ref{fig:ex2-32-distribution} we plot the distribution of
$q_{ali,1}$ and $q_{eq,1}$.
The left panel of Fig.~\ref{fig:ex2-32-distribution} shows that
$q_{ali,1}$ stays almost constant on the entire domain and its distribution is independent of
the corner singularity.
}

On the other hand, we recall that the exact Hessian for this example is infinite at $(0,0)$.
Thus, \revC{the more adapted the mesh is}, more elements are moved toward the origin
and \revA{the computed Hessian becomes larger at the corner as it gets closer to the true Hessian, 
          which is infinite at the origin}.
\revC{Since the metric is computed using the recovered Hessian, it also becomes very large at the origin.}
As a consequence, it is harder to generate a mesh with perfect equidistribution
and we see that $Q_{eq}$ is getting bigger as the MMPDE outer iterations.
Moreover, $Q_{eq}$ is defined as the maximum norm of $q_{eq}$, which tells about
the mesh quality of worst polygons. As shown in the right panel of Fig.~\ref{fig:ex2-32-distribution},
worst polygons occur near the origin.
The $L^2$ norm of the quantities is shown in Fig.~\ref{fig:ex2-32-QQhistoryL2}.
The results show that the $L^2$ norm decreases almost monotonically,
indicating that the MMPDE algorithm has successfully improved the quality of the majority of mesh elements
although the worst polygons remain ``bad''.

The $L^2$ norm and $H^1$ semi-norm of the approximation error are shown in Fig.~\ref{fig:ex2-32-errhistory}.
Both of them decrease monotonically for Example 2,
which further suggests that the MMPDE algorithm works effectively for Example 2.

\begin{figure}[ht]
\begin{center}
\includegraphics[width=3cm]{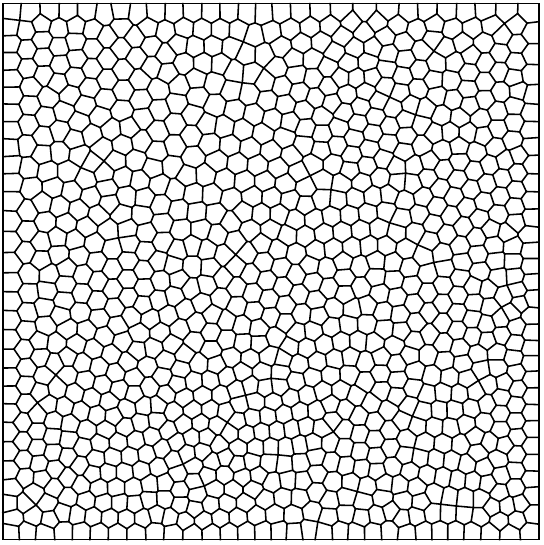}
\includegraphics[width=3cm]{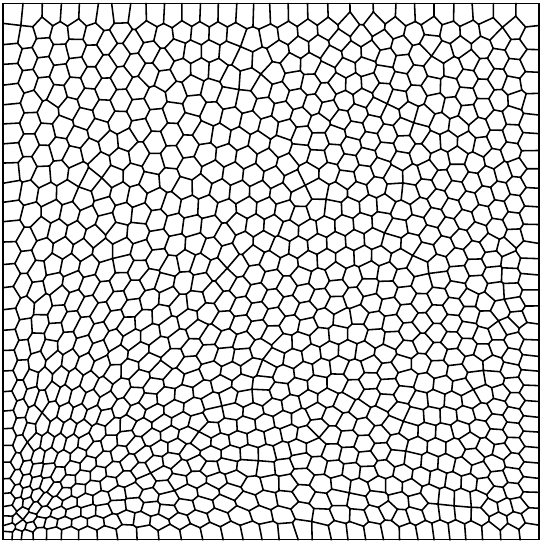}
\includegraphics[width=3cm]{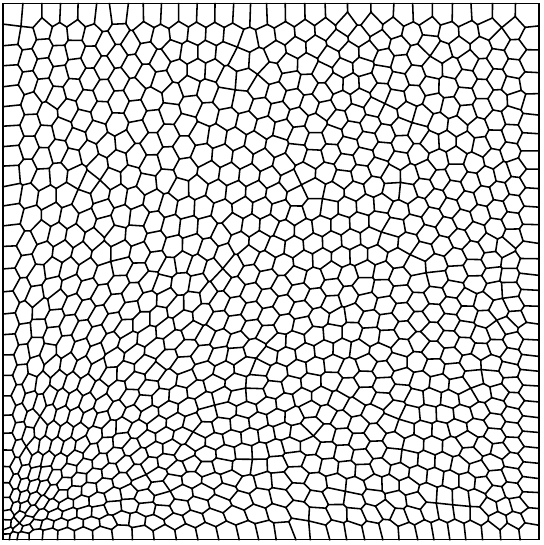}
\includegraphics[width=3cm]{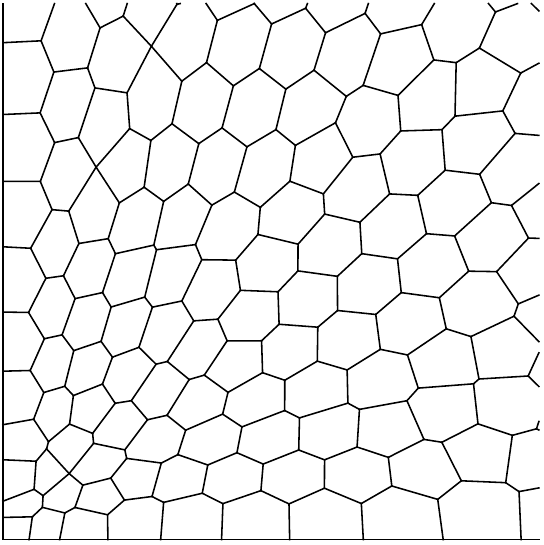}
\caption{Example 2, mesh with $32\times 32$ cells. From left to right:
$\T^{(0)}$, $\T^{(1)}$, $\T^{(10)}$, a close view of $\T^{(10)}$ near the origin.}
\label{fig:ex2-32-mesh}
\end{center}
\end{figure}

\begin{figure}[ht]
\begin{center}
\includegraphics[width=4.5cm]{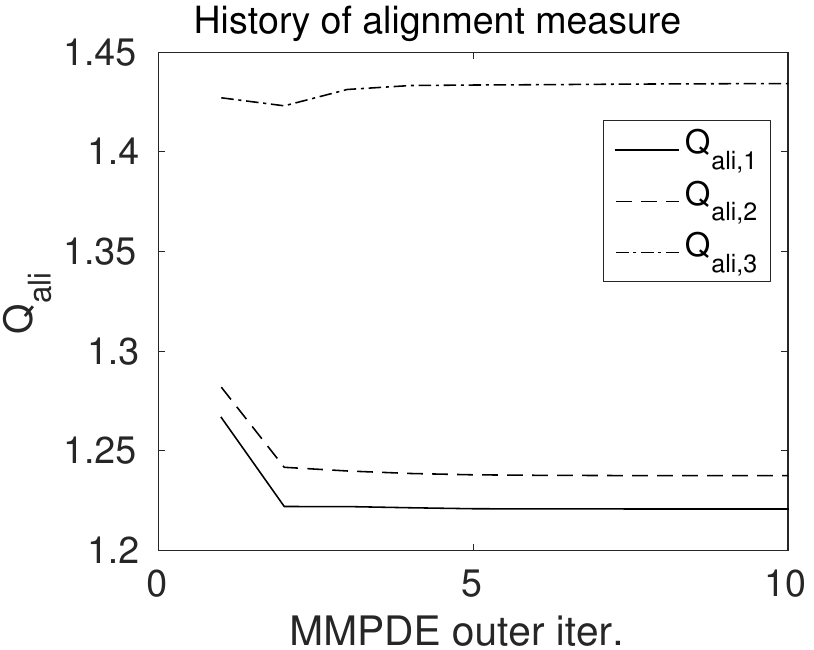}\quad\quad
\includegraphics[width=4.5cm]{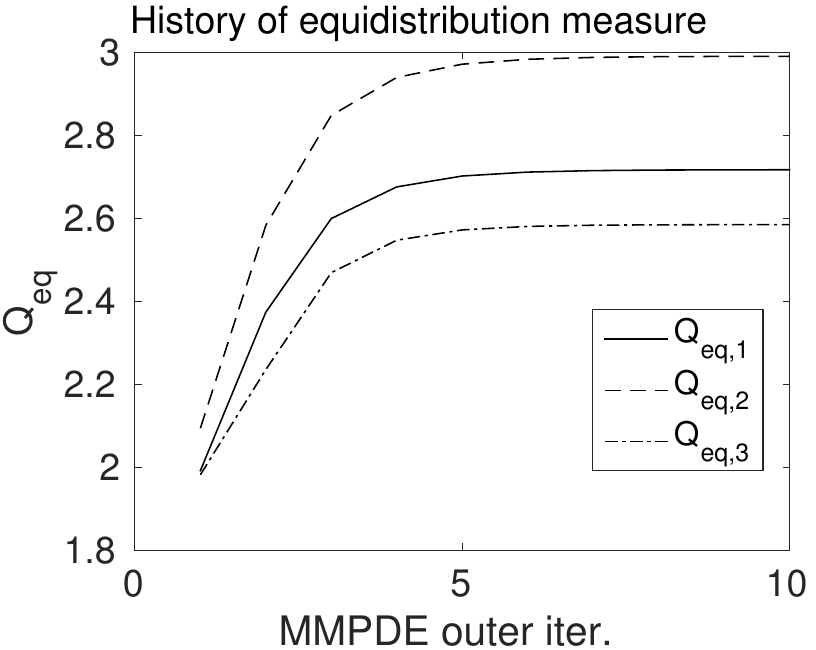}
\caption{Example 2, mesh with $32\times 32$ cells. History of $Q_{ali}$ and $Q_{eq}$.
Subdivision (b) was used for $Q_{ali,2}$ and $Q_{eq,2}$.} \label{fig:ex2-32-QQhistory}
\end{center}
\end{figure}

\begin{figure}[ht]
\begin{center}
\includegraphics[width=4.5cm]{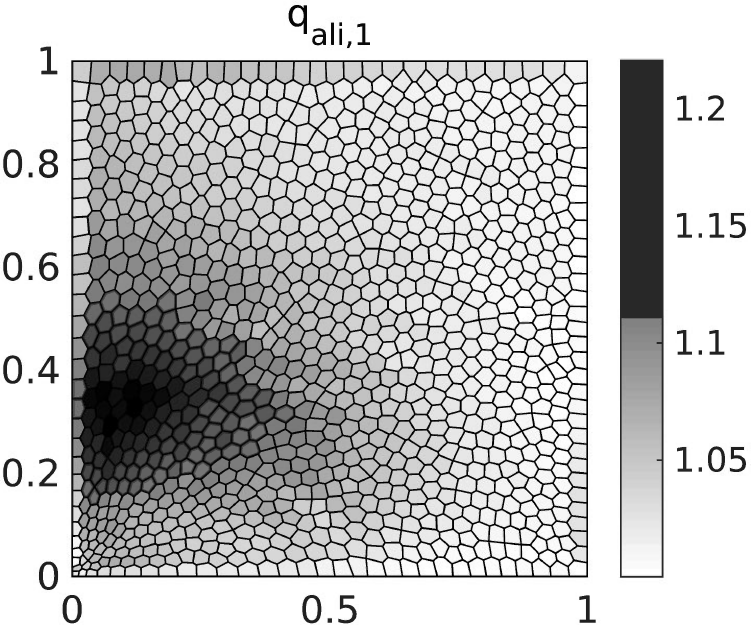}\quad\quad
\includegraphics[width=4.5cm]{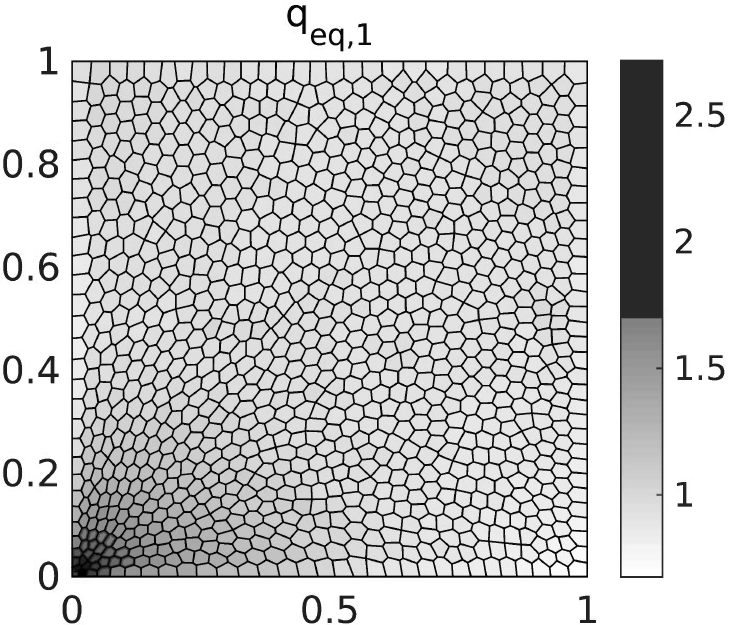}
\caption{Example 2, distribution of $q_{ali,1}$ and $q_{eq,1}$ on mesh with $32\times 32$ cells after 10 MMPDE outer iterations.}
\label{fig:ex2-32-distribution}
\end{center}
\end{figure}

\begin{figure}[ht]
\begin{center}
\includegraphics[width=4.5cm]{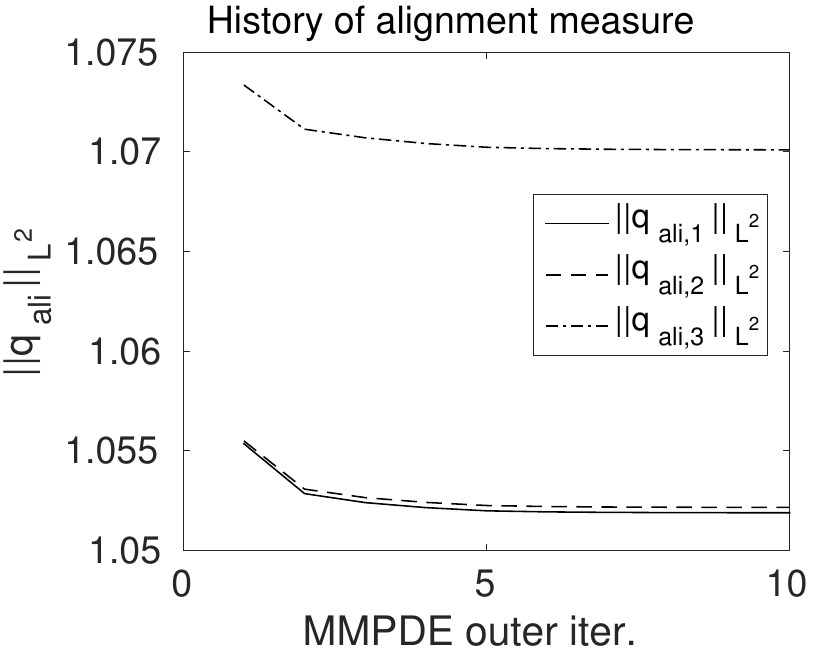}\quad\quad
\includegraphics[width=4.5cm]{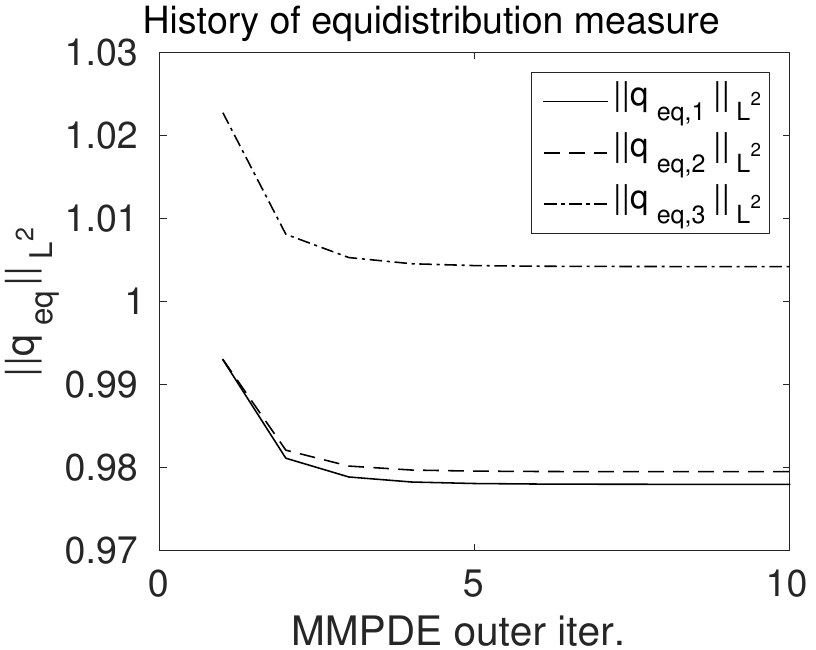}
\caption{Example 2, mesh with $32\times 32$ cells. History of $\|q_{ali}\|_{L^2}$ and $\|q_{eq}\|_{L^2}$.
Subdivision (b) was used for $Q_{ali,2}$ and $Q_{eq,2}$.} \label{fig:ex2-32-QQhistoryL2}
\end{center}
\end{figure}

\begin{figure}[ht]
\begin{center}
\includegraphics[width=4.5cm]{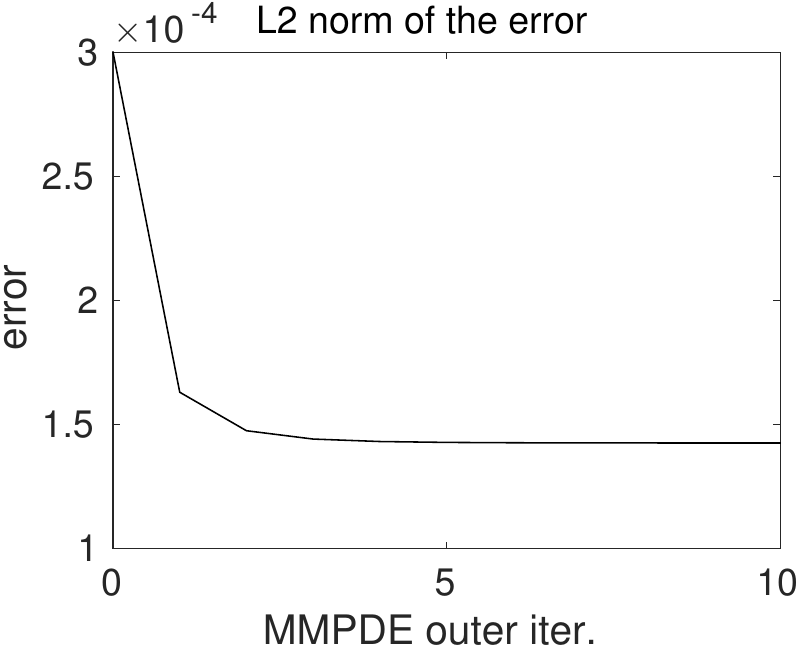}\quad\quad
\includegraphics[width=4.5cm]{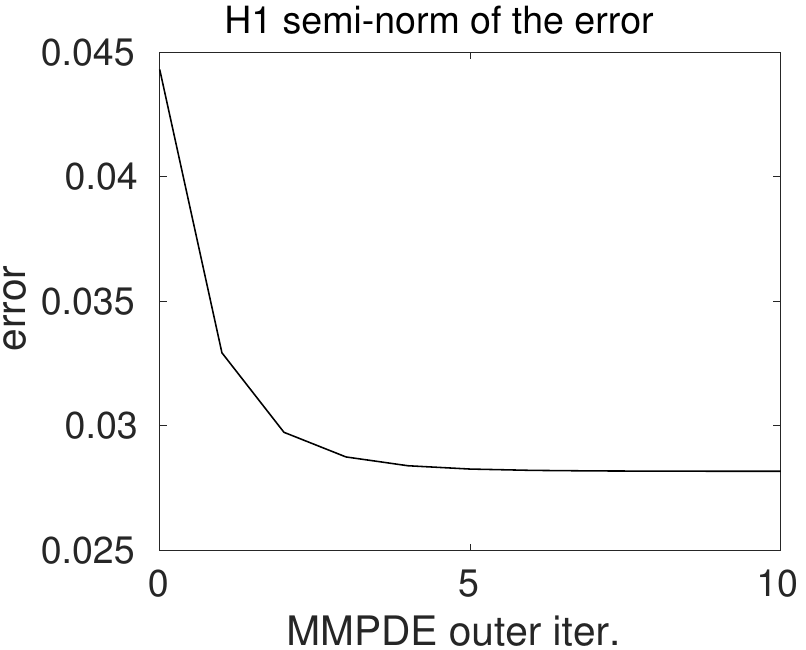}
\caption{Example 2, mesh with $32\times 32$ cells. History of $L^2$ norm and $H^1$ semi-norm of the error $u-u_h$.} \label{fig:ex2-32-errhistory}
\end{center}
\end{figure}

We have also tested the MMPDE algorithm for Example 2 on $N\times N$ meshes with different $N$.
As mentioned earlier, the error $u-u_h$ has at best the asymptotic order $O(N^{-0.5})$
in $H^1$ semi-norm and $O(N^{-1.5})$ in $L^2$ norm on $N\times N$ quasi-uniform meshes.
This can be seen clearly in Table~\ref{tab:ex2-conv}.
In Fig.~\ref{fig:ex2-conv} and Table~\ref{tab:ex2-conv}, the error for Example 2 under different mesh sizes
and MMPDE outer iterations is reported. From Fig.~\ref{fig:ex2-conv}, it is clear that increasing the number of MMPDE outer iterations affects the order of $L^2$ norm  and $H^1$ semi-norm of the approximation error.
The numerical values of the asymptotic order reported in Table \ref{tab:ex2-conv} give a clearer comparison.
After $5$ MMPDE outer iterations, the $L^2$ norm of the approximation error achieves almost $O(N^{-2})$.
The asymptotic order of $H^1$ semi-norm also improves, but not to the optimal $O(N^{-1})$ order.
This may be because the algorithm is designed to optimize the $L^2$ norm of the error
instead of the $H^1$ semi-norm of the error.
To see this, we use the metric tensor
\begin{equation}
\M = {\det \left(\alpha_h I +  |H(u_h)| \right)}^{- \frac{1}{4}}
\| \alpha_h I +  |H(u_h)| \|^{\frac{1}{2}}
\left [ \alpha_h I  +  |H(u_h)| \right ] ,
\label{M-2}
\end{equation}
where $\alpha_h$ is determined through 
$$
\revZ{\int_\Omega \sqrt{\det(\M)} d \V{x}= 2 \int_\Omega {\det\left( |H(u_h)|\right)}^{\frac{1}{4}} \big\| |H(u_h)| \big\|^{\frac{1}{2}} d \V{x} .}
$$
This metric tensor
is based on minimizing the $H^1$ semi-norm of linear interpolation error \cite{Hua05b}.
The numerical results are shown in Fig.~\ref{fig:ex2-convH1} and Table~\ref{tab:ex2-convH1}.
The convergence order in the $H^1$ semi-norm improves (around 0.9) while maintaining
the second order rate for the $L^2$ norm.

\begin{figure}[ht]
\begin{center}
\includegraphics[width=4.5cm]{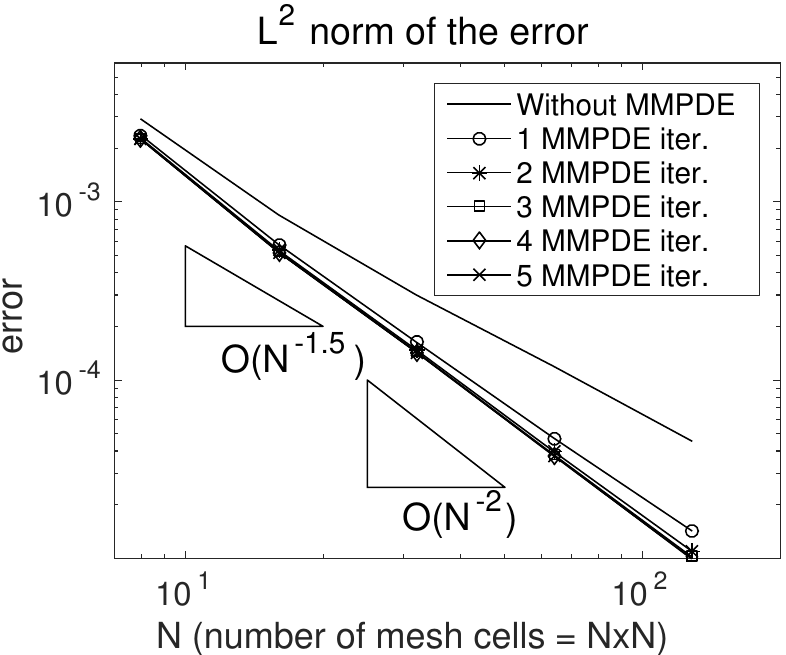}\quad\quad
\includegraphics[width=4.5cm]{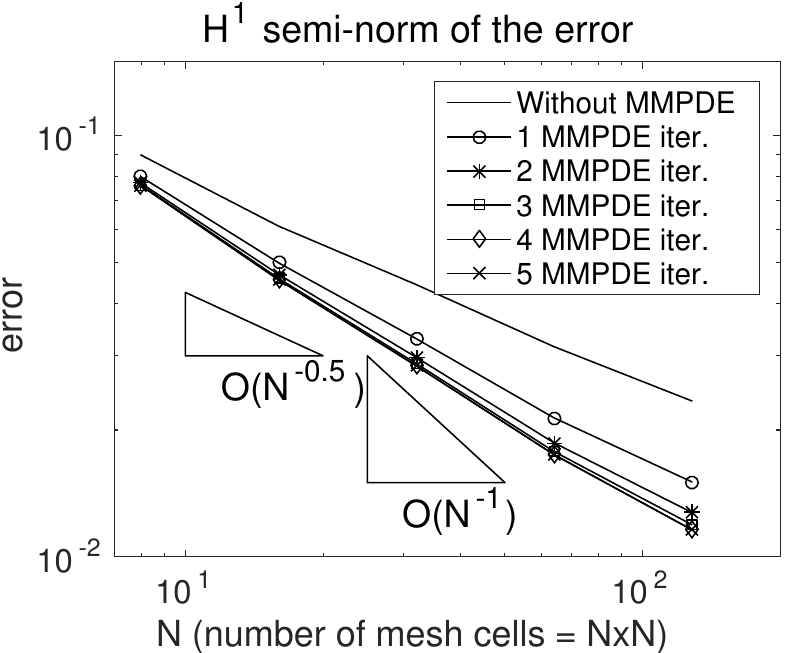}
\caption{Example 2, mesh with $N\times N$ cells, for $N=8,16,32,64,128$.
The $L^2$ norm and $H^1$ semi-norm of the error $u-u_h$ after different numbers of MMPDE outer iterations
are plotted as functions of $N$.}
\label{fig:ex2-conv}
\end{center}
\end{figure}

\begin{table}[ht]
  \caption{Example 2, mesh with $N\times N$ cells, for $N=8$, 16, 32, 64, and 128.
The $L^2$ norm and $H^1$ semi-norm of the error on $\T^{(0)}$, i.e., no MMPDE iteration,
and $\T^{(5)}$, i.e., $5$ MMPDE iterations.
The asymptotic order of the error is computed using two meshes with consecutive number of cells with respect to $\frac{1}{N}$.
}
\label{tab:ex2-conv}
  \begin{center}
    \begin{tabular}{|c|c|c|c|c|c|c|c|c|}
      \hline
      & \multicolumn{4}{|c|}{On mesh $\T^{(0)}$} & \multicolumn{4}{|c|}{On mesh $\T^{(5)}$} \\ \hline
      & \multicolumn{2}{|c|}{$L^2$ norm} & \multicolumn{2}{|c|}{$H^1$ semi-norm} & \multicolumn{2}{|c|}{$L^2$ norm} & \multicolumn{2}{|c|}{$H^1$ semi-norm} \\ \hline
      $N$ & error & order & error & order & error & order & error & order \\ \hline
      8  & 2.90e-3 &&  8.99e-2 && 2.23e-3 && 7.59e-2 &  \\ \hline
      16 & 8.43e-4 & 1.8 & 6.09e-2 & 0.6 & 5.19e-4 & 2.1 & 4.53e-2 & 0.7  \\ \hline
      32 & 3.00e-4 & 1.5 & 4.43e-2 & 0.5 & 1.43e-4 & 1.9 & 2.83e-2 & 0.7  \\ \hline
      64 & 1.19e-4 & 1.3 & 3.15e-2 & 0.5 & 3.74e-5 & 1.9 & 1.74e-2 & 0.7  \\ \hline
      128& 4.55e-5 & 1.4 & 2.35e-2 & 0.4 & 9.84e-6 & 1.9 & 1.14e-2 & 0.6  \\ \hline
    \end{tabular}
  \end{center}
\end{table}

\begin{figure}[ht]
\begin{center}
\includegraphics[width=4.5cm]{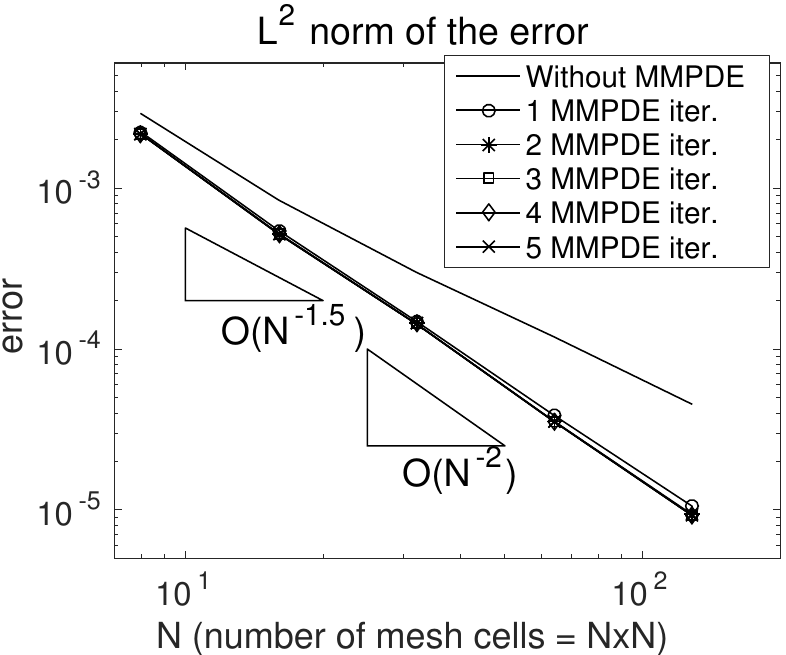}\quad\quad
\includegraphics[width=4.5cm]{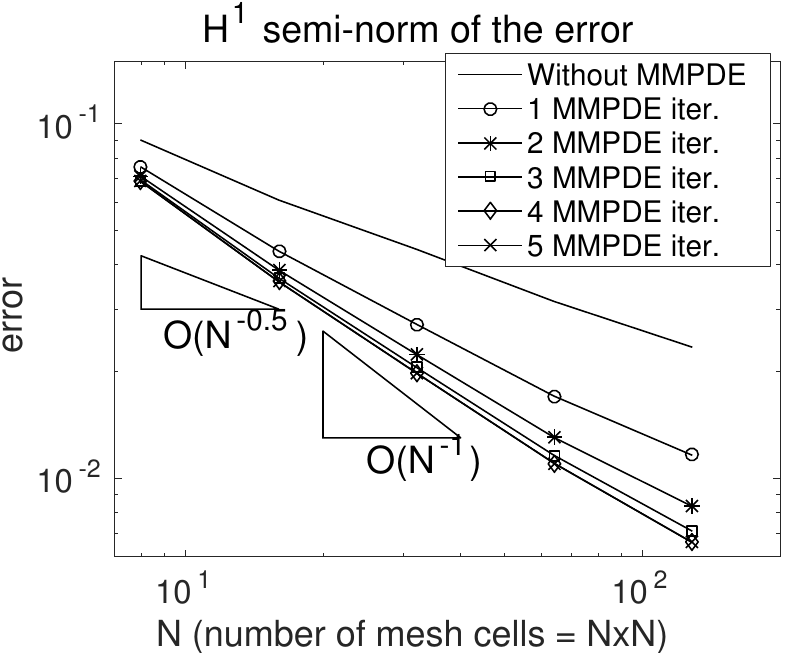}
\caption{Example 2, mesh with $N\times N$ cells, for $N=8,16,32,64,128$.
Optimized for $H^1$ semi-norm \revB{instead of} $L^2$ norm.
The $L^2$ norm and $H^1$ semi-norm of the error $u-u_h$ after different numbers of MMPDE outer iterations
are plotted as functions of $N$.
$H^1$ semi-norm based metric tensor (\ref{M-2}) is used.
} \label{fig:ex2-convH1}
\end{center}
\end{figure}

\begin{table}[ht]
  \caption{Example 2, mesh with $N\times N$ cells, for $N=8$, 16, 32, 64, and 128.
Optimized for $H^1$ semi-norm instead of for $L^2$ norm.
The $L^2$ norm and $H^1$ semi-norm of the error on $\T^{(0)}$, i.e., no MMPDE iteration,
and $\T^{(5)}$, i.e., $5$ MMPDE iterations.
The asymptotic order of the error is computed using two meshes with consecutive number
of cells with respect to $\frac{1}{N}$.
 $H^1$ semi-norm based metric tensor (\ref{M-2}) is used.
}
\label{tab:ex2-convH1}
  \begin{center}
    \begin{tabular}{|c|c|c|c|c|c|c|c|c|}
      \hline
      & \multicolumn{4}{|c|}{On mesh $\T^{(0)}$} & \multicolumn{4}{|c|}{On mesh $\T^{(5)}$} \\ \hline
      & \multicolumn{2}{|c|}{$L^2$ norm} & \multicolumn{2}{|c|}{$H^1$ semi-norm} & \multicolumn{2}{|c|}{$L^2$ norm} & \multicolumn{2}{|c|}{$H^1$ semi-norm} \\ \hline
      $N$ & error & order & error & order & error & order & error & order \\ \hline
      8  & 2.90e-03 &&  8.99e-02 && 2.17e-03 && 6.83e-02 &  \\ \hline
      16 & 8.43e-04 & 1.8 & 6.09e-02 & 0.6 & 5.18e-04 & 2.1 & 3.55e-02 & 0.9  \\ \hline
      32 & 3.00e-04 & 1.5 & 4.43e-02 & 0.5 & 1.43e-04 & 1.9 & 1.95e-02 & 0.9  \\ \hline
      64 & 1.19e-04 & 1.3 & 3.15e-02 & 0.5 & 3.55e-05 & 2.0 & 1.08e-02 & 0.9  \\ \hline
      128& 4.55e-05 & 1.4 & 2.35e-02 & 0.4 & 9.28e-06 & 1.9 & 6.36e-03 & 0.8  \\ \hline
    \end{tabular}
  \end{center}
\end{table}

  Next, we demonstrate the robustness of the proposed polygonal moving mesh PDE algorithm.
  In the previous numerical tests, the MMPDE algorithm starts from an initial CVT mesh, which is of high quality.
  One may wonder that, when starting from a general polygonal mesh, whether the MMPDE algorithm is still able to
  capture the anisotropic structure of the solution correctly and avoid mesh tangling.
  To this end, we tested Example 1 starting from a $32\times 32$ Voronoi (but not CVT) initial mesh,
  \revA{generated by applying $1$ Lloyd's iteration to a completely random initial mesh to make it a little smoother.}
  The rest of the settings are the same as before.
  \revC{Note that the reference mesh $\hat{\T}_C$ is still set as the non-CVT initial mesh, because
  we are not able to get a CVT with exactly the same topological structure as a given polygonal mesh.}
  The initial mesh and the physical meshes after $1$, $5$ and $10$ outer iterations are shown in Fig.~\ref{fig:ex3-32-mesh}.
  Again, the meshes correctly capture the rapid changes of the solution.
  \revC{One may notice that in areas where the solution is flat, the mesh keeps some pattern of the non-CVT initial mesh.
  This is understandable since the reference mesh is not a CVT. 
  But the MMPDE algorithm still focuses correctly on the rapid changing area instead of on the flat area.}
  To illustrate this, we draw the distribution of $q_{ali,1}$ and $q_{eq,1}$ for $\T^{(10)}$ in Fig.~\ref{fig:ex3-32-distribution}.
  Note that the equidistribution measure $q_{eq,1}$ has larger value around the two curves where the solution changes rapidly.
  Thus this area will be treated in priority while the mesh movements in the flat areas will be relatively small.

  We also report the history of the
  approximation error and mesh quality measures in Fig.~\ref{fig:ex3-32-errhistory} and Fig.~\ref{fig:ex3-32-QQhistoryL2}.
  It is interesting to see that there is a small bump in the history. This turns out to be the result of a few ``bad'' mesh cells
  on which the values of $q_{ali}$ and $q_{eq}$ become temporarily large. However, it seems that the MMPDE algorithm is in general
  quite robust and {\it there is no mesh tangling} despite of these few bad-shaped cells.

 \begin{figure}[htb]
\begin{center}
\includegraphics[width=3cm]{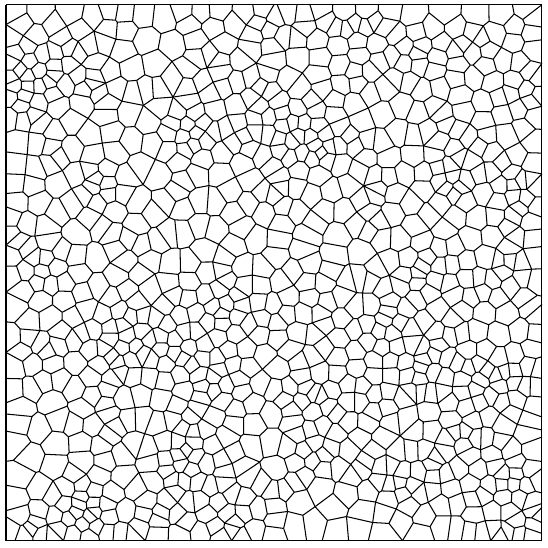}
\includegraphics[width=3cm]{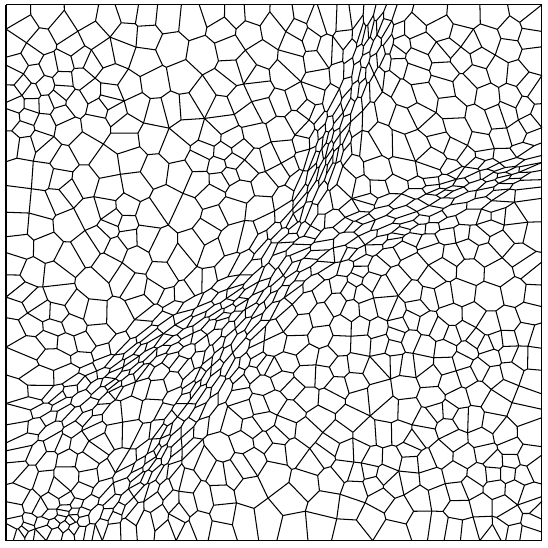}
\includegraphics[width=3cm]{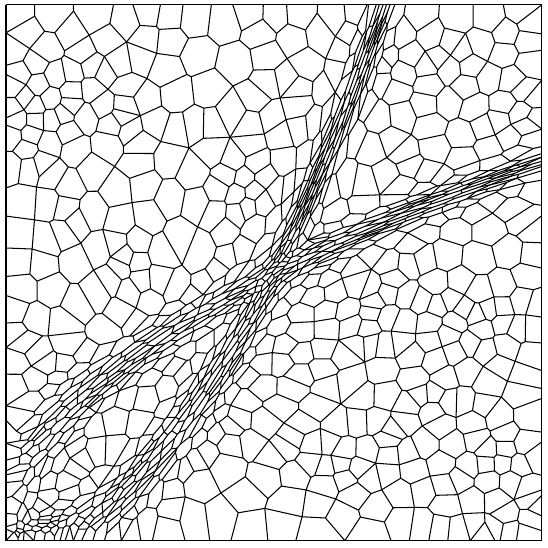}
\includegraphics[width=3cm]{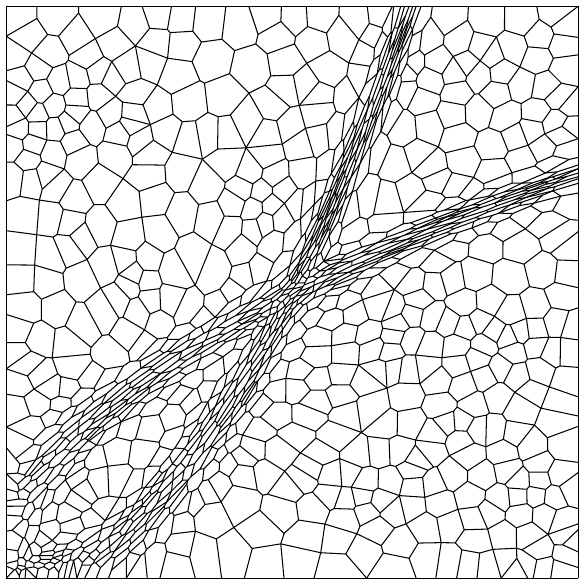}
\caption{Example 1 tested on $32\times 32$ non-CVT mesh. From left to right:
$\T^{(0)}$, $\T^{(1)}$, $\T^{(5)}$, and $\T^{(10)}$.}
\label{fig:ex3-32-mesh}
\end{center}
\end{figure} 

\begin{figure}[htb]
\begin{center}
\includegraphics[width=4.5cm]{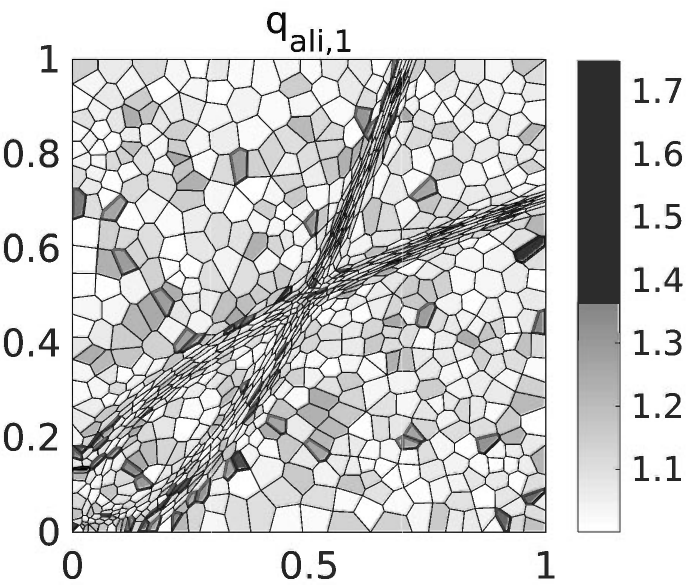}\quad\quad
\includegraphics[width=4.5cm]{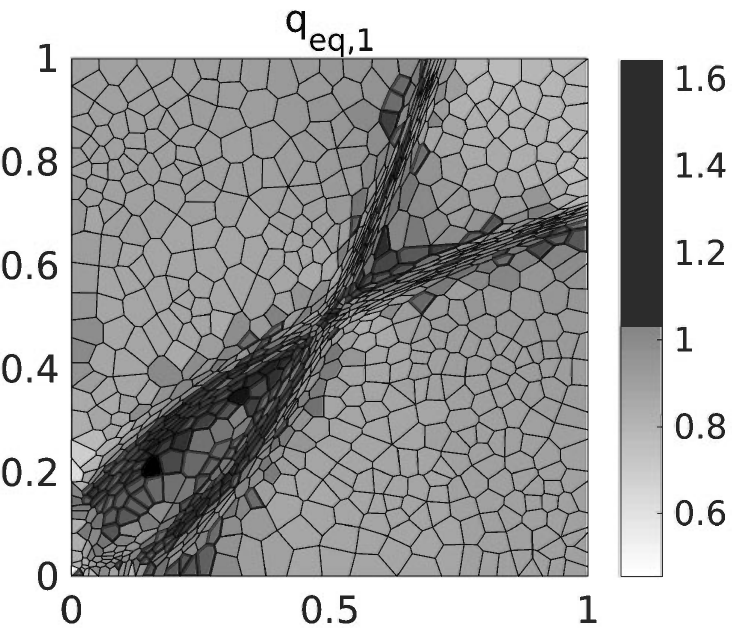}
\caption{Example 1 tested on $32\times 32$ non-CVT mesh. Distribution of $q_{ali,1}$ and $q_{eq,1}$ after 10 MMPDE outer iterations.}
\label{fig:ex3-32-distribution}
\end{center}
\end{figure}
 
\begin{figure}[htb]
\begin{center}
\includegraphics[width=4.5cm]{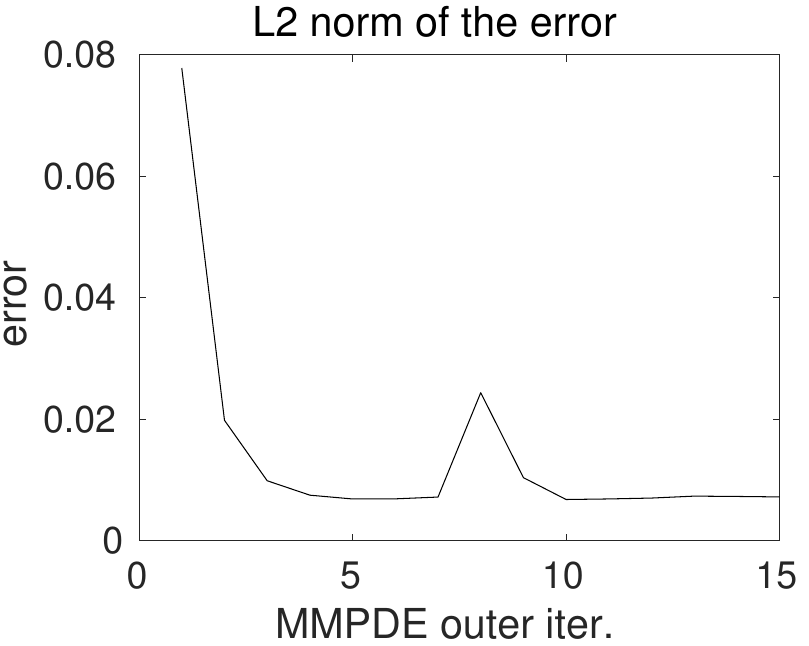}\quad\quad
\includegraphics[width=4.5cm]{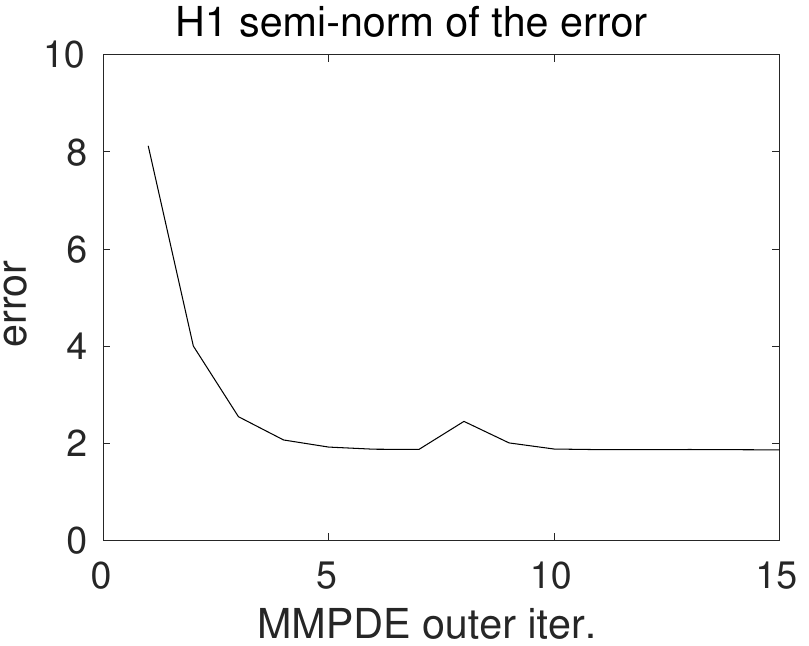}
\caption{Example 1 tested on $32\times 32$ non-CVT mesh. History of $L^2$ norm and $H^1$ semi-norm of the error $u-u_h$
is plotted as a function of the outer iteration number.}
\label{fig:ex3-32-errhistory}
\end{center}
\end{figure}

\begin{figure}[htb]
\begin{center}
\includegraphics[width=4.5cm]{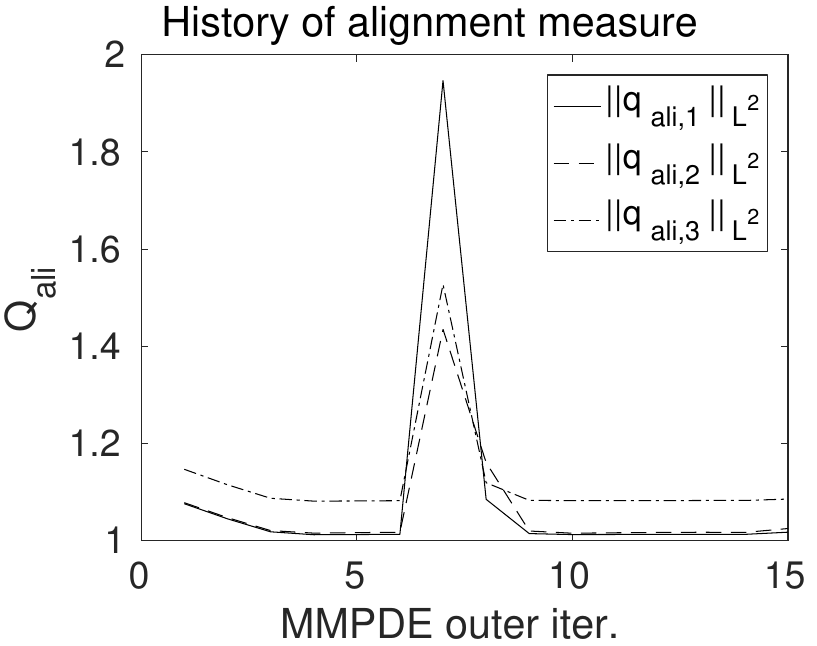}\quad\quad
\includegraphics[width=4.5cm]{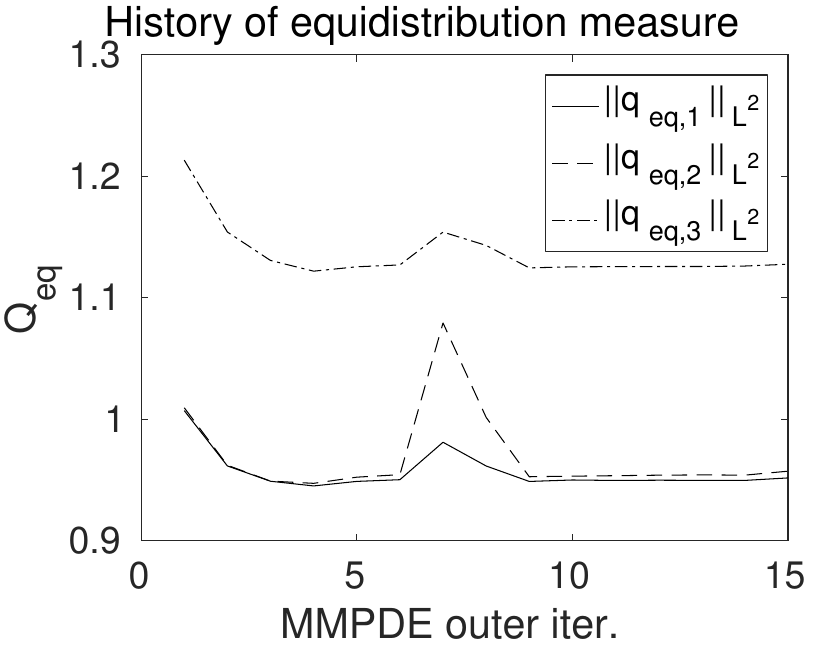}
\caption{Example 1 tested on $32\times 32$ non-CVT mesh. History of $\|q_{ali}\|_{L^2}$ and $\|q_{eq}\|_{L^2}$.
Subdivision (b) was used for $Q_{ali,2}$ and $Q_{eq,2}$.} \label{fig:ex3-32-QQhistoryL2}
\end{center}
\end{figure}

\revB{
   In examples 1 and 2, we have manually stopped the outer iteration after 10 steps, 
   which is long enough to get a good overview of the convergence history.
  From the history of mesh quality measures in Figs. \ref{fig:ex1-32-QQhistory} and \ref{fig:ex2-32-QQhistory},
   we see that the mesh quality measures do not decrease monotonically in the outer iterations. 
   This poses a problem on how to set the stopping criteria for the outer iteration.
   Here we suggest two possibilities based on the numerical results presented above.
   One may stop the outer iteration after three consecutive steps where the changes in the mesh quality measures stay in a given range,
   or one may stop when the $L^2$ norm or the $H^1$ norm of the error stops decreasing. 
   In practice, the above criteria can also be combined with other practical considerations. 
   Keep in mind that the MMPDE outer iterations, once interrupted, can restart from the interruption point instead of the initial step,
   as long as one \revZ{saves} the reference mesh.
   }

  The numerical results from examples 1 and 2 have sufficiently demonstrated the behavior of the MMPDE method \revZ{for} Poisson's equation.
  Based on these observations, we further test the method on three more general second-order elliptic equations of different types.
  For all three examples, the domain is still chosen as $\Omega=(0,1)\times (0,1)$ and the initial meshes as CVT meshes.
  The metric tensor $\M$ is set as in (\ref{M-1}).
  All other settings of the MMPDE solver remain the same as in the previous tests.
  \begin{description}
  \item{\bf Example 3:} This example is the boundary value problem of a strongly anisotropic diffusion equation,
    $$
    \begin{cases}
      -\nabla\cdot ( A\nabla u) = 0,\qquad &\textrm{in }\Omega\revZ{,}\\
      u = g, &\textrm{on }\partial\Omega\revZ{,}
      \end{cases}
    $$
    where
    $$
    \begin{aligned}
      A &= \begin{bmatrix}500.5 & 499.5\\ 499.5 & 500.5\end{bmatrix} 
        \quad \textrm{ and } \quad g = \begin{cases}1, &\textrm{on } [0, \frac{7}{8}]\times\{1\} \revZ{,}\\
                              8-8x, \quad &\textrm{on } (\frac{7}{8},1]\times\{1\} \revZ{,}\\
                              1, &\textrm{on } \{0\}\times[\frac{1}{8}, 1] \revZ{,}\\
                              8y, &\textrm{on } \{0\}\times[0,\frac{1}{8}) \revZ{,}\\
                              0, &\textrm{otherwise}.
        \end{cases}                                                                                  
    \end{aligned}
    $$               
  \item{\bf Example 4:} This example is a convection-dominant problem, i.e., 
    $$
    \begin{cases}
      -\varepsilon \Delta u + \boldsymbol{\beta}\cdot\nabla u  = 0,\qquad &\textrm{in }\Omega\revZ{,}\\
      u = g, &\textrm{on }x=0,\,x=1,\textrm{ and }y=0 \revZ{,}\\
      \frac{\partial u}{\partial n} = 0, \qquad  &\textrm{on }y=1 \revZ{,}
      \end{cases}
    $$
    where $\varepsilon$ is a small positive number,
    $$
    \boldsymbol{\beta} = \begin{bmatrix}10y^2-12x+1\\y+1\end{bmatrix},\qquad
      \quad \textrm{ and } \quad g =
      \begin{cases}
        0, &\textrm{on } \{0\}\times (0.5,1)\textrm{ and }(0.5,1]\times \{0\} \revZ{,}\\
        1, &\textrm{on } \{0\}\times (0,0.5]\textrm{ and }[0,0.5]\times \{0\} \revZ{,}\\
        sin^2(\pi y), &\textrm{on } \{1\}\times [0,1] .
      \end{cases}   
      $$
      Note that $g$ is discontinuous at the points $(0.5,0)$ and $(0,0.5)$, and hence does not belong to $H^{\frac{1}{2}}(\partial\Omega)$.
      Thus it is not a legitimate Dirichlet boundary condition.
      However, this type of Dirichlet boundary condition can still be implemented in the numerical discretization, as will be explained in details later.
      
    \item{\bf Example 5:} The last example is a diffusion-reaction problem with a thin boundary layer,
      $$
    \begin{cases}
      - \Delta u + \revZ{10^6} u  = \revZ{10^6}, \qquad &\textrm{in }\Omega \revZ{,}\\
      u = 0, &\textrm{on }\partial\Omega \revZ{.}
      \end{cases}
    $$
  \end{description}

  We examine these examples  one by one. Example 3 has been studied in \cite{Hua11, HW15, LH10}. Its exact solution is unknown but satisfies
  the maximum principle and its values lie strictly between $[0,1]$.
  Whether the numerical
  solution satisfies a discrete maximum principle or not depends heavily on the discretization \revZ{and mesh}, as demonstrated in \cite{Hua11, HW15, LH10}.
  So far \revA{there is no theoretical guarantee of the discrete maximum principle} of the Wachspress finite element on polygonal meshes.
  \revB{Numerical experiments have shown} that it does not hold on arbitrary polygonal meshes.
  Therefore, we want to test on Example 3 to check if the MMPDE method can capture the strongly anisotropic feature of the solution and if it helps to improve the solution in terms of the discrete maximum principle.
  In the test we use an initial $32\times 32$ CVT mesh.
  In Fig.~\ref{fig:ex3}, the mesh, \revZ{numerical solution}, distributions of $q_{ali,1}$ and $q_{eq,1}$ after 5 MMPDE outer iterations
  are reported, from which one can see that the MMPDE generates an {\it anisotropic polygonal mesh} aligned well with the structure
  of the solution. We also report the maximum and the minimum values of $u_h$ for each MMPDE iteration in Table \ref{tab:ex3}.
  From the table, one can see that the maximum and \revA{the minimum values of $u_h$ are outside of the range $[0,1]$}
  but the overshoot and undershoot quickly become very small after a few MMPDE iterations. This indicates that  
  the Wachspress finite element does not satisfy the discrete maximum principle
 on the polygonal mesh generated by the MMPDE method while the adaptation does improve
 the situation with much smaller overshoots/undershoots.
  It remains a future research topic to explore if it is possible to ensure the discrete maximum principle on polygonal meshes.
  On the other hand,  it can \revDD{readily} be seen that the MMPDE method correctly \revZ{captures
  the strongly anisotropic feature of the solution.}
  
\begin{figure}[htb]
\begin{center}
\includegraphics[width=3cm]{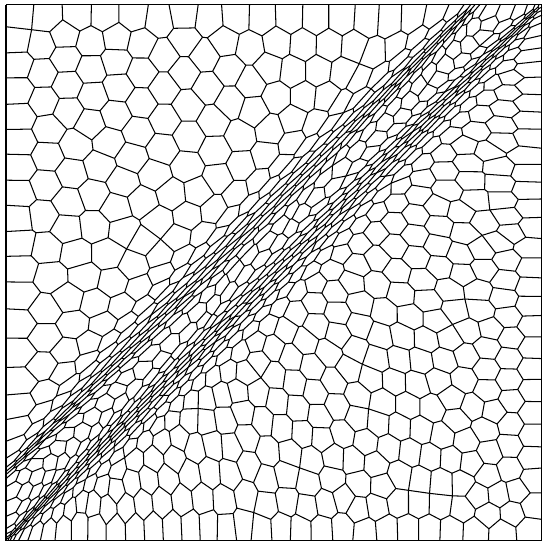}
\includegraphics[width=4cm]{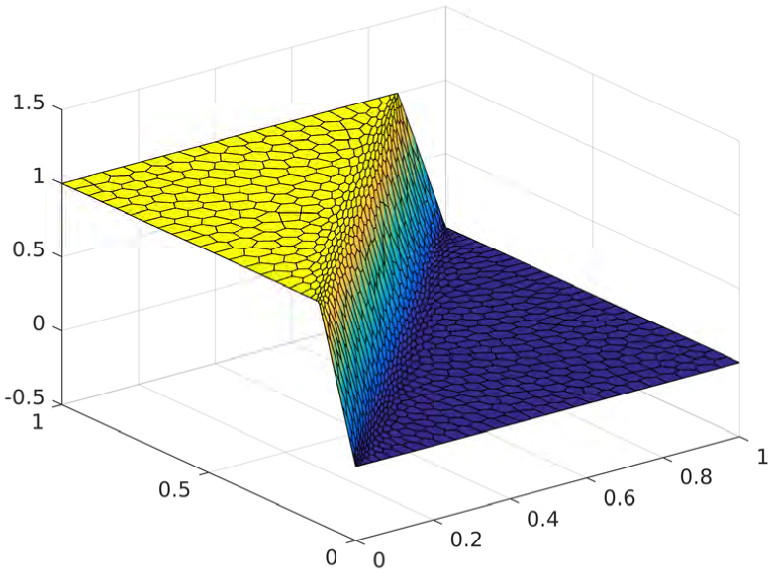}
\includegraphics[width=4cm]{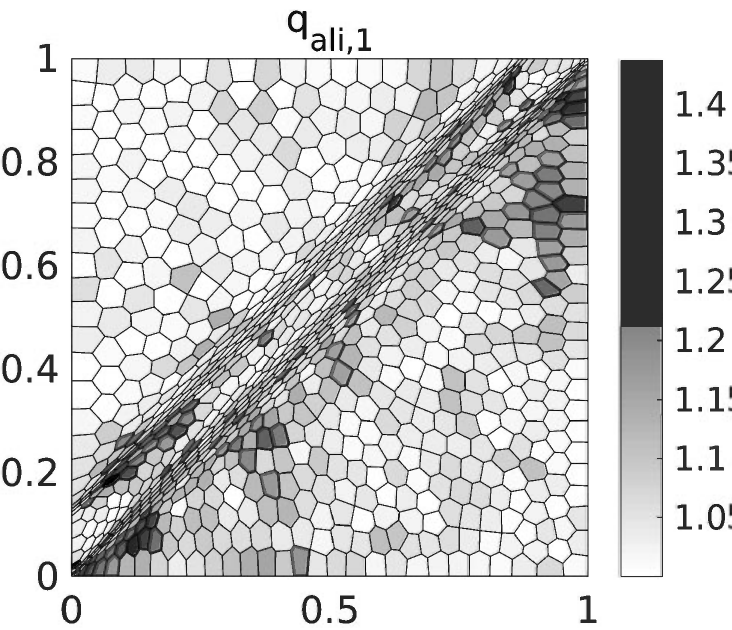}
\includegraphics[width=4cm]{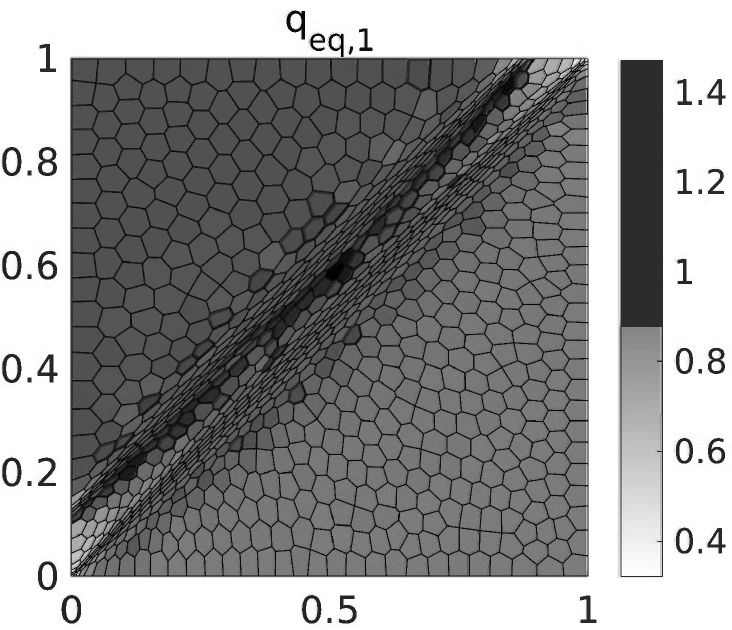}
\caption{Example 3 tested on $32\times 32$ mesh. The mesh, numerical solution, distributions of $q_{ali,1}$ and $q_{eq,1}$ after 5 MMPDE outer iterations.}
\label{fig:ex3}
\end{center}
\end{figure}

\begin{table}[ht]
  \caption{Example 3, the maximum and minimum values of $u_h$ at each MMPDE iteration steps, on  $32\times 32$ mesh.}
\label{tab:ex3}
  \begin{center}
    \begin{tabular}{|c|c|c|c|c|c|}
      \hline
      & iter. 1 &  iter. 2 &  iter. 3 &  iter. 4 &  iter. 5 \\ \hline
      $\max{u_h}$ & $1.0181$ & $1.0049$ & $1.0001$ & $1.0002$ & $1.0002$  \\ \hline
      $\min{u_h}$ & $-0.0231$& $-0.0055$& $-0.0001$& $-0.0002$& $-0.0002$  \\ \hline
    \end{tabular}
  \end{center}
\end{table}

Example 4 comes from modifying a hyperbolic benchmark problem in the Hermes2D C++ library available at {\tt http://www.hpfem.org/}.
The original hyperbolic benchmark has $\varepsilon = 0$ and does not require any boundary condition on the outflow boundary $y=1$.
Moreover, it allows the inflow boundary condition $g$ to have jumps at the points $(0.5,0)$ and $(0,0.5)$.
We add a small diffusion term to make the problem elliptic, and set a Neumann boundary condition on $y=1$.
We expect that the example still keeps the most interesting feature from the original hyperbolic problem, that is,
\revA{there are two internal layers that meet at a point on the boundary and are tangential to each other at the point},
as shown in Figs.~\ref{fig:ex4-1}-\ref{fig:ex4-4}.
Another interesting feature of the solution to Example 4  comes from the discontinuity of the Dirichlet boundary data $g$ at the points  $(0.5,0)$ and $(0,0.5)$.
Because of the jump, function $g$ does not belong to $H^{\frac{1}{2}}(\partial\Omega)$ and hence is not a legitimate boundary condition for the variational problem.
A famous example of this type of discontinuous Dirichlet boundary condition is the Stokes/Navier-Stokes driven cavity problem.
\revB{
  The finite element implementation can circumvent this problem by
  replacing the discontinuous Dirichlet boundary data with its continuous interpolation,
  which has the same effect of adding a smoothing layer with the width of one mesh element \cite{CaiWang09}.
}
However, it sometimes \revA{introduces a weak singularity} around these jump points.
Oscillations will occur in the numerical solution if the mesh cannot fully resolve the
rich details of the solution, as shown in the left graph of Fig.~\ref{fig:ex4-1}.

We first test Example 4 with $\varepsilon = 10^{-3}$ on a $32\times 32$ mesh.
In Fig.~\ref{fig:ex4-1}, the mesh after 10 MMPDE outer iterations is presented, together with the front and the back view of the numerical solution.
Small oscillations can be observed around the point $(0.5,0)$.
All other regions look fine, including where the two internal layers meet and the point $(0,0.5)$ where another jump of $g$ occurs.
By comparing the points $(0.5,0)$ and $(0,0.5)$, we suspect that the incidental angle of an internal layer with $\partial\Omega$
may affect the numerical solution. 
In Fig.~\ref{fig:ex4-2}, we plot the distributions of $q_{ali,1}$ and $q_{eq,1}$,
which further confirms that it is more difficult to get mesh elements around the point $(0.5,0)$ to align with the solution.

\begin{figure}[htb]
  \begin{center}
\includegraphics[width=4cm]{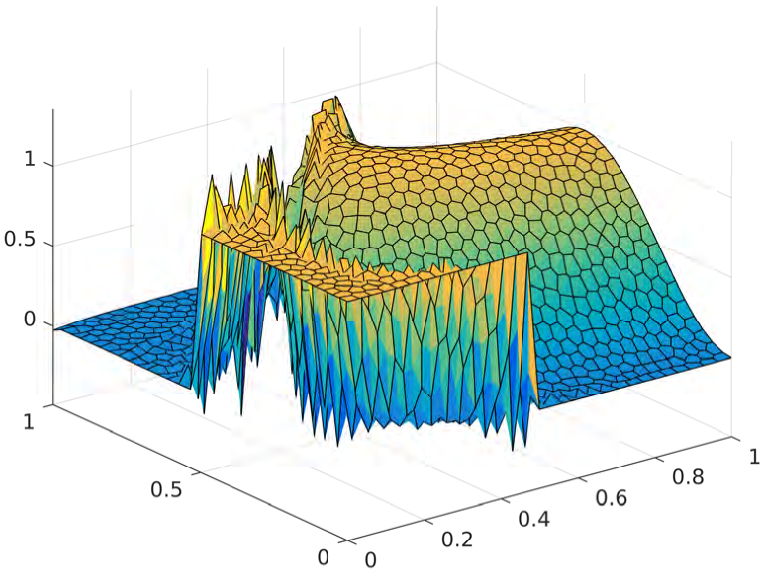}    
\includegraphics[width=3cm]{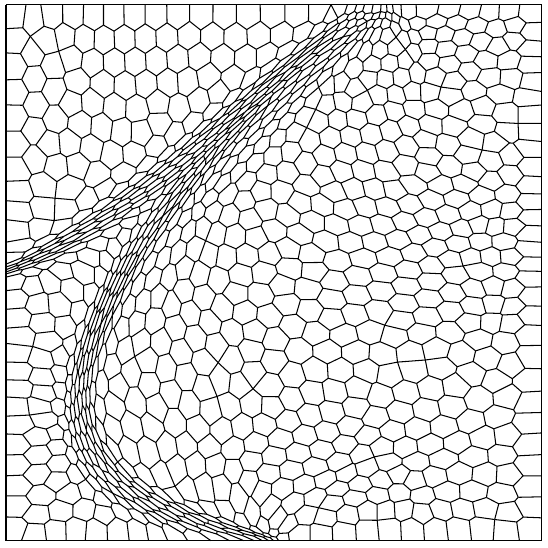}
\includegraphics[width=4cm]{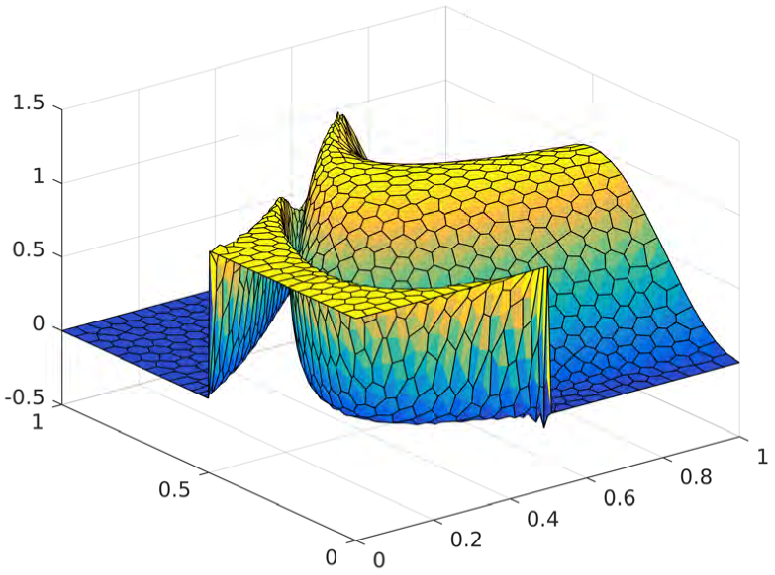}
\includegraphics[width=4cm]{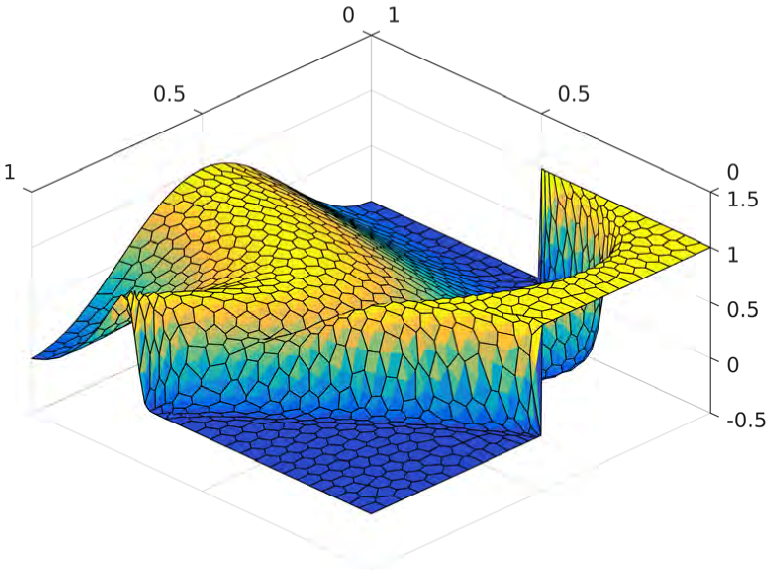}
\caption{Example 4 with $\varepsilon = 10^{-3}$ on $32\times 32$ mesh.
    The leftmost graph is the numerical solution on a uniform CVT mesh. 
    The rest of the graphs are the mesh,  the front and the back view of the numerical solution after 10 MMPDE outer iterations.}
\label{fig:ex4-1}
\end{center}
\end{figure}

\begin{figure}[htb]
  \begin{center}
\includegraphics[width=4cm]{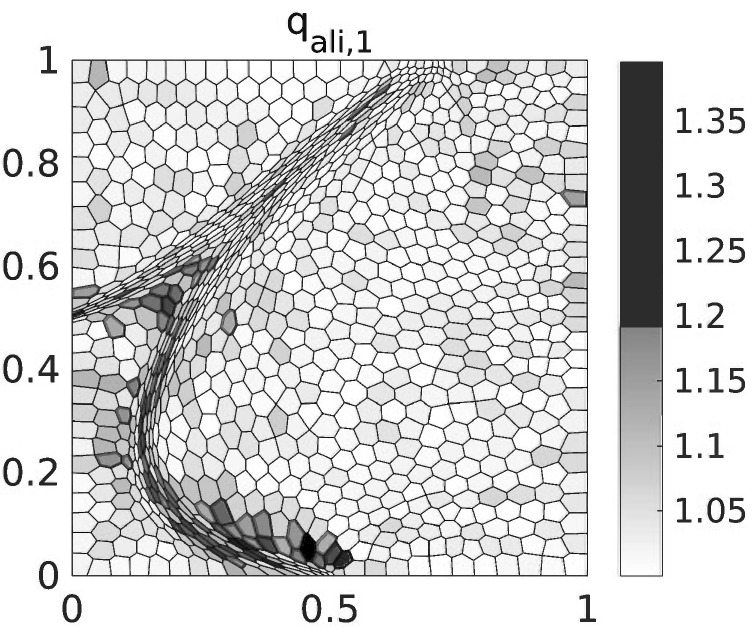}\quad \quad
\includegraphics[width=4cm]{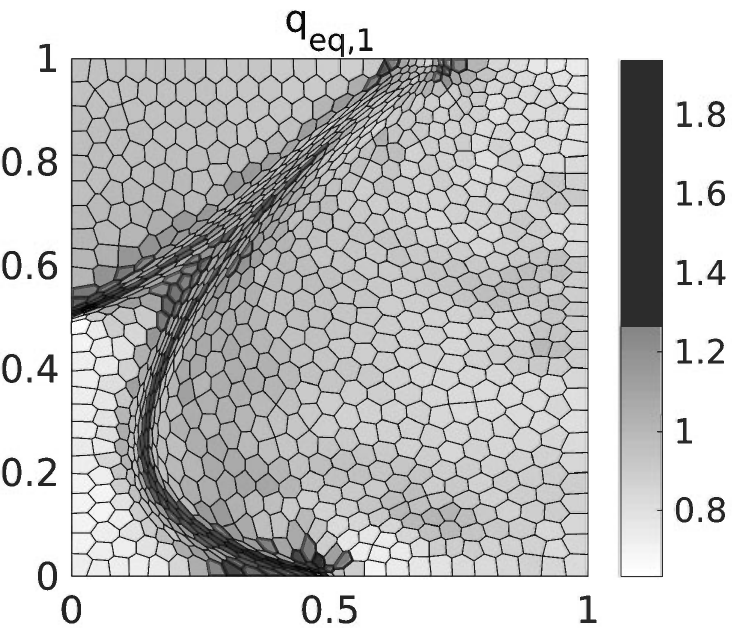}
\caption{Example 4 with $\varepsilon = 10^{-3}$ on $32\times 32$ mesh.
    The distributions of $q_{ali,1}$ and $q_{eq,1}$ after 10 MMPDE outer iterations.}
\label{fig:ex4-2}
\end{center}
\end{figure}

The diffusion coefficient $\varepsilon$ in Example 4 has a damping effect on the solution, especially in the area where two internal layers meet.
In  Fig.~\ref{fig:ex4-3}, we present solutions for $\varepsilon = 10^{-3}$, $10^{-4}$ and $5\times 10^{-5}$.
From the graphs one can see that, when $\varepsilon$ gets smaller,
the upper internal layer extends further towards the boundary $y=1$ and thus making the problem more difficult to solve.
We find it necessary to use finer meshes in order to get satisfying solution for smaller $\varepsilon$.
In the numerical experiments we have used $48\times 48$ mesh for $\varepsilon = 10^{-4}$
and $56\times 56$ mesh for $\varepsilon =5\times 10^{-5}$, both yield satisfying results after 10 MMPDE outer iterations.
For $\varepsilon = 10^{-4}$ and $\varepsilon =5\times 10^{-5}$, there are oscillations around the point where two internal layers meet
and around the point $(0.5,0)$, but not at the point $(0,0.5)$.
We suspect that the oscillation around where the two internal layers meet is due to the fact that
the meeting points gets closer to the Neumann boundary $x=1$ when $\varepsilon$ gets smaller,
which creates complications.
At the point $(0.5,0)$, \revA{the way the finite element method handles} the jump of $g$ is equivalent
to setting the width of the internal layer at $(0.5,0)$ to be $O(h)$.
Hence as long as $h$ is far greater than $\varepsilon$, i.e., the damping effect of $\varepsilon$ is limited
within the internal layer, such an oscillation cannot be eliminated by simply reducing $h$.
Again, by comparing the numerical solution at the points $(0.5,0)$ and $(0,0.5)$, we confirm our previous suspicion
that such oscillation is related to the the incidental angle of the internal layer with $\partial\Omega$.
\revA{In Figs.~\ref{fig:ex4-3}-\ref{fig:ex4-4}}, the final mesh as well as the distributions of $q_{ali,1}$ and $q_{eq,1}$ for $\varepsilon = 5\times 10^{-5}$ on $56\times 56$ mesh
are presented.
From the graphs, one sees that the MMPDE method correctly catches both internal layers, and more importantly,
successfully takes care of the tangential meeting of internal layers as well as the ``discontinuous'' Dirichlet boundary data.

\begin{figure}[htb]
  \begin{center}
    \includegraphics[width=4cm]{Ex4-1e3-32-mesh10-sola}\;
    \includegraphics[width=4cm]{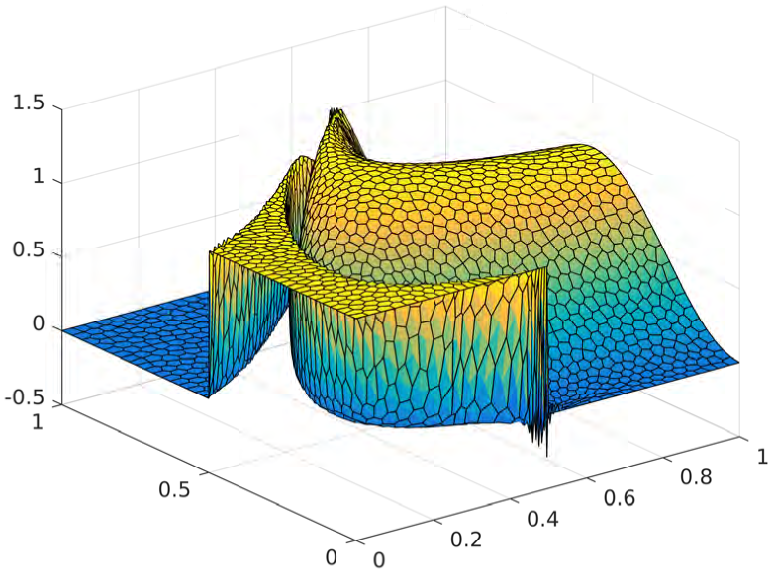}\;
    \includegraphics[width=4cm]{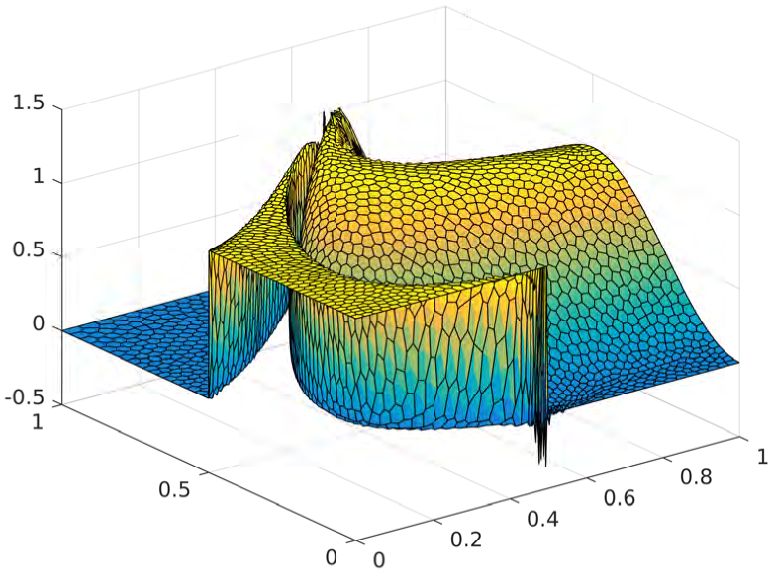} \\   
    \quad\includegraphics[width=4cm]{Ex4-1e3-32-mesh10-solb}\;
    \includegraphics[width=4cm]{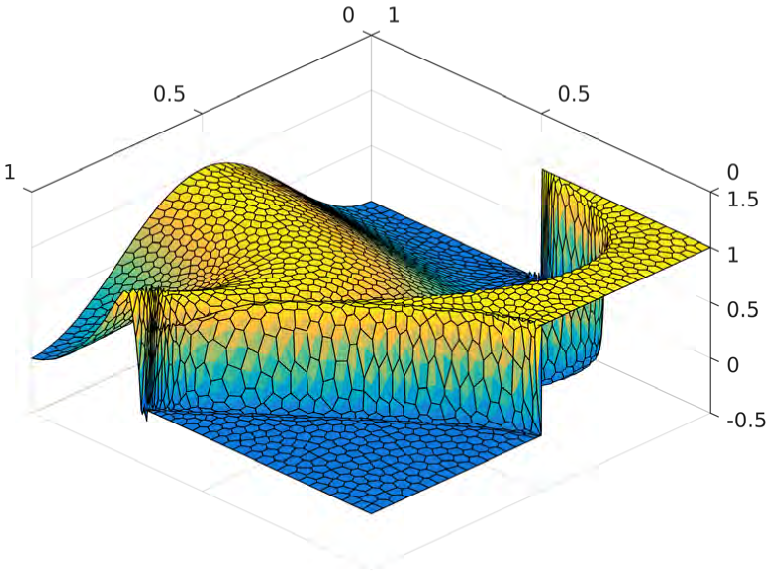}\;
    \includegraphics[width=4cm]{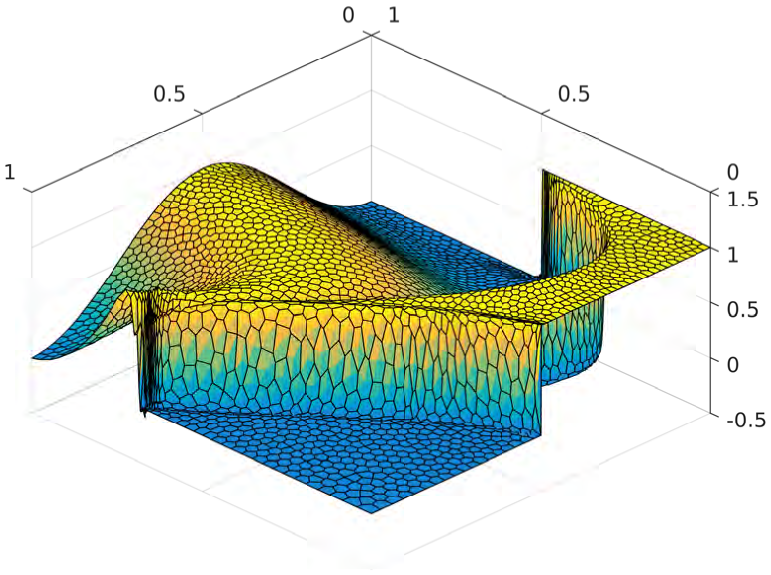}
    \caption{Example 4. From left to right: solutions for $\varepsilon = 10^{-3}$ on $32\times 32$ mesh,
        for $\varepsilon = 10^{-4}$ on  $48\times 48$ mesh and for $\varepsilon =5\times 10^{-5}$ on $56\times 56$ mesh,
    each taking 10 MMPDE outer iterations. The first row contains the front view and the second row contains the back view.}
\label{fig:ex4-3}
\end{center}
\end{figure}

\begin{figure}[htb]
  \begin{center}
    \includegraphics[width=3cm]{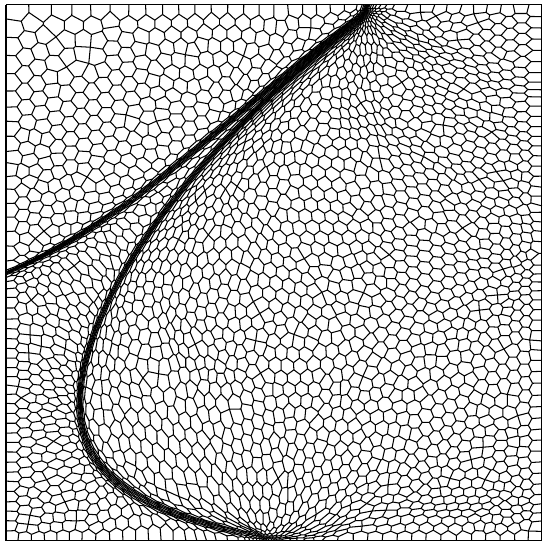}
    \includegraphics[width=4cm]{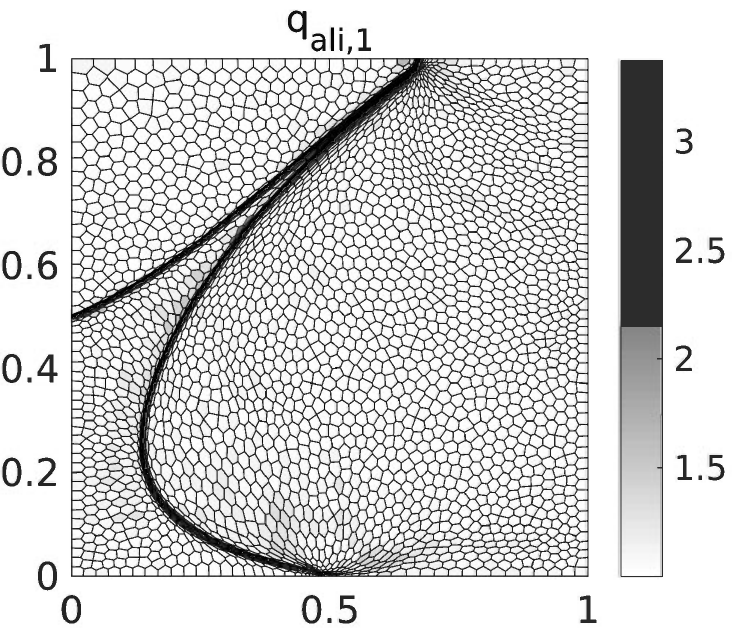}
    \includegraphics[width=4cm]{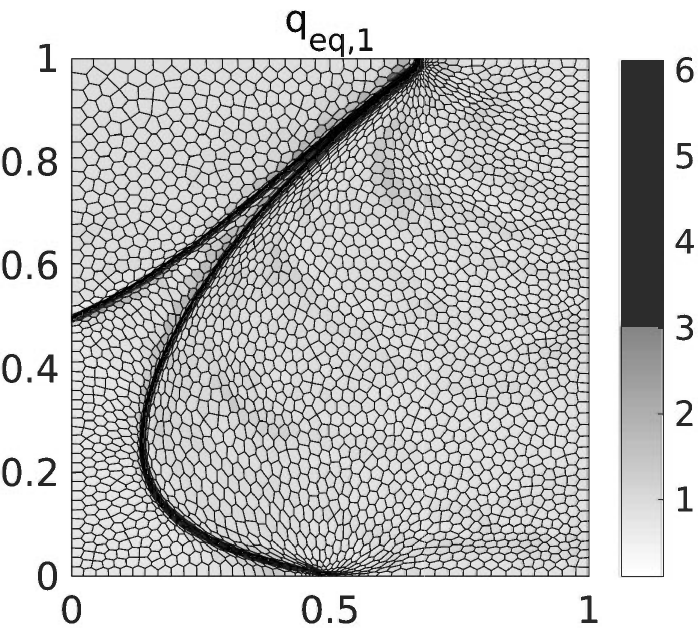}
    \caption{Example 4 with $\varepsilon = 5\times 10^{-5}$ on $56\times 56$ mesh,
    the mesh and the distributions of $q_{ali,1}$ and $q_{eq,1}$ after 10 MMPDE outer iterations.}
\label{fig:ex4-4}
\end{center}
\end{figure}

\revDD{
We also test a modification of Example 4, with the Neumann boundary condition $\frac{\partial u}{\partial n} = 0$ on the outflow boundary $y=1$ being replaced by
a homogeneous Dirichlet boundary condition $u=0$. This is known to create a boundary layer near $y=1$.
In Fig.~\ref{fig:ex4-alt}, the mesh after 10 MMPDE outer iterations is plotted, together with the front and back views of the numerical solution.
The results suggest that the MMPDE algorithm handles the boundary layer well for convection dominated convection-diffusion problems while correctly capturing the intersection of an internal layer and a boundary layer.

\begin{figure}[htb]
  \begin{center}
\includegraphics[width=2.8cm]{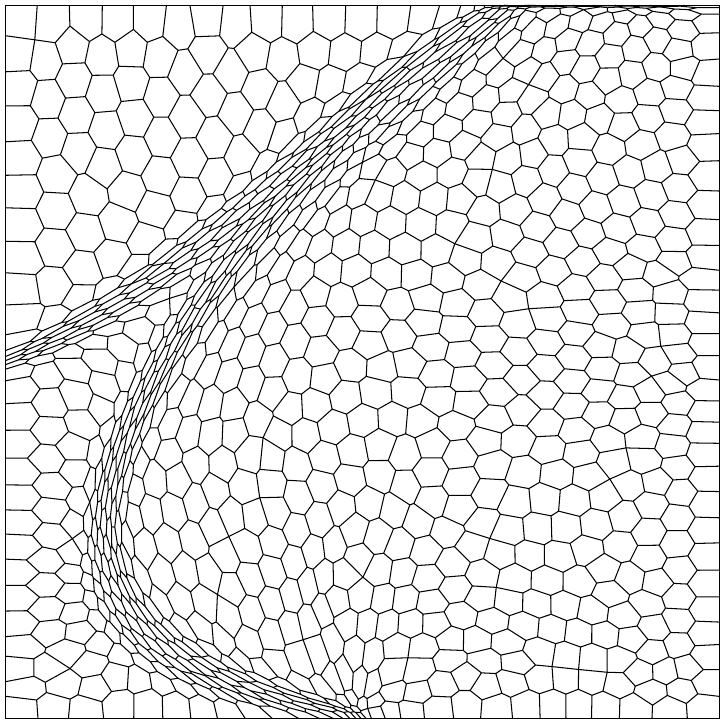}    
\includegraphics[width=3cm]{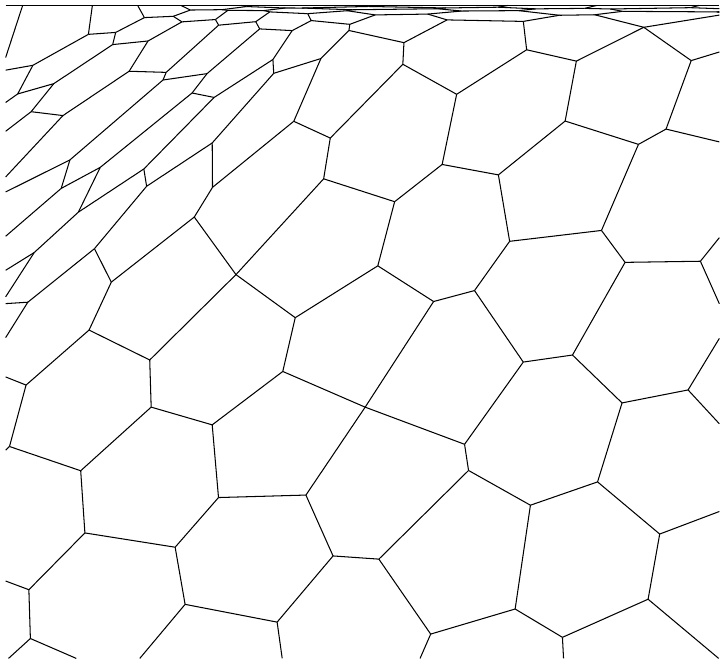}
\includegraphics[width=4cm]{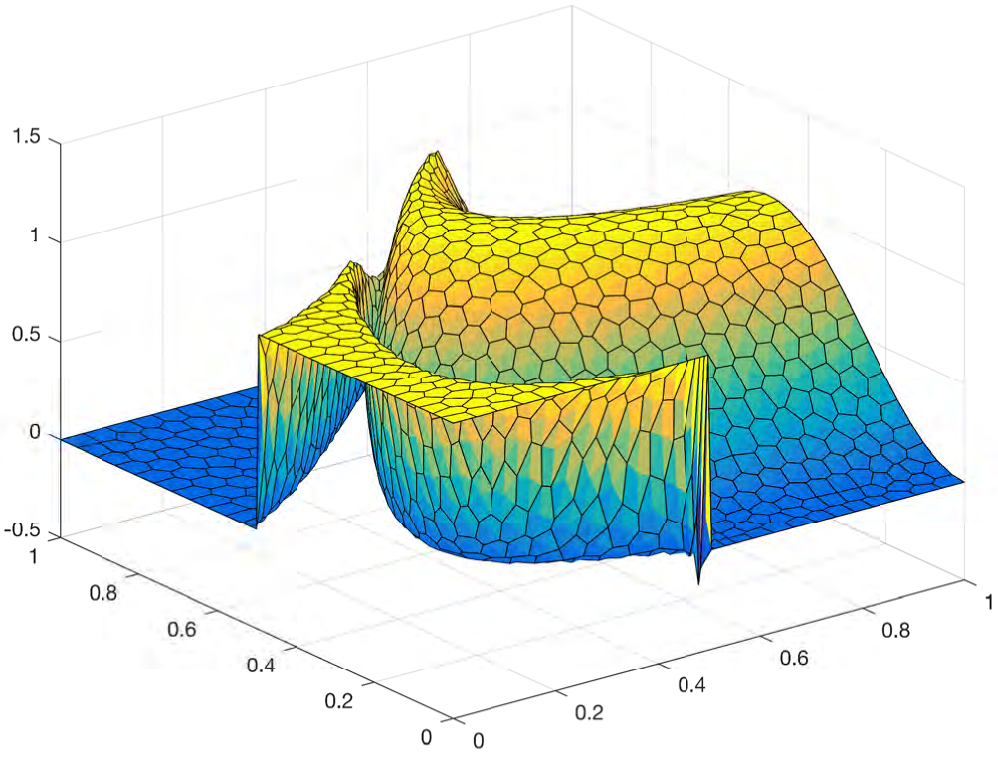}
\includegraphics[width=4cm]{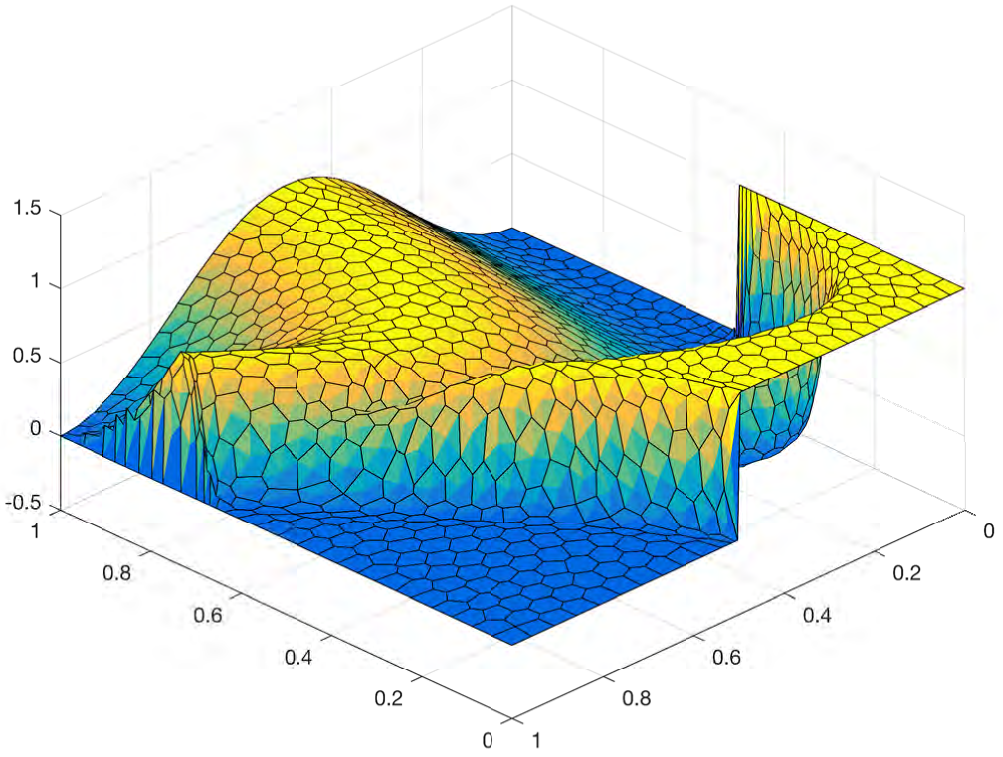}
\caption{Modified Example 4 (pure Dirichlet boundary condition) with $\varepsilon = 10^{-3}$ on $32\times 32$ mesh,
the mesh, a zoom-in to the upper-right region of the mesh, the front and the back view of the numerical solution after 10 MMPDE outer iterations.}
\label{fig:ex4-alt}
\end{center}
\end{figure}
}

Finally, Example 5 is a well-known singularly perturbed problem with a thin boundary layer around $\partial\Omega$.
When solved on a uniform mesh with characteristic size $h$ greater than the thickness of the boundary layer,
the numerical solution will be polluted by heavy oscillations, as shown in the leftmost graph of Fig.~\ref{fig:ex5-1}.
\revZ{We test the MMPDE method for this problem, using} an initial $32\times 32$ CVT mesh.
\revZ{The mesh and numerical solution} after 20 MMPDE outer iterations are given in Fig.~\ref{fig:ex5-1},
and the distributions of $q_{ali,1}$ and $q_{eq,1}$ are given in Fig.~\ref{fig:ex5-2}.
One can see that the MMPDE method accurately catches the strongly anisotropic thin boundary layer
\revZ{and} gives a correct solution without any oscillation.

\begin{figure}[htb]
  \begin{center}
\includegraphics[width=4cm]{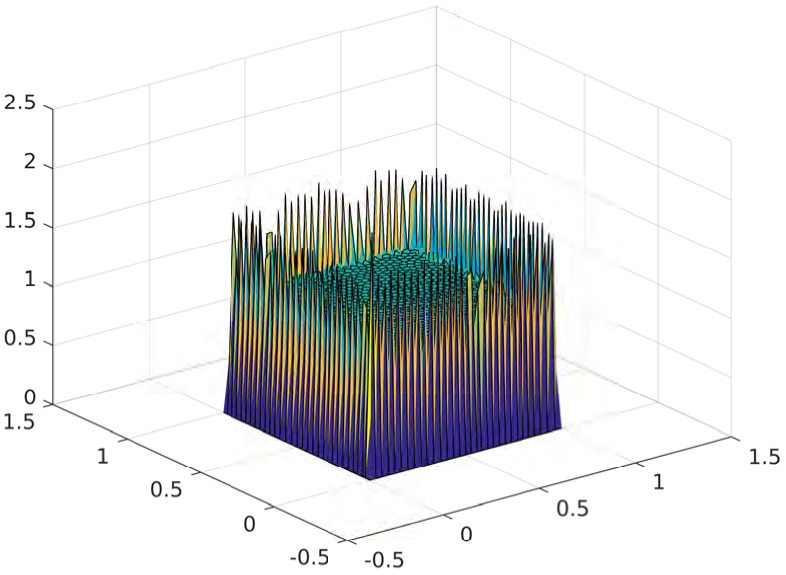}    
\includegraphics[width=3cm]{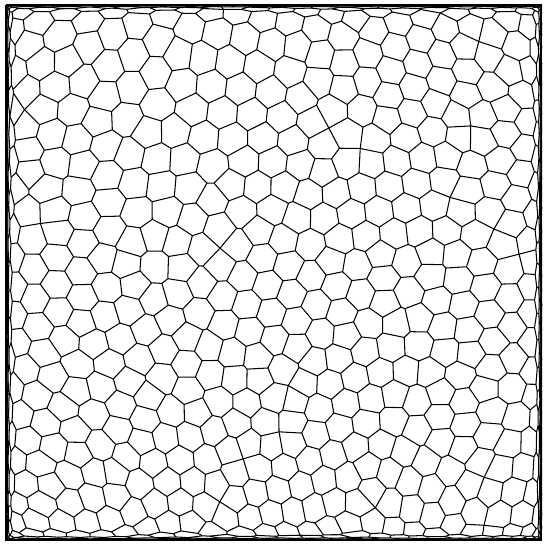}
\includegraphics[width=3cm]{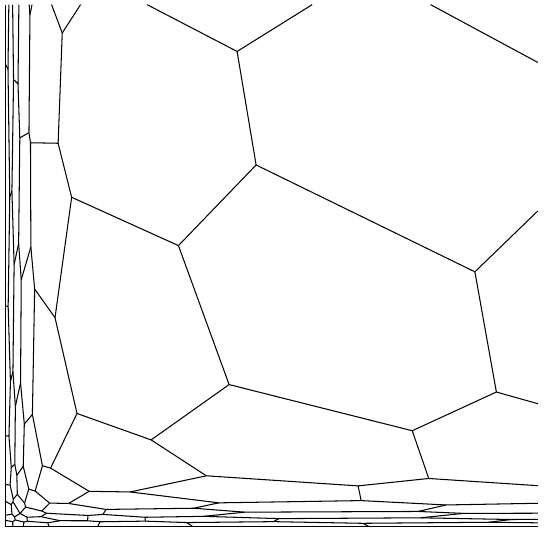}
\includegraphics[width=4cm]{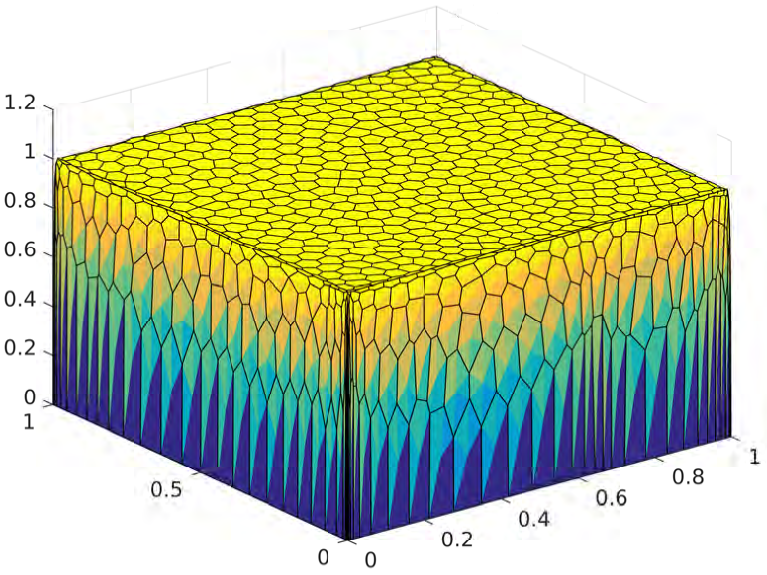}
\caption{Example 5 tested on $32\times 32$ mesh.
    The leftmost graph is the numerical solution on the initial mesh. 
    The rest of the graphs are the mesh,  its detail at the low-left corner, and the numerical solution after 20 MMPDE outer iterations.}
\label{fig:ex5-1}
\end{center}
\end{figure}

\begin{figure}[htb]
\begin{center}
\includegraphics[width=4.5cm]{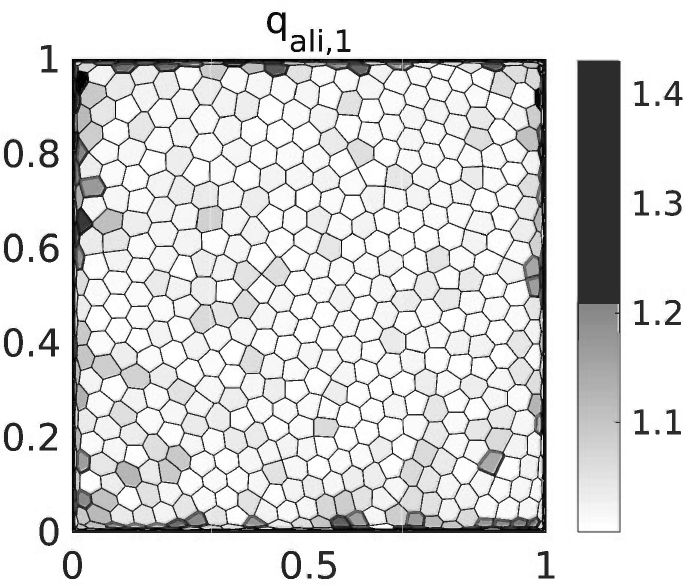}
\includegraphics[width=4.5cm]{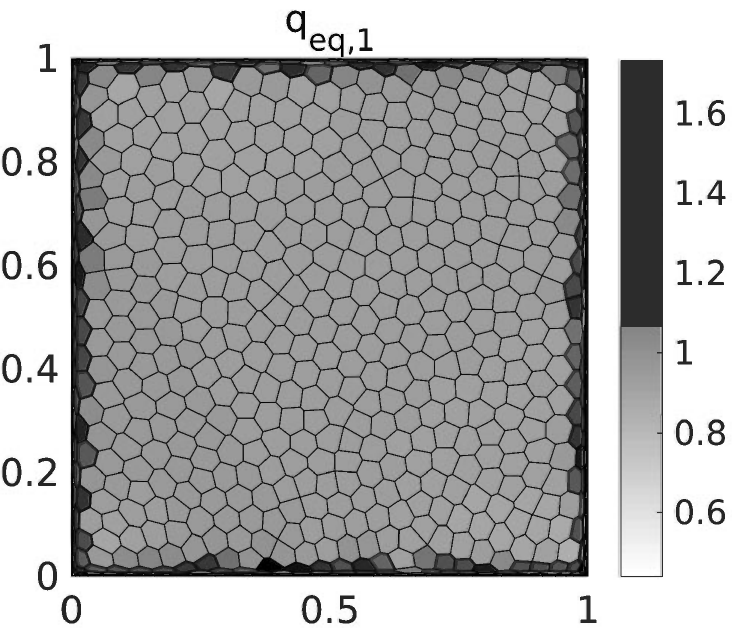}
\caption{Example 5 tested on $32\times 32$ mesh. The distributions of $q_{ali,1}$ and $q_{eq,1}$ after 20 MMPDE outer iterations.}
\label{fig:ex5-2}
\end{center}
\end{figure}

\section{Conclusions}

In the previous sections we have studied anisotropic mesh quality measures and adaptation
for polygonal meshes. Three sets of alignment (for shape) and equidistribution (for size)
quality measures have been developed. They can be viewed as numerical discretizations or approximations
of continuous measures (\ref{eq:contGlobalMeasurements}). Their major features are summarized
in subsection~\ref{sec:summary}. Numerical examples show that they all can give good \revZ{indications} for
the quality of polygonal meshes.

Among the three sets of measures, $Q_{ali,2}$ (\ref{ali-2}) and $Q_{eq,2}$ (\ref{eq-2}) provide
the toughest measurement. One of the special cases for defining $Q_{ali,2}$ and $Q_{eq,2}$ is
to use piecewise linear generalized barycentric mappings which can be based on triangulation
of each polygonal cell into triangles. The collection of those triangles forms a triangular mesh,
through which many adaptive mesh algorithms designed for triangular meshes can be adopted
for polygonal mesh adaptation. Though special care needs to be taken since a good polygonal
cell can have bad triangles and algorithms should focus on improving the quality of polygonal cells
instead of triangles.

Along this line, an anisotropic adaptive polygonal mesh algorithm based on the MMPDE
moving mesh method has been devised. For a given reference computational (polygonal)
mesh, it moves the mesh vertices such that the adaptive polygonal mesh has
a good quality with reference to the reference one. In the numerical examples, the reference computational
mesh has been taken as a CVT generated with Lloyd's algorithm. Numerical results confirm that
the proposed anisotropic adaptive polygonal mesh algorithm is able to produce desired mesh concentration
and lead to a more accurate solution than using a uniform polygonal mesh.

Finally, we would like to point out that the main idea in this work can be generalized
to three-dimensions to define anisotropic polyhedral mesh quality measures
and generate anisotropic adaptive polyhedral meshes.
However,  the technical details of the implementation and
the effectiveness and robustness of the method in three-dimensions
remain to be examined in future work.


\vspace{20pt}

\revi{
\noindent
{\bf Acknowledgements}

W. Huang was supported in part by the NSF under Grant DMS-1115118 and
the University of Kansas GRF Award \#2301056.
Y. Wang was supported in part by the National Natural Science Foundation of China
under grant numbers 11671210, 91630201,
and the FY2015 Spring Travel Program from the College of Art \& Science,
Oklahoma State University, during her visit to the University of Kansas in June 2015.
The authors are thankful to the anonymous referees for their valuable comments
in improving the quality of the paper.
}


\end{document}